\documentclass{article}
\usepackage{amsthm,amsmath,amssymb}
\usepackage[margin=.75in,bottom=1in]{geometry} 
\usepackage{caption}
\usepackage{subcaption}
\usepackage{graphicx}
\usepackage{epstopdf}
\usepackage{color}
\usepackage{framed}
\usepackage{lscape}
\usepackage{rotating}
\usepackage{algorithm}
\usepackage[noend]{algpseudocode}
\usepackage{grffile}
\usepackage{multirow}
\usepackage{xcolor}
\usepackage{comment}
\usepackage{xr}
\usepackage{esint}
\usepackage{mathtools}
\usepackage{cancel}
\usepackage{url}

\usepackage{tikz}
\usetikzlibrary{math}

\newtheorem{thm}{Theorem}[section]

\newtheorem{lem}[thm]{Lemma}

\newtheorem{cor}[thm]{Corollary}
\newtheorem{prop}[thm]{Proposition}

\theoremstyle{definition}
\newtheorem{rmk}{Remark}

\newcommand{\SC}{\mathrm{SC}}
\newcommand{\SW}{\mathrm{SW}}
\newcommand{\wass}{\mathrm W}
\newcommand{\monge}{\mathrm M}
\newcommand{\cdist}{\mathrm{C}}

\newcommand{\mi}{\mathrm{i}}

\newcommand{\EE}{\mathbb{E}}
\newcommand{\RR}{\mathbb{R}}
\newcommand{\BB}{\mathbb{B}}
\newcommand{\DD}{\mathbb{D}}

\newcommand{\Prob}{\mathbb{P}}

\newcommand{\what}{\widehat}
\newcommand{\wtilde}{\widetilde}

\newcommand{\calA}{\mathcal{A}}
\newcommand{\calF}{\mathcal{F}}
\newcommand{\calR}{\mathcal{R}}
\newcommand{\calP}{\mathcal{P}}
\newcommand{\calQ}{\mathcal{Q}}
\newcommand{\calU}{\mathcal{U}}
\newcommand{\calM}{\mathcal{M}}
\newcommand{\calT}{\mathcal{T}}
\newcommand{\calV}{\mathcal{V}}

\newcommand{\sign}{\mathrm{sign}}

\begin{document}

\title{
\textbf{On sliced Cram\'er metrics}
}
\author{\textbf{William Leeb} \\ \\
School of Mathematics \\
University of Minnesota, Twin Cities  \\
Minneapolis, MN}
\date{}
\maketitle

\begin{abstract}
This paper studies the family of sliced Cram\'er metrics,
quantifying their stability under
distortions of the input functions.
Our results bound the growth of
the sliced Cram\'er distance between
a function and its geometric deformation by
the product of the deformation's displacement size
and the function's mean mixed norm.
These results extend to sliced Cram\'er distances between tomographic projections.
In addition, we remark on the effect of convolution
on the sliced Cram\'er metrics.
We also analyze efficient
Fourier-based discretizations in 1D and 2D,
and prove that they are robust to heteroscedastic noise.
The results are illustrated by numerical experiments.

\end{abstract}

\maketitle

\section{Introduction}

This paper is concerned with properties of the
sliced $p$-Cram\'er metrics, a family of distances
that extend the classical univariate $p$-Cram\'er metrics
to functions in $\RR^d$ via the ``slicing'' operation.
The Cram\'er and sliced Cram\'er metrics have arisen in a variety
of different contexts.
These include statistical testing
\cite{ramdas2017wasserstein, szekely2013energy, rizzo2016energy,
szekely2023energy, cramer1928composition}
and probability theory \cite{zolotarev1984probability},
as well as
machine learning applications,
where they have been proposed as
an alternative to Wasserstein and sliced Wasserstein distances
\cite{bellemare2017cramer, lheritier2022cramer, kolouri2020sliced}.
The sliced $2$-Cram\'er distance is also known as the energy distance
\cite{szekely2013energy, rizzo2016energy, szekely2023energy, peyre2019computational}.
Well-known special cases in one variable include the $1$-Cram\'er distance,
which coincides with the $1$-Wasserstein distance, and the $\infty$-Cram\'er
distance, which coincides with the Kolmogorov metric.

A fundamental question about any metric is
how stable it is under
distortions of its inputs,
such as geometric deformations or noise.
It appears that, in spite of the long-standing
interest in Cram\'er and sliced Cram\'er metrics,
there has been little work on
characterizing their stability.
In this paper, we address this topic
by describing stability properties of the sliced Cram\'er distances
in settings of practical interest. 
We provide bounds on sliced Cram\'er distances' growth
under geometric deformations of the inputs,
and compare and contrast these bounds
with those of Wasserstein and Lebesgue distances.
We examine sliced Cram\'er metrics' behavior under convolutions,
a setting which occurs frequently in signal and image processing applications.
We also bound the sampling error of Fourier-based discretizations in 1D and 2D,
and prove their robustness to heteroscedastic sub-Gaussian noise.

\subsection{Geometric deformations}
\label{sec:intro_deformations}

A natural question to ask about a metric
is how stable it is to deformations of its inputs.
For example, in image processing, one
may seek a metric that is insensitive to small
translations, rotations, changes of scale,
or perturbations in the parameters that generate the images
being compared.
To make this question more precise,
one must specify a class of deformations
and a reasonable measure of the ``size''
of a deformation.
There are many classes of deformations
of practical interest, and 
in this paper we consider the family
of push-forward deformations of a function: if
$\Phi$ is a $C^1$ bijection mapping into the domain of
a function $f$, it induces the deformation
$f_\Phi(x) = f(\Phi(x)) |\nabla \Phi(x) |$ of $f$.
Such deformations preserve $f$'s integral and $L^1$ norm.
Other works have considered metric stability under
different classes of deformations,
such as those that preserve
$f$'s $L^\infty$ norm \cite{mallat2012group, anden2014deep};
however, integral-preserving deformations
are more natural
for understanding sliced Cram\'er metrics, which are typically used to compare
probability measures.

Given this class of deformations, a natural definition
of a deformation $\Phi$'s size
is its maximum displacement,
$\max_{x} |x - \Phi(x)|$.
In this paper, we will consider a family
of measures of deformation size,
defined in Section \ref{sec:pushforwards},
which includes the maximum displacement as a special case.
Proposition \ref{prop:deformations}
in Section \ref{sec:pushforwards} relates these different measures of deformation size.

Having defined an appropriate measure of the size $\varepsilon(\Phi)$ of a deformation,
we can quantify a metric $\mathrm{D}$'s robustness to deformations
by the rate of growth of $\mathrm{D}(f,f_\Phi)$ as a function of $\varepsilon(\Phi)$.
Because of the triangle inequality, this growth rate
controls how ``smoothly'' the metric changes
under deformations.
More precisely,
if $\mathrm{D}$  satisfies a bound of the form
$\mathrm{D}(f,f_\Phi) \le C(f) \varepsilon(\Phi)$, then the triangle inequality implies
that for any functions $f$ and $g$,
\begin{align}
|\mathrm{D}(f,g) - \mathrm{D}(f_\Phi,g)| \le \mathrm{D}(f,f_\Phi) \le C(f) \varepsilon(\Phi).
\end{align}
This bound can be interpreted
(see Remarks \ref{rmk:deformation_distance} and \ref{rmk:smoothness})
as saying that
$\mathrm{D}(f_\Phi,g)$
is a ``smooth'' function of the deformation $\Phi$,
when $\Phi$ is close to the identity transformation.

Prior work on metrics' behavior under
deformations
have focused primarily on
Wasserstein-type distances \cite{leeb2016holder,rao2020wasserstein,shi2025fast}.
Specifically, \cite{rao2020wasserstein} and \cite{shi2025fast} show that
they are robust
to rigid deformations (translations and rotations)
when comparing 2D tomographic
projections of a 3D volume; these are useful
properties for comparing images in single particle cryogenic
electron microscopy (cryo-EM) \cite{singer2020computational, bendory2020single,
doerr2016single},
as well as other
scientific problems for which
the measurement modality only permits
observing projections of an object
\cite{deans2007radon, natterer1986mathematics,
helgason2010integral, herman2009fundamentals, singer2013two, coifman2008graph}.
In Section \ref{section:preliminaries}
we state quite general robustness
properties of Wasserstein and sliced Wasserstein distances,
which, though we have not seen them published
previously, follow straightforwardly from
prior work.

In this paper, we bound the
growth rate of the sliced Cram\'er metrics under general deformations,
and extend these bounds to the distances between tomographic projections.
The growth rates depend both on the aforementioned measures of
displacement size and on the function's
mean mixed norm,
which we define in Section \ref{sec:mean_mixed_norms}.
We also compare and contrast the bounds we prove
for sliced Cram\'er metrics to the behavior
of Wasserstein-type distances and Lebesgue distances.

\subsection{Convolutions}

In typical scientific applications,
each observation will be a convolution
of the underlying physical object with a filter 
that arises from the measurement apparatus.
It is therefore of interest to understand how a metric
changes when the inputs are each convolved
by a common function. Since convolution is a smoothing operation,
it is natural to  expect that the convolved functions will be closer
to each other (up to a scaling factor) than the clean functions. Indeed,
this is the case for $L^p$ distances (due to Young's
convolution inequality), and is also true for
Wasserstein and sliced Wasserstein distances
(see Theorem \ref{thm:wasserstein_convo}
and Corollary \ref{cor:sw_convo} below).
In this paper, we observe that the same
is also true for sliced Cram\'er distances. This fact
is essentially a rephrasing of a result from
the prior work \cite{zhang2023cramer},
where it is stated in terms of random variables
and specific to the case $p=2$.
Our bound is stated for all $p \ge 1$,
and is sharper when the convolution filter
has both positive and negative values,
which can occur in certain scientific imaging
applications like cryo-electron microscopy,
among others.

\subsection{Discretizations and additive noise}

Of course, for a metric to be of practical interest there
must exist computationally efficient algorithms
for computing it based on finite data.
In this work we study the convergence of Fourier-based discretizations
of the 1D Cram\'er metric and 2D sliced Cram\'er metric;
the latter is
based on the method
described in \cite{shi2025fast},
and may be computed rapidly via the non-uniform fast Fourier transform (NUFFT).
These approximations take as inputs samples of the functions
on an equispaced grid, as is standard in signal processing applications.
We provide bounds on the discretization error
when approximating certain classes of functions.

In addition to bounding
the approximation errors on noiseless samples,
we will also show that these discretizations are robust to additive,
heteroscedastic, sub-Gaussian noise:
that is, when applied to samples from a
signal-plus-noise model, the estimated distance converges to the
distance between the signals only, and filters out the noise.
This is a useful property for a metric to have, though
one that is by no means unique to the sliced Cram\'er metrics (there are many
methods for filtering out noise).
By contrast, many Wasserstein-type distances
are not a priori defined (without modification)
on noisy data, as they typically require their inputs
to be probability measures;
and Lebesgue distances do not possess such a robustness property,
and can be severely distorted by additive noise.

\subsection{Notation and conventions}

Throughout the paper, we will assume
familiarity with basic concepts of measure and integration,
e.g.\ at the level of \cite{folland1999real}
or \cite{heil2019introduction}, and
probability theory at the level of \cite{vershynin2025high}.
In this section we briefly review several definitions
and introduce the notational conventions that we will use.

\subsubsection{Lebesgue norms}

Throughout, we denote by $\|f\|_{L^p} = \left( \int_{\RR^d} |f(x)|^p dx\right)^{1/p}$
the Lebesgue $p$-norm of a function $f: \RR^d \to \RR$ (with the obvious modification
when $p=\infty$), and denote the standard inner product as
\begin{math}
\langle f, g\rangle = \int_{\RR^d} f(x) g(x) dx.
\end{math}
If $A \subset \RR^d$, we will say that $f$ is in $L^p(A)$
if $f$ is defined on $\RR^d$ with $\|f\|_{L^p} < \infty$ and $f(x) = 0$
for $x \notin A$.
We also define the $p$-norm $\|x\|_{p}$ for vectors $x$ in $\RR^d$
as 
\begin{math}
\|x\|_{p} = \left(\sum_{j=1}^{d} |x_j|^p \right)^{1/p}
\end{math}
(again, with the obvious modification
when $p=\infty$).
When convenient, we will also denote the $2$-norm of a vector $x$ in $\RR^d$
by $|x|$.
We also denote the inner product between two vectors $x$ and $y$ in $\RR^d$ by
\begin{math}
\langle x,y\rangle = \sum_{j=1}^d x_j y_j.
\end{math}

\subsubsection{Absolute continuity}

For a given interval
$(a,b) \subset \RR$,
we denote by $\calA_0 = \calA_0(a,b)$ the set of
absolutely continuous functions $G$ on $(a,b)$ satisfying
$G(b)=0$; note that such functions may be written in the form
\begin{math}
G(x) = -\int_{x}^{b} g(t) \, dt
\end{math}
where $g \in L^1$  and $g = G'$ almost everywhere.

\subsubsection{The Fourier transform}

We denote the Fourier transform
of a function $f : \RR^d \to \RR$ in $L^1$ by
\begin{math}
\what{f}(\xi) = \int_{\RR^d} f(x) e^{-2 \pi \mi \langle x, \xi \rangle} d x.
\end{math}
With this convention, the Fourier inversion formula may be written as
\begin{math}
f(x) = \int_{\RR^d} \what{f}(\xi) e^{2 \pi \mi \langle x, \xi \rangle} d \xi.
\end{math}

\subsubsection{$C^r$ functions}

Let $r > 0$, and write $r = m+s$ where $m$ is an integer and $0 \le s < 1$.
Let $f : \RR^d \to \RR$ be compactly supported.
We say that $f$ is $C^r$
if all partial derivatives of $f$ of degree up to and including $m$ exist
and are continuous, and, if $s > 0$, the partial derivatives
of degree $m$ are $s$-H\"older; that is, there is a $C > 0$
such that for all $x \in \RR^d$ and $i=1,\dots,d$,
\begin{align}
\sup_{h \ne 0} \frac{\partial^m f(x + h e_i) - \partial^m f(x)}{|h|^s} \le C,
\end{align}
where $\partial^m f$ denotes any partial derivative of degree $m$.
As is well-known, for such $f$, $|\what{f}(\xi)| = O(|\xi|^{-r})$.

\subsubsection{Tomographic projections and the Radon transform}

Let
$\calU = \left(u^{(1)},\dots,u^{(r)}\right) \in \mathbb{S}^{d-1} \times \cdots \times \mathbb{S}^{d-1}$
(where $\mathbb{S}^{d-1} \subset \RR^d$
is the $(d-1)$-dimensional unit sphere)
denote an ordered collection of $r$ orthonormal vectors in $\RR^d$.
Let $u^{(r+1)},\dots,u^{(d)}$ denote $d-r$ orthonormal vectors that are
orthogonal to $u_1,\dots,u_r$.
We define the tomographic projection $\calP_\calU$
onto $\mathrm{span}\{u^{(1)},\dots,u^{(r)}\}$ by
\begin{align}
\label{eq:tomographic_projection}
(\calP_\calU f)(t_1,\dots,t_r)
    = \int_{\RR^{d-r}} f(t_1 u^{(1)} + \dots + t_r u^{(r)}
        + s_1 u^{(r+1)} + \dots + s_{d-r} u^{(d)}) \, ds.
\end{align}
When $r=1$, we denote the tomographic projection of $f$ onto
the span of a unit vector $u$ by $\calP_u f$.
Note that in this case, the \emph{Radon transform}
$\calR f : \RR \times \mathbb{S}^{d-1} \to \RR$ of the function $f$
is defined by $(\calR f)(t,u) = (\calP_u f)(t)$.
For more background on these transforms,
see, for example, the references \cite{natterer1986mathematics, helgason2010integral}.
A standard result that we will use is the \emph{Fourier slice theorem}:
\begin{math}
\what{(\calP_u f)}(\xi) = \what{f}(\xi u),
\end{math}
for any unit vector $u \in \RR^d$ and real number $\xi$.

\subsubsection{Push-forwards}

If $\Omega \subset \RR^d$ is
a (non-empty) open set,
$\mu$ is a finite, signed measure on $\Omega$,
and $\Psi : \Omega \to \RR^d$ is a measurable function,
we denote by $\Psi_\sharp \mu$ the push-forward measure,
\begin{math}
(\Psi_\sharp \mu)(E) = \mu(\Psi^{-1}(E));
\end{math}
e.g.\ see \cite{peyre2019computational}.
Note that $(\Psi_\sharp \mu)(\Psi(\Omega)) = \mu(\Psi^{-1}(\Psi(\Omega))) = \mu(\Omega)$.
When $\mu$ is induced from a function $f$ supported on $\Omega$, i.e.\
$\mu(E) = \int_{E} f(x) dx$, and $\Psi$ is a diffeomorphism between $\Omega$ and $\Psi(\Omega)$
with inverse $\Phi = \Psi^{-1}$,
then $\Psi_\sharp \mu$ has density
$f(\Phi(x))|\det(\nabla \Phi (x))|$,
where $\nabla \Phi (x)$ denotes the Jacobian
matrix of $\Phi$ at $x$.
We will write $(\Psi_\sharp f)(x) = (\Phi^{-1}_\sharp f)(x)= f(\Phi(x))|\det(\nabla \Phi(x))|$,
or $f_\Phi(x) = f(\Phi(x))|\det(\nabla \Phi(x))|$ for short.

\subsubsection{Sub-Gaussian random variables}

If $Z$ is a sub-Gaussian random variable, we denote by
$\|Z\|_{\psi_2}$ its sub-Gaussian norm, defined by
\begin{align}
\label{eq:subgaussian_norm}
\|Z\|_{\psi_2} = \inf\{\sigma > 0 \,: \, \EE[\exp\{Z^2 / \sigma^2\}] \le 2\}.
\end{align}

\subsection{Outline of the remainder of the paper}

The remainder of the paper is structured as follows:

\begin{enumerate}

\item
Section \ref{section:preliminaries} contains key definitions,
background material, and several preliminary results.
Theorem \ref{thm:wasserstein_projections}
and Theorem \ref{thm:sw_deformations} establish
general robustness properties of
Wasserstein and sliced Wasserstein distances,
respectively; these results appear to be new,
though they follow easily from prior literature.

\item
Section \ref{section:properties} states the main theorems on the robustness
of sliced Cram\'er distances to geometric deformations.
Theorem \ref{thm:main_deformations} and Corollary \ref{cor:main_projections}
give bounds under quite general conditions. Sharper bounds are proved
for several special cases of interest in Section \ref{section:sharper},
including rotations, translations, and dilations.
In addition, Theorem \ref{thm:convolutions}
describes the behavior of sliced Cram\'er metrics under convolutions.

\item
Section \ref{section:discrete} analyzes Fourier-based approximations
to the  Cram\'er distances and the 2D  sliced  Cram\'er distances,
the latter with respect to the uniform measure over $\mathbb{S}^1$.
Theorems \ref{thm:discrete_main1d} and \ref{thm:discrete_main2d}
show that these discretizations are robust to additive, heteroscedastic,
sub-Gaussian noise.

\item
Section \ref{section:numerical} shows the results of several numerical
experiments illustrating the theoretical results,
including comparisons between the sliced Cram\'er,
sliced Wasserstein, and Lebesgue distances.

\item
Section \ref{section:proof_details}
contains results on Fourier analysis, approximation theory,
and concentration of measure that are used in the proofs
from Sections \ref{section:properties}
and \ref{section:discrete}.

\item
Section \ref{section:conclusion} concludes the paper, providing
a summary and topics for future research.

\end{enumerate}

\section{Background and preliminary results}
\label{section:preliminaries}

This section introduces key definitions
and background results
that will be referred to in the rest of the paper.
We draw particular attention to Theorem \ref{thm:wasserstein_projections}
and Theorem \ref{thm:sw_deformations} on
the growth of Wasserstein and sliced Wasserstein distances
under deformations, which we have
not seen published in
the literature (but which follow straightforwardly from previously published results).
These results will provide a useful point of reference
for our findings on sliced Cram\'er metrics.
Key definitions that are used throughout the paper,
along with their locations in the text,
may be found in Table \ref{table:definitions}.

\begin{table}
\centering
\begin{tabular}{| c | c  | c |}
\hline  
 Symbol &  Description & Textual reference  \\
\hline
$\calP_\calU f$ & Tomographic projection
        & \eqref{eq:tomographic_projection} \\
\hline
$\|Z\|_{\psi_2}$ & Sub-Gaussian norm
        & \eqref{eq:subgaussian_norm} \\
\hline
$\varepsilon_\infty(\Phi)$ & Maximum displacement
        & \eqref{eq:maximum_displacement} \\
\hline
$\varepsilon_{\eta,p}(\Phi)$ & Mean displacement
        & \eqref{eq:mean_displacement} \\
\hline
$\|f\|_{M_\eta^{p,r}}$ & Mean mixed norm
        & \eqref{eq:mean_mixed_norm} \\
\hline
$\wass_p(f,g)$ & Wasserstein distance
        & \eqref{eq:wasserstein} \\
\hline
$\calV f$ & Volterra operator
        & \eqref{eq:volterra_operator} \\
\hline
$\mathrm{SW}_{\eta,p}(f,g)$ & Sliced Wasserstein distance
        & \eqref{eq:sliced_wasserstein} \\
\hline
$\|f\|_{V^p}$ & Volterra norm
        & \eqref{eq:volterra_norm} \\
\hline
$\cdist_p(f,g)$ & Cram\'er distance
        & \eqref{eq:cramer_distance} \\
\hline
$\SC_{\eta,p}(f,g)$ & Sliced Cram\'er distance
        & \eqref{eq:sliced_cramer} \\
\hline
$\|f\|_{SV_\eta^{p}}$ & Sliced Volterra norm
        & \eqref{eq:sliced_volterra} \\
\hline
\end{tabular}
\caption{Key definitions and their introduction in the text.}
\label{table:definitions}
\end{table}

\subsection{Deformations and displacement}
\label{sec:pushforwards}

Given an open, non-empty $\Omega \subset \RR^d$
and a $C^1$, invertible $\Psi$ defined on $\Omega$ with
inverse $\Phi = \Psi^{-1}$,
we will call the push-forward $f_\Phi(x) = f(\Phi(x))|\det(\nabla \Phi(x))|$
a \emph{deformation of $f$},
and will also call the mappings $\Phi$ and $\Psi$ themselves deformations.

\subsubsection{Maximum displacement}

We define the \emph{maximum displacement} of the deformation $\Phi$
by
\begin{align}
\label{eq:maximum_displacement}
\varepsilon_\infty(\Phi) = \max_{x \in \Psi(\Omega)} |x - \Phi(x)|.
\end{align}
For example, if $u$ is a fixed unit vector, then the function
$\Phi(x) = x + \delta u$ has maximum displacement $\varepsilon_\infty(\Phi) = \delta$.

If $u$ is a unit vector and $\Phi$ is a $C^1$, $1$-to-$1$ mapping defined on $\Omega$,
we define the maximum displacement of $\Phi$ along $u$ to be
\begin{align}
\varepsilon(\Phi,u) = \max_{x \in \Omega} |\langle x - \Phi(x), u\rangle |.
\end{align}
Note that the maximum displacement can be written as
\begin{align}
\varepsilon_\infty(\Phi) = \max_{u \in \mathbb{S}^{d-1}} \varepsilon(\Phi,u).
\end{align}

\begin{rmk}
\label{rmk:deformation_distance}

If we define the distance
$\Delta$
between deformations on a domain $\Omega$
by
\begin{align}
\Delta(\Psi, \Phi) = \max_{y \in \Omega} | \Psi(y) - \Phi (y) |,
\end{align}
then one can write $\Delta(\Psi,\Phi)$
in terms of the maximum displacement $\varepsilon_\infty$:
\begin{align}
\Delta(\Phi,\Psi) = \varepsilon_{\infty}(\Phi \circ \Psi^{-1}).
\end{align}
Indeed,
\begin{align}
\varepsilon_\infty(\Phi \circ \Psi^{-1}) = \max_{x \in \Psi(\Omega)}|x - \Phi (\Psi^{-1} (x))|
= \max_{y \in \Omega} | \Psi(y) - \Phi (y) |.
\end{align}

\end{rmk}

\subsubsection{Mean displacement}

For a given probability distribution $\eta$ over $\mathbb{S}^{d-1}$,
and a value $1 \le p < \infty$,
we define the \emph{mean displacement} of $\Phi$ as
\begin{align}
\label{eq:mean_displacement}
\varepsilon_{\eta,p}(\Phi) = \left(\int_{\mathbb{S}^{d-1}} \varepsilon(\Phi,u)^p \, d \eta(u) \right)^{1/p}.
\end{align}

Note that if $\Psi = \Phi^{-1}$, then for all $u$,
\begin{align}
\varepsilon(\Phi,u) &= \max_{x \in \Omega} |\langle x - \Phi(x), u\rangle |
\nonumber \\
&= \max_{y \in \Phi(\Omega)} |\langle \Psi(y) - \Phi(\Psi(y)), u\rangle |
\nonumber \\
&= \max_{y \in \Phi(\Omega)} |\langle \Psi(y) - y, u\rangle |
\nonumber \\
&= \varepsilon(\Psi,u).
\end{align}
It then follows immediately that $\varepsilon_\infty(\Psi) = \varepsilon_\infty(\Phi)$
and $\varepsilon_{\eta,p}(\Psi) = \varepsilon_{\eta,p}(\Phi)$.
When $\eta$ is the uniform measure, we will denote $\varepsilon_{\eta,p}(\Phi)$
by $\varepsilon_{p}(\Phi)$.
Also, when $d=1$, all measures of distortion are identical, and we will
denote their common value by $\varepsilon(\Phi)$.

Clearly, $\varepsilon_{\eta,p}(\Phi) \le \varepsilon_\infty(\Phi)$
for all $p$.
The following result shows what is lost when the inequality is reversed:
\begin{prop}
\label{prop:deformations}
Let $x^* = \operatorname{arg\,max}_x |x - \Phi(x)|$,
and let $u^* = (x^* - \Phi(x^*)) / |x^* - \Phi(x^*)|$. Then
\begin{align}
\varepsilon_{\eta,p}(\Phi)
\ge \varepsilon_{\infty}(\Phi)
    \cdot \left(\int_{\mathbb{S}^{d-1}} |\langle u^*, u\rangle|^p \, d\eta(u) \right)^{1/p}.
\end{align}
\end{prop}

\begin{proof}
By definition, $\varepsilon_\infty(\Phi) = |x^* - \Phi(x^*)|$.
For all unit vectors $u$,
\begin{align}
\varepsilon(\Phi,u) = \max_{x} |\langle x - \Phi(x),u\rangle|
\ge |\langle x^* - \Phi(x^*),u\rangle|
= |x^* - \Phi(x^*)| \cdot |\langle u^*,u\rangle|
= \varepsilon_{\infty}(\Phi) \cdot |\langle u^*,u\rangle|.
\end{align}
The result then follows from averaging over $u$.

\end{proof}

For example, if $d=2$ and $\eta$ is the uniform measure over $\mathbb{S}^1$,
\begin{align}
\left(\int_{S^{1}} |\langle u^*, u\rangle|^p \, d\eta(u) \right)^{1/p}
= \left(\frac{2}{\pi}\int_{0}^{\pi/2} \cos(\theta)^p \, d\theta \right)^{1/p}
= \left(\frac{ \Gamma(p/2+1/2)}{\Gamma(p/2 + 1) \sqrt{\pi}  } \right)^{1/p}.
\end{align}

\subsection{Mean mixed norms}
\label{sec:mean_mixed_norms}

Fix a probability measure $\eta$ over the unit sphere $\mathbb{S}^{d-1}$ in $\RR^{d}$.
We define the \emph{mean mixed norm} of $f : \RR^d \to \RR$ by
\begin{align}
\label{eq:mean_mixed_norm}
\|f\|_{M_\eta^{p,r}} = \left(\int_{\mathbb{S}^{d-1}}\| \calP_u(|f|) \|_{L^p(\RR)}^r d\eta(u) \right)^{1/r}
\end{align}
for any $1 \le p \le \infty$ and $1 \le r < \infty$;
and
\begin{align}
\|f\|_{M^{p,\infty}} =
\|f\|_{M_\eta^{p,\infty}} =
\operatorname*{ess \, sup}_{u \in \mathbb{S}^{d-1}} \| \calP_u(|f|) \|_{L^p(\RR)}.
\end{align}
When $p=r$, we will define $\|f\|_{M_\eta^{p}} \equiv \|f\|_{M_\eta^{p,p}}$.

\begin{rmk}
For a fixed $u$, $\| \calP_u(|f|) \|_{L^p(\RR)}$
is an example of a \emph{mixed norm} \cite{huang2021function},
namely, the $L^p$ norm of the $L^1$ norm of $f$. The mean mixed norm
is then obtained by averaging this mixed norm over the choice of
$L^p$ variable.
\end{rmk}

\begin{rmk}
\label{rmk:mp_vs_lp}
If $f$ is supported on a bounded open set and $f \in L^p$,
then $\|f\|_{M_\eta^{p,r}} < \infty$.
More precisely, if the support of $f$ has diameter $L>0$,
then, $\|f\|_{M_\eta^{p,r}} \le L^{(d-1)(p-1)/p} \|f\|_{L^p}$.
\end{rmk}

\begin{rmk}
If $\eta$ is the uniform measure,  or if $r = \infty$,
then $\|f\|_{M_\eta^{p,r}}$
is rotationally invariant.
\end{rmk}

\subsection{Sliced metrics}
\label{sec:sliced_distances}

Given a metric $\mathrm{D}(f,g)$ between univariate functions,
a value $p \ge 1$, and a probability density $\eta$ over the unit sphere in $\RR^d$,
one can define a corresponding \emph{sliced}
metric $\mathrm{SD}_{\eta,p}(f,g)$ defined between functions $f$ and $g$
of $d$ variables, as follows:
\begin{align}
\mathrm{SD}_{\eta,p}(f,g)
= \left(\int_{\mathbb{S}^{d-1}}
            \mathrm{D}(\calP_u f, \calP_u g)^p d\eta(u) \right)^{1/p},
\end{align}
with the obvious modification when $p=\infty$.
That is, $\mathrm{SD}_{\eta,p}(f,g)$ is obtained by averaging the
distances between the one-dimensional projections of $f$ and $g$
over all directions.
For background on sliced metrics,
see, for instance, the references
\cite{rabin2011wasserstein, bonneel2015sliced,
kolouri2019generalized,
kolouri2018sliced, deshpande2018generative, nietert2022statistical, shi2025fast,
nguyen2023energy, nguyen2021distributional}.
In typical applications of sliced metrics,
$\eta$ is the uniform measure over $\mathbb{S}^{d-1}$;
however, other choices and methods for choosing $\eta$
have been proposed in the literature
\cite{nguyen2023energy, nguyen2021distributional},
and in this paper we consider general $\eta$.
There also exist generalizations of sliced metrics
to other geometries \cite{le2019tree, tran2025tree},
though we will not consider these in the present work.

\subsection{Wasserstein distances}
\label{sec:wasserstein}

Recall that if $f$ and $g$ are probability densities on a subset $\Omega \subset \RR^d$,
their \emph{$p$-Wasserstein distance} $\wass_p(f,g)$ (also known as the
Kantorovich distance) is defined as
\begin{align}
\label{eq:wasserstein}
\wass_p(f,g) = \inf_{\Pi \in \mathcal{M}(f,g)} 
    \left(\int_{\Omega} \int_{\Omega} |x-y|^p d\Pi(x,y)\right)^{1/p},
\end{align}
where $\mathcal{M}(f,g)$ denotes the space of all probability measures on
$\Omega \times \Omega$ with marginals equal to $f$ and $g$, respectively
\cite{villani2003topics, villani2008optimal}.
That is, $\Pi \in \mathcal{M}(f,g)$ if for all measurable $E \subset \Omega$,
\begin{align}
\Pi(E \times \Omega) = \int_E f(x) dx,
\end{align}
and
\begin{align}
\Pi(\Omega \times E) = \int_E g(y) dy.
\end{align}

Informally, $\wass_p(f,g)$ is the minimal cost of rearranging
a unit of mass with distribution $f$ into one with distribution
$G$, where the cost of moving mass between locations $x$ and $y$
is $|x - y|^p$.
The distance $\wass_1(f,g)$ is also known as the \emph{Earth Mover's Distance (EMD)}
between the probability measures $f$ and $g$ \cite{villani2003topics, villani2008optimal}.
The Wasserstein distances and their variants have been widely used in statistics,
machine learning, image processing, and related areas
\cite{panaretos2019statistical, peyre2019computational, santambrogio2015optimal,
bonneel2015sliced, rabin2011wasserstein, rubner2000earth,
levina2001earth, bernton2019parameter, rigollet2019uncoupled,
moosmuller2023linear, cloninger2019people}.

\subsubsection{Wasserstein distances and deformations}

The Wasserstein distance
is a relaxation of the Monge distance, defined by
\begin{align}
\monge_p(f,g)
= \inf_{\Phi \in \mathcal{T}(f,g)} 
    \left(\int_{\Omega} |x-\Phi(x)|^p f(x) \,dx \right)^{1/p}
\end{align}
where $\mathcal{T}(f,g)$ contains those functions $\Phi : \Omega \to \Omega$
such that $\int_{E} g(x) dx = \int_{\Phi^{-1}(E)} f(x) dx$,
that is, which push $f$ onto $g$.
Indeed, any $\Phi$ in $\calT(f,g)$ induces a measure $\Pi_\Phi$
in $\calM(f,g)$, with
\begin{align}
\label{eq:wp_monge}
\int_{\Omega} \int_{\Omega} |x-y|^p d\Pi_\Phi(x,y)
= \int_{\Omega} |x-\Phi(x)|^p f(x) \,dx,
\end{align}
and hence $\wass_p(f,g) \le \monge_p(f,g)$. (In fact, when $\monge_p(f,g)$
is finite, equality holds; see \cite{santambrogio2015optimal}.)
Consequently, if $\Phi$ is a smooth bijection on $\Omega$
and $f_\Phi(x) = f(\Phi(x))|\det(\nabla \Phi(x))|$,
then $\Phi$ is contained in $\calT(f,f_\Phi)$, and so
\begin{align}
\label{eq:wasserstein_deformation2}
\wass_p(f,f_\Phi)
&\le \monge_p(f,f_\Phi)
\nonumber \\
&\le
\left(\int_{\Omega} |x-\Phi(x)|^p f(x) \,dx \right)^{1/p}
\nonumber \\
&\le \varepsilon_\infty(\Phi) \left(\int_{\Omega} f(x) \,dx \right)^{1/p}
\nonumber \\
&= \varepsilon_\infty(\Phi).
\end{align}

In fact, a more general robustness result may be easily shown,
which we state now.
\begin{thm}
\label{thm:wasserstein_projections}
Suppose $f$ is a probability density supported on
a bounded, open set $\Omega \subset \RR^D$,
and let $\Phi : \Omega \to \Omega$ be a smooth bijection.
Let $\calQ$ be a tomographic projection operator onto a $d$-dimensional
subspace, $d \le D$.
Then for all $p \ge 1$,
\begin{align}
\wass_p(\calQ f, \calQ f_\Phi) \le \varepsilon_\infty(\Phi).
\end{align}

\begin{proof}
An identical proof to that of Lemma 1 in \cite{rao2020wasserstein} shows that
$\wass_p(\calQ f, \calQ f_\Phi) \le \wass_p(f,f_\Phi)$
(note that the left side refers to transportation in $\RR^{d}$,
and the right side to $\RR^D$).
The result then follows from \eqref{eq:wasserstein_deformation2}.

\end{proof}

\end{thm}

\subsubsection{Wasserstein distances in 1D}

It is well-known that Wasserstein distances in 1D take a on a particularly simple form.
Denote by $\calV$ the Volterra operator \cite{gohberg1970theory} on $L^1([a,b])$, defined by
\begin{align}
\label{eq:volterra_operator}
(\calV f)(x) = \int_{a}^{x} f(t) dt.
\end{align}
Then when $d=1$, it is known \cite{santambrogio2015optimal}
that  $\wass_p(f,g)$ may be written as follows:
\begin{align}
\label{eq:wasserstein_1d}
\wass_p(f,g) = \| (\calV f)^{-1} - (\calV g)^{-1} \|_{L^p}.
\end{align}
Here, $(\calV f)^{-1}$ denotes the functional inverse of
$\calV f$, defined as
\begin{align}
(\calV f)^{-1}(x) = \inf\{t \in [a,b] \ : \ (\calV f)(t) \ge x\}.
\end{align}
When $p=1$, it is also true that
$\wass_1(f,g) = \| \calV f - \calV g \|_{L^1}$.

\subsubsection{Sliced Wasserstein}

Given $p \ge 1$ and a probability density $\eta$ over $\mathbb{S}^{d-1}$,
we denote by $\SW_{\eta,p}$ the sliced Wasserstein distance:
\begin{align}
\label{eq:sliced_wasserstein}
\mathrm{SW}_{\eta,p}(f,g)
= \left(\int_{\mathbb{S}^{d-1}}
            \wass_p(\calP_u f, \calP_u g)^p d\eta(u) \right)^{1/p}.
\end{align}
Sliced Wasserstein distances have been the subject of
considerable research activity in recent years \cite{rabin2011wasserstein, bonneel2015sliced,
kolouri2019generalized,
kolouri2018sliced, deshpande2018generative, nietert2022statistical, shi2025fast}.
The work \cite{shi2025fast} proves that sliced Wasserstein
distances are robust to rotations and translations,
and also describes a fast discretization.

We prove an analogue of Theorem \ref{thm:wasserstein_projections}
for sliced Wasserstein distances:

\begin{thm}
\label{thm:sw_deformations}
Suppose $f$ is a probability density supported on
a bounded, open set $\Omega \subset \RR^D$, and let $\Phi : \Omega \to \Omega$ be a smooth bijection.
Let $\calQ$ be a tomographic projection operator onto $d$-dimensional
subspace, $d \le D$.
Then for all $p \ge 1$,
\begin{align}
\SW_{\eta,p}(\calQ f, \calQ f_\Phi) \le \varepsilon_\infty(\Phi).
\end{align}

\end{thm}

\begin{proof}
Without loss of generality, suppose $\calQ$
projects onto the first $d$ coordinates.
Let $u \in \mathbb{S}^{d-1}$.
It is easy to see (and will be shown later, in the proof of Corollary \ref{cor:main_projections}
in Section \ref{section:properties}) that
\begin{math}
(\calP_u \calQ f)(t) = (\calP_{(u,0)} f)(t),
\end{math}
which is the tomographic projection of $f$ onto the span of
$(u,0) \in \RR^d \times \RR^{D-d}$.
Consequently, from Theorem \ref{thm:wasserstein_projections},
\begin{align}
\wass_p(\calP_u \calQ f, \calP_u \calQ f_\Phi)
= \wass_p(\calP_{(u,0)} f, \calP_{(u,0)} f_\Phi)
\le \varepsilon_\infty(\Phi).
\end{align}
The result now follows by averaging over $u$.

\end{proof}

\subsubsection{Wasserstein distances and convolution}

In signal and image processing applications,
one typically observes signals/images that have been
convolved with a filter induced from the measurement
process. It is therefore of interest to understand
how metrics behave when their inputs are convolved
by a common function. For Wasserstein  distances,
we have the following result:

\begin{thm}
\label{thm:wasserstein_convo}
Suppose $f$, $g$  and $w$ are  probability
densities on $\RR^d$. Then for all $p \ge 1$,
\begin{align}
\wass_p(f \ast w, g \ast w) \le \wass_p(f,g).
\end{align}
\end{thm}

This property is referred to by Zolotarev
as \emph{regularity} of the metric
\cite{zolotarev1979ideal, zolotarev1976metric,
zolotarev1984probability}; note that an equivalent
statement of regularity for
a metric $\mathrm{D}$ between random variables
is that for all random variables $X$, $Y$ and $Z$,
where $Z$ is independent of $X$ and $Y$,
$D(X+Z,Y+Z) \le D(X,Y)$.
While the result is certainly known and is straightforward to
prove, we had difficulty locating
a reference containing a proof of the precise statement;
for the reader's convenience we provide a proof here.

\begin{proof}[Proof of Theorem \ref{thm:wasserstein_convo}]
By Kantorovich duality (see, e.g., Chapter 5 in \cite{villani2008optimal}),
\begin{align}
\wass_p(f,g)^p
= \sup_{(\varphi,\psi) \in \calF_p} \left\{ \int f(x) \varphi(x) \, dx
        + \int g(y) \psi(y) \, dy \right\},
\end{align}
where $\calF_p$ contains all pairs $(\varphi,\psi)$
of integrable functions $\varphi$ and $\psi$
satisfying $\varphi(x) + \psi(y) \le |x-y|^p$.
Take any such $\varphi$ and $\psi$.
Letting $\wtilde w(z) = w(-z)$,
we have
\begin{align}
\int (f \ast w)(x) \varphi(x) \, dx
= \int (\varphi \ast \wtilde w)(z) f(z) \, dz,
\end{align}
and
\begin{align}
\int (g \ast w)(y) \psi(y) \, dy
= \int (\psi \ast \wtilde w)(z) g(z) \, dz.
\end{align}
For any $x$ and $y$ we have
\begin{align}
(\varphi \ast \wtilde w)(x) + (\psi \ast \wtilde w)(y)
= \int (\varphi(x-z)+\psi(y-z)) \wtilde w(z) \,dz
\le |x-y|^p \int \wtilde w(z) \,dz
= |x-y|^p.
\end{align}
Therefore, $(\varphi \ast \wtilde w,  \psi \ast \wtilde w) \in \calF_p$,
and
\begin{align}
\int (f \ast w)(x) \varphi(x) \, dx
        + \int (g \ast w)(y) \psi(y) \, dy
= \int (\varphi \ast \wtilde w)(z) f(z) \, dz
        + \int (\psi \ast \wtilde w)(z) g(z) \, dz
\le \wass_p(f,g),
\end{align}
and so taking the supremum over all $(\varphi,\psi) \in \calF_p$
proves the result.

\end{proof}

\begin{cor}
\label{cor:sw_convo}
Suppose $f$, $g$  and $w$ are probability
densities on $\RR^d$, and let $\eta$ be a probability
density over $\mathbb{S}^{d-1}$.
Then for all $p \ge 1$,
\begin{align}
\SW_p(f \ast w, g \ast w) \le \SW_p(f,g).
\end{align}

\end{cor}

\begin{proof}
It is straightforward to see (and will be shown in Section \ref{section:convolutions})
that for any unit vector $u$ in $\RR^d$,
$\calP_u (f \ast w) = (\calP_u f) \ast (\calP_u w)$.
Then from the $d=1$ case of Theorem \ref{thm:wasserstein_convo},
\begin{align}
\wass_p(\calP_u (f \ast w), \calP_u (g \ast w))
= \wass_p((\calP_u f) \ast (\calP_u w), (\calP_u g) \ast (\calP_u w))
\le \wass_p(\calP_u f, \calP_u g).
\end{align}
Averaging over all $u \in \mathbb{S}^{d-1}$ then proves the result.

\end{proof}

\begin{rmk}
Young's convolutional inequality (e.g.\ see Chapter 8
in \cite{folland1999real}) states that the same
property holds for the ordinary $L^p$ distances on $\RR^d$, namely,
for $f$ and  $g$ in $L^p$ and $w$ in $L^1$,
\begin{math}
\|f\ast w - g \ast w\|_{L^p} \le \|w\|_{L^1} \|f-g\|_{L^p}.
\end{math}

\end{rmk}

\subsection{Cram\'er and sliced Cram\'er metrics}
\label{sec:cramer_defined}

We introduce the primary objects of interest
in this paper: the Cram\'er
and sliced Cram\'er distances.

\subsubsection{Cram\'er distances}

Let $f$ be in $L^1(a,b)$. For any value $1 \le p  \le \infty$,
we will refer to
\begin{align}
\label{eq:volterra_norm}
\|f\|_{V^p} \equiv \|\calV f\|_{L^p}
\end{align}
as the \emph{Volterra $p$-norm} of $f$.
Note that, because $\calV f$ is in $L^\infty(a,b)$, the Volterra $p$-norm of $f$ is finite
for any function $f$ in $L^1(a,b)$.
We denote by
\begin{align}
\label{eq:cramer_distance}
\cdist_p(f,g) = \|f-g\|_{V^p}
\end{align}
the \emph{$p$-Cram\'er metric}
(or \emph{$p$-Cram\'er distance})
between functions $f$ and $g$,
named after Harald Cram\'er \cite{cramer1928composition,
rizzo2016energy, szekely2013energy}. (In \cite{zolotarev1984probability}, Zolotarev
refers to the $p$-Cram\'er distance more simply as the ``$L_p$-metric''.)

\begin{rmk}
While the Cram\'er metrics are typically used to compare probability
distributions, they are well-defined for any $f$ and $g$ in $L^1(a,b)$.
Note, however, that it is most natural to compare functions
supported on $(a,b)$ and for which
$\int_{a}^{b} f = \int_{a}^{b} g$, since otherwise
the distance may change by enlarging the interval.
The results in this paper assume that $\int_{a}^{b} f = \int_{a}^{b} g$;
however, they do not require $f$ and $g$ to be non-negative.
\end{rmk}

The following result provides a dual formulation of the Volterra $p$-norm
(and hence of the $p$-Cram\'er metric)
that will be useful in our subsequent analysis.
It essentially appears as Theorem 1 in \cite{maejima1987ideal};
because of its key role in this paper,
we provide a self-contained proof for the reader's convenience.

\begin{prop}
\label{prop:variational}
Let $1 \le p  \le \infty$ and let $q$ be the conjugate exponent:
\begin{math}
1/p + 1/q = 1.
\end{math}
Then for any  function $f$ in $L^1(a,b)$,
\begin{align}
\label{eq:variational}
\|f\|_{V^p}
= \sup_{G \in \calA_0 : \|G'\|_{L^q} \le 1} \langle f, G \rangle.
\end{align}
\end{prop}

\begin{proof}

First, note that
the adjoint transform $\calV^*$ is given by
\begin{align}
(\calV^* f)(x) = \int_{x}^{b} f(t) dt.
\end{align}
This operator satisfies
\begin{align}
\langle \calV f, g \rangle = \langle f, \calV^* g \rangle
\end{align}
where $f$ and $g$ are two functions in $L^1(a,b)$.

By duality of $L^p$ and $L^q$, we have:
\begin{align}
\|f\|_{V^p} = \|\calV f\|_{L^p}
= \sup_{g:\|g\|_{L^q} \le 1} \int_{a}^{b}  (\calV f)(x) g(x)dx
= \sup_{g:\|g\|_{L^q} \le 1} \langle \calV f ,g \rangle
= \sup_{g:\|g\|_{L^q} \le 1} \langle f , \calV^* g \rangle.
\end{align}
Any function of the form $\calV^* g$ is contained in $\calA_0$,
and any function $G$ in $\calA_0$ is of the form
$G = \calV^* g$ where $g = G'$ almost everywhere. Consequently:
\begin{align}
\|f\|_{V^p}
= \sup_{g:\|g\|_{L^q} \le 1} \langle f , \calV^* g \rangle
= \sup_{G \in \calA_0:\|G'\|_{L^q} \le 1} \langle f , G \rangle,
\end{align}
which completes the proof.

\end{proof}

\subsubsection{Sliced Cram\'er distances}

Following the framework from
Section \ref{sec:sliced_distances},
given a probability measure $\eta$
over the unit sphere $\mathbb{S}^{d-1} \subset \RR^d$,
for all $1 \le p < \infty$
we define \emph{sliced $p$-Cram\'er metric}
between $f,g : \RR^d \to \RR$ as
\begin{align}
\label{eq:sliced_cramer}
\SC_{\eta,p}(f,g)
= \left( \int_{\mathbb{S}^{d-1}} \cdist_p(\calP_u f,\calP_u g)^p \, d\eta(u) \right)^{1/p},
\end{align}
and when $p=\infty$ we define
\begin{align}
\SC_{\infty}(f,g) = \SC_{\eta,\infty}(f,g)
= \sup_{u \in \mathbb{S}^{d-1}} \cdist_\infty(\calP_u f,\calP_u g).
\end{align}

\begin{rmk}
Cram\'er and sliced Cram\'er metrics have garnered attention in recent years
as metrics for comparing probability measures
in machine learning applications
\cite{nadjahi2020statistical, kolouri2022generalized,
bellemare2017cramer, lheritier2022cramer,
kolouri2020sliced}
and image processing \cite{shi2025fast}.
The $\infty$-Cram\'er metric is also known as the Kolmogorov Metric between $f$ and $g$
\cite{gibbs2002choosing},
and arises in the context of goodness-of-fit
testing in statistics \cite{massey1951kolmogorov}.
The $1$-Cram\'er metric is equal to the $1$-Wasserstein distance
between the probability distributions
$f$ and $g$,
described in Section \ref{sec:wasserstein}.
The sliced $2$-Cram\'er distances are also equal (up to a change in scaling)
to the multi-dimensional energy distances used in statistical
testing \cite{rizzo2016energy, szekely2013energy},
as remarked upon in \cite{xu2025manifold}.
\end{rmk}

For a function $f$ of two variables, we define its \emph{sliced Volterra norm} by
\begin{align}
\label{eq:sliced_volterra}
\|f\|_{SV_\eta^{p}} = \left(\int_{\mathbb{S}^{d-1}} \|\calP_u f\|_{V^p}^p \, d \eta(u) \right)^{1/p}
\end{align}
when $1 \le p < \infty$,
and
\begin{align}
\|f\|_{SV^{\infty}} = \|f\|_{SV_\eta^{\infty}} 
= \sup_{u \in \mathbb{S}^{d-1}} \|\calP_u f\|_{V^\infty}.
\end{align}
Then for two functions $f$ and $g$, $\SC_{\eta,p}(f,g) = \|f-g\|_{SV_\eta^{p}}$.
When $\eta$ is the uniform measure over $\mathbb{S}^{d-1}$, we will denote the sliced
Cram\'er metric more simply as $\SC_p(f,g)$, and the sliced Volterra norm
more simply as $\|f\|_{SV^p}$.

In Section \ref{section:properties}, we will study the geometric
properties of sliced Cram\'er metrics,
characterizing their growth under deformations.
In Section \ref{section:discrete}
we will analyze efficient, Fourier-based discretizations for the 1D and
2D distances (the latter based on those in \cite{shi2025fast})
between functions with equal integrals,
and prove their  robustness to additive heteroscedastic noise.

\section{Properties of sliced Cram\'er metrics}
\label{section:properties}

In this section, we will provide bounds on
the sliced Cram\'er distance between a function and its
deformation in terms of the deformation size.
As described in Remark \ref{rmk:smoothness} below,
these bounds also imply that
the sliced Cram\'er distances
are locally smooth
with respect to the distance $\Delta$ between deformations
defined in Section \ref{sec:pushforwards}.
We will also extend these results
to the sliced Cram\'er distance between tomographic projections.
We will compare these bounds to those satisfied by
the Wasserstein and sliced Wasserstein distances
found in Theorem \ref{thm:wasserstein_projections}
and Theorem \ref{thm:sw_deformations}, respectively,
as well as to the Lebesgue distances.
Bounds for general deformations are
provided in Theorem \ref{thm:main_deformations} and Corollary \ref{cor:main_projections},
in Section \ref{section:robustness_general}. Sharper
bounds are then derived for more specific deformations
in Section \ref{section:sharper}.
In addition, Section \ref{section:convolutions} analyzes the behavior of
the sliced Cram\'er distances under convolution of the inputs,
stating a bound analagous to Theorem \ref{thm:wasserstein_convo}
and Corollary \ref{cor:sw_convo}.

\subsection{Growth under deformations}
\label{section:robustness_general}

We start with a general result that quantifies
sliced Cram\'er metrics' growth under deformations.

\begin{thm}
\label{thm:main_deformations}
Let $1 \le p \le \infty$.
Suppose that $A$ and $B$ are non-empty, bounded, open sets in $\RR^d$.
Let $\Phi : B \to A$ be a $C^1$ bijection,
and $\eta$ be a probability measure over $\mathbb{S}^{d-1}$.
If $f$ is in $L^p(A)$,
then
\begin{align}
\label{eq:main1}
\SC_{\eta,p}(f,f_\Phi) \le
2^{(p-1)/p} \cdot \|f\|_{M_\eta^p} \cdot \varepsilon_\infty(\Phi)
\end{align}
and
\begin{align}
\label{eq:main2}
\SC_{\eta,p}(f,f_\Phi) \le
2^{(p-1)/p} \cdot \|f\|_{M^{p,\infty}} \cdot \varepsilon_{\eta,p}(\Phi).
\end{align}
If $f$ is in $L^1(A)$, then
\begin{align}
\label{eq:main3}
\SC_{\eta,p}(f,f_\Phi) \le \|f\|_{L^1} \cdot \varepsilon_{\eta,1}(\Phi)^{1/p}.
\end{align}

\end{thm}

We record several observations about Theorem \ref{thm:main_deformations},
before stating a corollary and providing the proof.

\begin{rmk}
\label{rmk:switch}
Because $f = (f_\Phi)_{\Phi^{-1}}$
and $\varepsilon_{\eta,p}(\Phi) = \varepsilon_{\eta,p}(\Phi^{-1})$,
one can switch the roles of $f$ and $f_\Phi$,
and thereby replace the term $\|f\|_{M_\eta^{p}}$
by $\min\{\|f\|_{M_\eta^{p}}, \|f_\Phi\|_{M_\eta^{p}}\}$
in \eqref{eq:main1},
and replace $\|f\|_{M_\eta^{p,\infty}}$
by $\min\{\|f\|_{M_\eta^{p,\infty}}, \|f_\Phi\|_{M_\eta^{p,\infty}}\}$
in \eqref{eq:main2},
to obtain sharper bounds.
\end{rmk}

\begin{rmk}
It is not possible to replace the mean mixed norms
with the $L^1$ norm in \eqref{eq:main1} and \eqref{eq:main2}.
Indeed, when $d=1$, let $f$ be the point mass at $0$,
and let $\Phi(x) = x+\epsilon$;
then $\cdist_p(f,f_\Phi) = \epsilon^{1/p}$,
showing that a bound of the form
$\cdist_p(f,f_\Phi) \le K(p) \|f\|_{L^1} \varepsilon(\Phi)$
is not possible in general.

\end{rmk}

\begin{rmk}
\label{rmk:smoothness}
One may reinterpret Theorem \ref{thm:main_deformations}
as saying that, for fixed functions $f$ and $g$,
the mapping $\Phi \mapsto \SC(f_\Phi,g)$
is locally Lipschitz, with respect to the
metric $\Delta$ defined in Remark \ref{rmk:deformation_distance},
in a neighborhood of the identity.
Indeed, defining  $\mathrm{I}(x) = x$,
the triangle inequality implies
\begin{align}
|\SC_{\eta,p}(f_\Phi,g) - \SC_{\eta,p}(f,g)|
&\le \SC_{\eta,p}(f_\Phi,f)
\nonumber \\
&\le 2^{(p-1)/p} \cdot \|f\|_{M_\eta^p} \cdot \varepsilon_\infty(\Phi \circ \mathrm{I}^{-1})
\nonumber \\
&= 2^{(p-1)/p} \cdot \|f\|_{M_\eta^p} \cdot \Delta(\Phi,\mathrm{I}).
\end{align}
Similarly, when $f$ is in $L^1$ the distances are locally
$1/p$-H\"older:
\begin{align}
|\SC_{\eta,p}(f_\Phi,g) - \SC_{\eta,p}(f,g)|
= \|f\|_{L^1} \cdot \Delta(\Phi,\mathrm{I})^{1/p}.
\end{align}

\end{rmk}

\begin{rmk}
\label{rmk:compare_wasserstein}
We compare the bounds from Theorem \ref{thm:main_deformations}
to the bounds
on the Wasserstein and sliced Wasserstein
distances described in Sections \ref{sec:wasserstein} and \ref{sec:sliced_distances}.
First, it is convenient to combine
the bounds \eqref{eq:main1} and \eqref{eq:main3}
to get the slightly weaker but simpler bound
\begin{align}
\label{eq:bounds_combined}
\SC_{\eta,p}(f,f_\Phi) \le
\begin{cases}
2^{(p-1)/p} \cdot \|f\|_{M_\eta^p} \cdot \varepsilon_\infty(\Phi), &\, \text{ if }
    \varepsilon_\infty(\Phi) \le
            \frac{1}{2}\left( \frac{\|f\|_{L^1}}{\|f\|_{M_\eta^p}}\right)^{p/(p-1)} \\
\|f\|_{L^1} \cdot \varepsilon_\infty(\Phi)^{1/p} , &\, \text{ otherwise}
\end{cases}.
\end{align}
Like Wasserstein, $\SC_{\eta,p}(f,f_\Phi)$ is bounded by a linear
function of the maximal displacement $\varepsilon_\infty(\Phi)$.
However, we note two differences between
the metrics.
First, the sliced Cram\'er bound
exhibits slower growth with respect to $\varepsilon_\infty(\Phi)$
when $\varepsilon_\infty(\Phi)$
is large, growing like $\varepsilon_\infty(\Phi)^{1/p}$
instead of $\varepsilon_\infty(\Phi)$.
Second, for small $\varepsilon_\infty(\Phi)$, while both bounds
are linear in $\epsilon_\infty(\Phi)$,
they differ in their dependence on $f$,
with the bound on the Wasserstein distance scaling with $\|f\|_{L^1}$
and the bound on the sliced Cram\'er distance scaling with $\|f\|_{M_\eta^p}$.
While the norms
$\|f\|_{L^1}$ and $\|f\|_{M_\eta^p}$ are generally not commensurable,
we observe that
if the support of $f$
has diameter $\delta > 0$, say, then $\|f\|_{L^1} \le \delta^{(p-1)/p} \|f\|_{M_\eta^p}$;
consequently, if $f$ is a probability density, then
\begin{align}
\wass_p(f,f_\Phi) \le \varepsilon_\infty(\Phi)
\le \delta^{(p-1)/p} \cdot \|f\|_{M_\eta^p} \cdot \varepsilon_\infty(\Phi),
\end{align}
which is smaller than the bound on $\SC_{\eta,p}(f,f_\Phi)$
when $\delta < 2$.

\end{rmk}

\begin{rmk}
\label{rmk:compare_translations}
We compare Theorem \ref{thm:main_deformations} with the $L^p$ distance,
in the case of translations.
Suppose $f$ is supported on a set of diameter $\delta$,
with $\|f\|_{L^p} = 1$,
and let $f_\epsilon(x) = f(x - \epsilon u)$,
where $u$ is a fixed unit vector.
Then for any $\epsilon > \delta$,
$\|f - f_\epsilon\|_{L^p} = 2^{1/p}$.
Therefore, close to $\epsilon = 0$, viewing
$\|f - f_\epsilon\|_{L^p}$
as a function of $\epsilon$,
its slope over $[0,\delta]$ is $2^{1/p} / \delta$,
which can be arbitrarily big when $\delta$
is small, i.e.\ when $f$ has a large oscillation.
We contrast this with the behavior of $\SC_{\eta,p}(f,f_\Phi)$.
Note that in the case of translations,
Theorem \ref{thm:translations} below shows that the factor $2^{(p-1)/p}$
in the bounds \eqref{eq:main1} and \eqref{eq:main2}
may be removed, and so the sliced Cram\'er distances satisfy
the simplified bound
\begin{align}
\label{eq:translations_combined}
\SC_{\eta,p}(f,f_\Phi) \le
\begin{cases}
\|f\|_{M_\eta^p} \cdot \epsilon, &\, \text{ if }
    \epsilon \le
            \left( \frac{\|f\|_{L^1}}{\|f\|_{M_\eta^p}}\right)^{p/(p-1)} \\
\|f\|_{L^1} \cdot \epsilon^{1/p} , &\, \text{ otherwise}
\end{cases}.
\end{align}
With this, we make two observations. First, since
\begin{math}
\|f\|_{L^1} \le \delta^{(p-1)/p} \cdot \|f\|_{M_\eta^p},
\end{math}
the transition between the $\epsilon$
and $\epsilon^{1/p}$ regimes
in \eqref{eq:translations_combined}
occurs at the translation size
\begin{align}
\epsilon =
\left( \frac{\|f\|_{L^1}}{\|f\|_{M_\eta^p}}\right)^{p/(p-1)}
\le \delta;
\end{align}
this is qualitatively
similar behavior to the Lebesgue distance,
which also changes behavior when $\epsilon > \delta$,
becoming constant.
However, the sliced Cram\'er distance
exhibits much slower (and hence smoother) growth
than the Lebesgue distance
in the small $\epsilon$ regime;
indeed,
we can bound the sliced Cram\'er distance,
for $f$ with $\|f\|_{L^p}=1$, by
\begin{align}
\SC_{\eta,p}(f, f_\epsilon) \le \|f\|_{M_\eta^p} \cdot \epsilon
\le \delta^{(d-1)(p-1)/p} \cdot \|f\|_{L^p} \cdot \epsilon
= \delta^{(d-1)(p-1)/p} \cdot \epsilon,
\end{align}
i.e. the slope is bounded by $\delta^{(d-1)(p-1)/p}$,
in contrast to the Lebesgue distance's slope of
$2^{1/p} / \delta$.

\end{rmk}

Theorem \ref{thm:main_deformations} is easily extended
to comparing tomographic projections of a function and its deformation.
This is of interest when measuring the distance between two 2D projections
of a 3D volume, such as in the analysis of images in cryo-electron microscopy.

\begin{cor}
\label{cor:main_projections}
Let $1 \le p \le \infty$.
Let $A$ and $B$ be  non-empty, bounded, open sets in $\RR^D$,
$f$ be in $L^p(A)$,
$\Phi : B \to A$ be a $C^1$ bijection,
and
\begin{math}
f_\Phi(x) = f(\Phi(x)) |\det(\nabla \Phi(x))|
\end{math}
on $B$, and $0$ elsewhere.
For $d \le D$, let $\calQ$ denote the tomographic projection operator onto
a $d$-dimensional subspace of $\RR^D$.
Then for any probability measure $\eta$ over $\mathbb{S}^{d-1}$,
\begin{align}
\label{eq:proj1}
\SC_{\eta,p}(\calQ f,\calQ f_\Phi)
\le 2^{(p-1)/p} \cdot \|\calQ (|f|)\|_{M_\eta^p} \cdot \varepsilon_\infty(\Phi),
\end{align}
and
\begin{align}
\label{eq:proj2}
\SC_{\eta,p}(\calQ f,\calQ f_\Phi)
\le \|f\|_{L^1} \cdot \varepsilon_\infty(\Phi)^{1/p}.
\end{align}

\end{cor}

\begin{proof}[Proof of Corollary \ref{cor:main_projections}]

Without loss of generality, suppose $\calQ$ projects onto the first
$d$ coordinates, that is,
\begin{align}
(\calQ f)(x) = \int_{\RR^{D-d}} f(x,y) dy.
\end{align}

Let $u$ be a unit vector in $\RR^d$, and
let $u^{(2)},\dots,u^{(d)}$ denote any orthonormal vectors
completing the basis, so that for $h$ on $\RR^d$,
\begin{align}
(\calP_u h)(t) = \int_{\RR^{d-1}} h(t u + s_2 u^{(2)} + \dots + s_{d} u^{(d)}) \, d s,
\end{align}
and so
\begin{align}
(\calP_u (\calQ h))(t)
&= \int_{\RR^{d-1}} (\calQ h)(t u + s_2 u^{(2)} + \dots + s_{d} u^{(d)}) \, ds
\nonumber \\
&= \int_{\RR^{d-1}} \int_{\RR^{D-d}} h(t u + s_2 u^{(2)} + \dots + s_{d} u^{(d)},y) \,dy \,ds
\nonumber \\
&= (\calP_{(u,0)} h)(t),
\end{align}
which is the tomographic projection of $h$ onto the span of
$(u,0) \in \RR^d \times \RR^{D-d}$.

Denote by $\wtilde \eta$ the distribution
over the unit sphere $\mathbb{S}^{D-1}$, supported on
$\mathbb{S}^{d-1} \equiv \{(u,0) \in \RR^{d} \times \RR^{D-d}: |u|=1\}$
and defined by
$d\wtilde \eta((u,0)) = d\eta(u)$ for $u \in \RR^d$.
We have
\begin{align}
\SC_{\wtilde \eta,p}(f, f_\Phi)^p
&= \int_{\mathbb{S}^{D-1}} \cdist_p(\calP_v f, \calP_v f_\Phi)^p \, d\wtilde \eta(v)
\nonumber \\
&= \int_{\mathbb{S}^{d-1}} \cdist_p(\calP_{(u,0)} f, \calP_{(u,0)} f_\Phi)^p \, d\eta(u)
\nonumber \\
&= \int_{\mathbb{S}^{d-1}} \cdist_p(\calP_{u} \calQ f, \calP_{u} \calQ f_\Phi)^p \, d\eta(u)
\nonumber \\
&= \SC_{\eta,p}(\calQ f, \calQ f_\Phi)^p.
\end{align}

Furthermore,
\begin{align}
\|f\|_{M_{\wtilde \eta}^p}^p
&= \int_{\mathbb{S}^{D-1}} \|\calP_v (|f|)\|_{L^p}^p \, d\wtilde \eta(v)
\nonumber \\
&= \int_{\mathbb{S}^{d-1}} \|\calP_{(u,0)} (|f|) \|_{L^p}^p \, d\eta(u)
\nonumber \\
&= \int_{\mathbb{S}^{d-1}} \|\calP_{u} \calQ (|f|) \|_{L^p}^p \, d\eta(u)
\nonumber \\
&= \|\calQ(|f|)\|_{M_{\eta}^p}^p.
\end{align}
The inequality \eqref{eq:proj1}  then follows by applying
\eqref{eq:main1} in Theorem \ref{thm:main_deformations}
with the measure $\wtilde \eta$. The bound \eqref{eq:proj2}
follows from \eqref{eq:main3} and the fact that
$\varepsilon_{\wtilde \eta,1}(\Phi) \le \varepsilon_\infty(\Phi)$.

\end{proof}

The proof of Theorem \ref{thm:main_deformations}
is immediate from the following lemma:

\begin{lem}
\label{lem:projpert}
Let $1 \le p \le \infty$.
Let $A$ and $B$ be  non-empty, bounded, open sets in $\RR^d$,
$f$ be in $L^p(A)$,
$\Phi : B \to A$ be a $C^1$ bijection
and
\begin{math}
f_\Phi(x) = f(\Phi(x)) |\det(\nabla \Phi(x))|
\end{math}
on $B$, and $0$ elsewhere.
Then for any $u \in \mathbb{S}^{d-1}$,
\begin{align}
\label{eq:projpert}
\cdist_p(\calP_u f, \calP_u f_\Phi)
\le \min\left\{2^{(p-1)/p} \cdot \|\calP_u (|f|)\|_{L^p}  \cdot \varepsilon(\Phi,u) , \
               \|f\|_{L^1} \cdot \varepsilon(\Phi,u)^{1/p}\right\}.
\end{align}
\end{lem}

Theorem \ref{thm:main_deformations} follows easily from averaging each side over $u$.

\begin{proof}[Proof of Lemma \ref{lem:projpert}]

Without loss of generality, suppose $u = e_1 = (1,0,\dots,0)$;
then
\begin{align}
\varepsilon(\Psi,u) = \max_{(x,y) \in \RR \times \RR^{d-1}} |x - \psi_1(x,y)|.
\end{align}

For brevity, if $h : \RR^d \to \RR$ is a function of $d$ variables,
let
\begin{math}
\calP h = \calP_{e_1} h.
\end{math}
That is,
\begin{align}
(\calP h)(x) = \int_{\RR^{d-1}} h(x,y) dy.
\end{align}

For $(x,y) \in \RR \times \RR^{d-1}$, let $I_{(x,y)}$ be the interval $[x,\psi_1(x,y)]$
when $x \le \psi_1(x,y)$,
and $[\psi_1(x,y),x]$ when $x > \psi_1(x,y)$;
and let $\chi(x,y,t)$ be $1$ if $t \in I_{(x,y)}$,
and $0$ otherwise;
that is, $\chi(x,y,t) = 1$
if either $x \le t \le \psi_1(x,y)$
or $\psi_1(x,y) \le t \le x$.

\paragraph{Step 1.}

We will show that for all $(x,y) \in \RR \times \RR^{d-1}$,
\begin{align}
\label{eq:56400}
\int \chi(x,y,t) \, dt \le \varepsilon(\Psi,u),
\end{align}
and for all $t$,
\begin{align}
\label{eq:56401}
\int \sup_{y} \chi(x,y,t) \, dx \le 2\varepsilon(\Psi,u).
\end{align}

For the first inequality, for fixed $(x,y)$, suppose
without loss of generality that $x \le \psi_1(x,y)$. Then
$\chi(x,y,t) = 1$ if and only if $x \le t \le \psi_1(x,y)$,
and so
\begin{align}
\int \chi(x,y,t) \, dt = |\psi_1(x,y) - x|
\le \varepsilon(\Psi,u),
\end{align}
which is \eqref{eq:56400}.

For the second inequality: for any $x$ and $t$, $\sup_{y} \chi(x,y,t) = 1$
if and only if there exists a vector $y$ such that
either $x \le t \le \psi_1(x,y)$ or $\psi_1(x,y) \le t \le x$.
In this case, since $|x - \psi_1(x,y)| \le \varepsilon(\Psi,u)$,
we must also have $|x-t| \le \varepsilon(\Psi,u)$, and so $x$ lies in the interval
$[t-\varepsilon(\Psi,u),t+\varepsilon(\Psi,u)]$ of length $2 \varepsilon(\Psi,u)$; hence
\begin{align}
\int \sup_{y} \chi(x,y,t) \, dx \le 2\varepsilon(\Psi,u),
\end{align}
which is \eqref{eq:56401}.

\paragraph{Step 2.}

We will prove that
\begin{align}
\label{eq:step2_bound}
\cdist_p(\calP_u f, \calP_u f_\Phi)
\le 2^{(p-1)/p} \cdot \|\calP_u (|f|)\|_{L^p} \cdot \varepsilon(\Phi,u).
\end{align}

Let $G \in \calA_0$, with derivative $g = G'$ satisfying
$\|g\|_{L^q} \le 1$.
Performing the change of variables $w = \Phi(x,y)$ gives
\begin{align}
\label{eq:401021-0}
\int_{\RR} G(x) (\calP f_\Phi)(x) \, dx
&= \int_{\RR} G(x) \int_{y:(x,y) \in B}
        f(\Phi(x,y)) |\det(\nabla\Phi(x,y))| \,dy \,dx
\nonumber \\
&= \int_{B} G(x)  f(\Phi(x,y)) |\det(\nabla\Phi(x,y))| \,dy \,dx
\nonumber \\
&= \int_{A} G(\psi_1(w)) f(w)  \,dw
\nonumber \\
&= \int_{\RR} \int_{y: (x,y) \in A} G(\psi_1(x,y)) f(x,y)   \, dy \,dx
\end{align}
and similarly,
\begin{align}
\label{eq:401021-1}
\int_{\RR} G(x) (\calP f)(x) \, dx
= \int_{\RR} \int_{y: (x,y) \in A} G(x) f(x,y)   \, dy \,dx.
\end{align}
Combining \eqref{eq:401021-0} and \eqref{eq:401021-1},
and applying H\"{o}lder's inequality,
\begin{align}
\label{eq:3121031}
\int_{\RR} G(x) ((\calP f)(x) - (\calP f_\Phi)(x)) \, dx
&= \int \int [G(x) - G(\psi_1(x,y))]  f(x,y) \, dy \, dx
\nonumber \\
&\le \int \left(\sup_{y} |G(x) - G(\psi_1(x,y))| \right) 
    \left(\int_{\RR^{d-1}} |f(x,y)| \, dy \right) dx
\nonumber \\
&= \int \left(\sup_{y} |G(x) - G(\psi_1(x,y))| \right)  \calP(|f|)(x)  dx
\nonumber \\
&\le \|\calP(|f|)\|_{L^p}
    \cdot \left\| \sup_{y} |G(x) - G(\psi_1(x,y))| \right\|_{L^q(dx)},
\end{align}
where $1/p + 1/q = 1$.

We will show that
\begin{align}
\label{eq:50401011}
\left\| \sup_{y} |G(x) - G(\psi_1(x,y))| \right\|_{L^q(dx)}
\le 2^{(p-1)/p} \, \varepsilon(\Psi,u),
\end{align}
which yields \eqref{eq:step2_bound} by applying Proposition \ref{prop:variational}.
We have
\begin{align}
|G(x) - G(\psi_1(x,y))| = \left| \int_{x}^{\psi_1(x,y)} g(t) dt \right|
= \left| \int g(t) \chi(x,y,t) dt \right|.
\end{align}

Suppose temporarily that $1 < p,q < \infty$.
Applying H\"older's inequality and using
\eqref{eq:56400} and \eqref{eq:56401},
we get
\begin{align}
\int \left(\sup_{y} |G(x) - G(\psi_1(x,y))| \right)^{q} dx
&= \int \sup_{y} |G(x) - G(\psi_1(x,y))|^{q} dx
\nonumber \\
&= \int \sup_{y} \left| \int g(t) \chi(x,y,t) dt \right|^{q} dx
\nonumber \\
&\le \int  \sup_{y} \left(\int |g(t)|^q \chi(x,y,t) dt \right)
        \left( \int \chi(x,y,t) \, dt \right)^{q/p} \,dx
\nonumber \\
&\le \varepsilon(\Psi,u)^{q/p} \int  \sup_{y} \int |g(t)|^q \chi(x,y,t) \, dt \,dx
\nonumber \\
&\le \varepsilon(\Psi,u)^{q/p} \int \int |g(t)|^q \sup_{y}  \chi(x,y,t) \, dt \,dx
\nonumber \\
&= \varepsilon(\Psi,u)^{q/p} \int |g(t)|^q \int  \sup_{y}  \chi(x,y,t) \,dx \, dt
\nonumber \\
&\le 2 \varepsilon(\Psi,u)^{q/p+1} \int |g(t)|^q \, dt,
\end{align}
and so taking the $q$-th root we get the bound
\begin{align}
\left\| \sup_{y} |G(x) - G(\psi_1(x,y))| \right\|_{L^q(dx)}
\le 2^{1/q} \varepsilon(\Psi,u)^{1/p + 1/q} \|g\|_{L^q}
\le 2^{(p-1)/p} \varepsilon(\Psi,u).
\end{align}

Now suppose $p=1$ and  $q=\infty$.
Then from \eqref{eq:56400},
\begin{align}
\left\| \sup_{y} |G(x) - G(\psi_1(x,y))| \right\|_{L^\infty(dx)}
&=\sup_{x,y} |G(x) - G(\psi_1(x,y))|
\nonumber \\
&= \sup_{x,y}\left| \int g(t) \chi(x,y,t) dt \right|
\nonumber \\
&\le \|g\|_{L^\infty} \sup_{x,y} \int \chi(x,y,t) \, dt
\nonumber \\
&\le \|g\|_{L^\infty} \varepsilon(\Psi,u)
\nonumber \\
&\le \varepsilon(\Psi,u).
\end{align}

Finally, suppose $p=\infty$ and $q=1$.
Then from \eqref{eq:56401},
\begin{align}
\left\| \sup_{y} |G(x) - G(\psi_1(x,y))| \right\|_{L^1(dx)}
&= \int \sup_{y} |G(x) - G(\psi_1(x,y))| \, dx
\nonumber \\
&= \int \sup_{y}\left| \int g(t) \chi(x,y,t) dt \right| dx
\nonumber \\
&\le \int \int |g(t)| \sup_{y} \chi(x,y,t) \, dt \, dx
\nonumber \\
&= \int |g(t)| \int \sup_{y} \chi(x,y,t) \, dx \, dt
\nonumber \\
&\le \|g\|_{L^1} \sup_t \int \sup_{y} \chi(x,y,t) \,dx
\nonumber \\
&\le \|g\|_{L^1} 2 \varepsilon(\Psi,u)
\nonumber \\
&= 2 \varepsilon(\Psi,u).
\end{align}

This completes the proof of \eqref{eq:50401011},
and hence proves \eqref{eq:step2_bound}.

\paragraph{Step 3.}
To conclude the proof, we will prove the inequality
\begin{align}
\cdist_p(\calP_u f, \calP_u f_\Phi)
\le \|f\|_{L^1} \cdot \varepsilon(\Psi,u)^{1/p}.
\end{align}

Let $G \in \calA_0$, with derivative $g = G'$ satisfying
$\|g\|_{L^q} \le 1$.
From \eqref{eq:3121031}, taking $p=1$ and $q=\infty$,
\begin{align}
\int_{\RR} G(x) ((\calP f)(x) - (\calP f_\Phi)(x)) \, dx
\le \|f\|_{L^1} \sup_{x,y}|G(x) - G(\psi_1(x,y))|,
\end{align}
and so by using Proposition \ref{prop:variational},
it is enough to show that for all $(x,y)$,
\begin{align}
|G(x) - G(\psi_1(x,y))| \le \varepsilon(\Psi,u)^{1/p}.
\end{align}

H\"{o}lder's inequality yields
\begin{align}
|G(x) - G(\psi_1(x,y))|
&= \left| \int g(t) \chi(x,y,t) dt \right|
\nonumber \\
&\le \|g\|_{L^q} \left(\int \chi(x,y,t)^p \, dt \right)^{1/p}
\nonumber \\
&\le \left(\int \chi(x,y,t) \, dt \right)^{1/p}
\nonumber \\
&\le \varepsilon(\Psi,u)^{1/p},
\end{align}
where the last inequality follows from \eqref{eq:56400}.
This completes the proof.

\end{proof}

Next, we will consider special cases for which quantitatively tighter bounds can be shown.

\subsection{Sharper bounds in special cases}
\label{section:sharper}

In this section, we consider
specific classes of deformations, and
prove sharper bounds
than \eqref{eq:main1} and \eqref{eq:main2} from Theorem \ref{thm:main_deformations}.

\subsubsection{Rotations in 2D}

We consider the case where $A = B = \DD \subset \RR^2$, the open unit disc
centered at $(0,0)$.
If $\Phi$ is a rotation around the origin by angle $\theta$,
then the corresponding maximum displacement is $\varepsilon_\infty(\Phi) = 2 \sin(\theta/2)$,
and so the bound \eqref{eq:main1} in Theorem \ref{thm:main_deformations} is
\begin{align}
\SC_{\eta,p}(f,f_\Phi) \le 
2^{(p-1)/p} \cdot \|f\|_{M_\eta^p} \cdot 2 \sin(\theta/2).
\end{align}
(We do not consider the bound \eqref{eq:main2}, as for this
choice of $\Phi$ it is never stronger
than \eqref{eq:main1}.)

We can prove a sharper estimate:

\begin{thm}
\label{thm:sliced_rotations2D}
Let $1 \le p \le \infty$,
and $f: \RR^2 \to \RR$ be in $L^p(\DD)$.
Suppose $0 \le \theta < \pi$,
and define $f_\theta$ by
\begin{align}
f_\theta(x,y) = f(x \cos(\theta) + y \sin(\theta), y \cos(\theta) - x \sin(\theta)).
\end{align}
Then for
any probability distribution $\eta$ over $\mathbb{S}^{d-1}$,
\begin{align}
\SC_{\eta,p}(f,f_\theta)
\le \|f\|_{M_\eta^p} \cdot \Delta_p(\theta),
\end{align}
where
\begin{align}
\Delta_p(\theta) =
\begin{cases}
2 \sin(\theta/2) \cdot (2\cos(\theta/2))^{(p-1)/p}, &\text{ if } 0 \le \theta < \pi / 2 \\
2 \sin(\theta/2)^{1/p}, &\text{ if } \pi/2 \le \theta < \pi
\end{cases}.
\end{align}
\end{thm}

The result follows from the following lemma:
\begin{lem}
\label{lem:rotations2D}
Using the notation from the statement of Theorem \ref{thm:sliced_rotations2D},
if $u$ is any unit vector in $\RR^2$, then
\begin{align}
\cdist_{p}(\calP_u f,\calP_u f_\theta)
\le \|\calP_u (|f|) \|_{L^p} \cdot \Delta_p(\theta).
\end{align}
\end{lem}

Theorem \ref{thm:sliced_rotations2D} follows immediately
by taking the $p$-th power and averaging over all $u$.

Before giving a detailed proof of Lemma \ref{lem:rotations2D},
we make some simple observations.
First, to simplify the notation, we first note that we can assume, without
loss of generality, that $u = (1,0)$.
The proof of Lemma \ref{lem:rotations2D}
is identical to that of Lemma \ref{lem:projpert}
after replacing the bound
$\sup_t \int \sup_{y} \chi(x,y,t) dx \le 2\varepsilon(\Psi,u)$
from \eqref{eq:56401}
with the bound
\begin{align}
\int \sup_{y} \chi(x,y,t) dx \le
\begin{cases}
2 \sin(\theta), &\text{ if } 0 \le \theta < \pi / 2 \\
2, &\text{ if } \pi/2 \le \theta < \pi  
\end{cases}
\end{align}
for all $|t| < 1$. Indeed, when $0 \le \theta < \pi/2$, we can then replace
\eqref{eq:50401011}
with the upper bound
\begin{align}
\left\| \sup_{y} |G(x) - G(\psi_1(x,y))| \right\|_{L^q(dx)}
&\le (2 \sin(\theta/2))^{1/p} \cdot (2 \sin(\theta))^{1/q}
\nonumber \\
&= (2 \sin(\theta/2))^{1/p} \cdot (4 \sin(\theta/2)\cos(\theta/2))^{1/q}
\nonumber \\
&= 2 \sin(\theta/2) \cdot (2\cos(\theta/2))^{1/q}
\nonumber \\
&= 2 \sin(\theta/2) \cdot (2\cos(\theta/2))^{(p-1)/p}
\end{align}
whereas when $\pi/2 \le \theta < \pi$ the bound becomes
\begin{align}
\left\| \sup_{y} |G(x) - G(\psi_1(x,y))| \right\|_{L^q(dx)}
\le (2 \sin(\theta/2))^{1/p} \cdot 2^{1/q}
= 2 \sin(\theta/2)^{1/p}.
\end{align}

The bound $\int \sup_{y} \chi(x,y,t) dx \le 2$
is immediate, since the integrand is bounded by $1$ and
the integral is over $|x| \le 1$. Hence the proof
of Lemma \ref{lem:rotations2D} only requires showing that
\begin{align}
\label{eq:491021}
\int \sup_{y} \chi(x,y,t) dx \le 2 \sin(\theta)
\end{align}
whenever $0 \le \theta < \pi / 2$.
For brevity,
let $c = \cos(\theta)$ and $s = \sin(\theta)$,
and note that in this case, $c \ge 0$ and $s \ge 0$,
and the rotation $\Phi(x,y) = (cx + sy, cy - sx)$,
with inverse $\Psi(x,y) = (cx - sy, cy + sx)$.
Taking $|t| < 1$, and supposing, without loss of generality,
that $0 \le t \le 1$, it is enough to show
\begin{align}
\label{eq:bound55011}
\int \sup_{y} \chi(x,y,t) \, dx \le
\begin{cases}
2s\sqrt{1-t^2}, &\text{ if } 0 \le t < c \\
1 - ct + s\sqrt{1-t^2}, &\text{ if } c \le t \le 1
\end{cases}.
\end{align}
Indeed, the right side of \eqref{eq:bound55011} is decreasing in $t$
and so is maximized when $t=0$,
which yields the desired bound \eqref{eq:491021}.

With these observations, the proof of Lemma \ref{lem:rotations2D}
(and hence of Theorem \ref{thm:sliced_rotations2D})
then reduces to showing \eqref{eq:bound55011}.
To that end, we introduce an additional piece of notation:
for $0 \le t \le 1$,
let $S_t$ denote the set of all $x$, $|x| \le 1$, satisfying $\sup_y \chi(x,y,t) = 1$.
Observe that $\int \sup_{y} \chi(x,y,t) \, dx = |S_t|$.
We now prove several technical bounds, which will in turn allow
us to easily complete
the proof of Lemma \ref{lem:rotations2D}.

\begin{lem}
\label{lem:50403001}
Let $0 \le t \le 1$.

\begin{enumerate}

\item Suppose $t \le x$. Then $x \in S_t$ if and only if
$cx - s\sqrt{1 - x^2} \le t$.

\item Suppose $x \le t$. Then $x \in S_t$ if and only if
$t \le cx + s\sqrt{1 - x^2}$.

\end{enumerate}

\end{lem}

\begin{proof}
First suppose that $t \le x$.
If $x \in S_t$, then there exists $y$ with $\chi(x,y,t) = 1$, that is,
for which $(x,y) \in \DD$ and $\psi_1(x,y) = cx - sy \le t \le x$.
Since $cx - sy$ only gets smaller
as $y$ grows, we can always take $y = \sqrt{1 - x^2}$. The converse is immediate.

Next suppose that $x \le t$.
If $x \in S_t$, then there exists $y$ for which $(x,y) \in \DD$
and $x \le t \le cx - sy$.
Since $cx - sy$ only gets bigger
as $y$ shrinks, we can always take $y = -\sqrt{1 - x^2}$.
Again, the converse is immediate.

\end{proof}

\begin{lem}
\label{lem:50403002}
Let $0 \le t \le 1$. Then for all $x \in S_t$,  $x \ge ct - s\sqrt{1 - t^2}$.
\end{lem}

\begin{proof}
We will break this into two cases, depending on whether $x \le t$
or $x > t$.

\textbf{Case 1.}  $x \le t$.
By Lemma \ref{lem:50403001}, $cx + s \sqrt{1 - x^2} \ge t$.
Therefore,
 $s \sqrt{1 - x^2} \ge t - cx$
and since we assume $x \le t$, $t-cx \ge 0$;
therefore, squaring each side gives
\begin{align}
s^2 - s^2 x^2 \ge t^2 + c^2 x^2 - 2ctx,
\end{align}
and hence
\begin{align}
x^2 - 2ctx + t^2 - s^2 \le 0,
\end{align}
and since the roots of the polynomial are $x = ct \pm s\sqrt{1 - t^2}$,
in particular it follows that $x \ge ct - s \sqrt{1 - t^2}$.

\textbf{Case 2.} $x > t$.
In this case, the result follows immediately from the inequality
$ct - s\sqrt{1 - t^2} \le ct \le t < x$.

\end{proof}

\begin{lem}
\label{lem:50403003}
Suppose that $0 \le t \le c$. Then all $x \in S_t$ satisfy $x \le ct + s \sqrt{1-t^2}$.
\end{lem}

\begin{proof}
First, note that the assumption $0 \le t \le c$ implies that $t \le ct + s \sqrt{1-t^2}$.
Indeed, the latter inequality is equivalent to $t(1-c) \le s \sqrt{1-t^2}$,
and which is equivalent to $t^2 + t^2 c^2 - 2ct^2 \le s^2 - t^2 s^2$,
which is equivalent to $2t^2(1 - c)\le s^2 = 1-c^2 = (1-c)(1+c)$,
which is equivalent to $2t^2 \le 1+c$,
which is immediate from $0 \le t \le c \le 1$.

Now, let $x \in S_t$. If $x \le t$, then $x \le ct + s \sqrt{1-t^2}$
as well, by what we have just shown.
Therefore, we can suppose instead that $x > t$.
By Lemma \ref{lem:50403001}, $cx - s \sqrt{1 - x^2} \le t$.
Note that $cx - s \sqrt{1 - x^2}$ is an increasing function of $x$.
Hence, if $x > t$ satisfies
\begin{align}
\label{eq:max_St}
cx - s\sqrt{1 - x^2} = t,
\end{align}
then for all $x' > x$, we have $x' > t$ and $cx' - s\sqrt{1 - (x')^2} > t$,
and hence by Lemma \ref{lem:50403001} $x' \notin S_t$.
For $x > t$ satisfying \eqref{eq:max_St}, we have
\begin{align}
& cx - t = s \sqrt{1 - x^2}
\nonumber \\
\implies \, & c^2 x^2 + t^2 - 2ctx = s^2 - s^2 x^2
\nonumber \\
\implies \, & x^2 - 2ctx + t^2 - s^2 = 0
\nonumber \\
\implies \, & x = ct \pm s \sqrt{1-t^2}
\nonumber \\
\implies \, & x = ct + s \sqrt{1-t^2},
\end{align}
where the last implication is because $x > t$.
This completes the proof.

\end{proof}

With the bounds from Lemmas
\ref{lem:50403002} and \ref{lem:50403003},
the proof of Lemma \ref{lem:rotations2D}
is now straightforward:

\begin{proof}[Proof of Lemma \ref{lem:rotations2D}]

From Lemmas \ref{lem:50403002} and \ref{lem:50403003},
when $0 \le t < c$, all $x \in S_t$ lie
between $ct - s\sqrt{1 - t^2}$ and $ct + s\sqrt{1 - t^2}$;
hence $|S_t| \le 2 s \sqrt{1-t^2}$.
On the other hand, when $t \ge c$,
by Lemma \ref{lem:50403002}
all $x \in S_t$ must lie
between $ct - s\sqrt{1 - t^2}$ and $1$;
hence $|S_t| \le 1 - ct + s \sqrt{1-t^2}$.
This concludes the proof of \eqref{eq:bound55011},
and hence of Lemma \ref{lem:rotations2D}.

\end{proof}

\subsubsection{Monotonically increasing deformations in 1D}

When $d=1$, the bounds \eqref{eq:main1} and \eqref{eq:main2}
in Theorem \ref{thm:main_deformations} become
\begin{align}
\cdist_p(f,f_\Phi) \le 2^{(p-1)/p} \cdot \|f\|_{L^p} \cdot \varepsilon_\infty(\Phi).
\end{align}
Now we show that the factor of $2^{(p-1)/p}$
can be removed
when the deformation $\Phi$
is monotonically increasing.

\begin{thm}
\label{thm:deformation_1d}

Let $1 \le p \le \infty$.
Let $I$ and $J$ be non-empty, bounded, open intervals in $\RR$,
$f$ be in $L^p(I)$,
$\Phi : J \to I$ be a $C^1$ bijection with inverse $\Psi$
and $\Phi'(x) > 0$ on $J$,
and
\begin{math}
f_\Phi(x) = f(\Phi(x)) \Phi'(x))
\end{math}
on $J$, and $0$ elsewhere.
Then
\begin{align}
\label{eq:4012}
\cdist_{p}(f,f_\Phi) \le 
\|f\|_{L^p} \cdot \varepsilon_\infty(\Phi).
\end{align}
\end{thm}

\begin{rmk}
To see that the monotonocity of $\Phi$ is required
for this sharper bound, consider the following example.
Fix $\eta > \delta > 0$, and let $f$ be defined by
\begin{align}
f(x)=
\begin{cases}
1 , &\text{ if } -\eta \le x < 0, \\
-1 , &\text{ if } 0 \le x \le \eta, \\
0 , &\text{ otherwise}.
\end{cases}
\end{align}
Let $\Phi : [-\delta,\delta] \to [-\eta,\eta]$
be defined by $\Phi(x) = -(\eta/\delta) x$.
Then
\begin{align}
f_\Phi(x)=
\begin{cases}
-\eta/\delta , &\text{ if } -\delta \le x < 0, \\
\eta/\delta , &\text{ if } 0 \le x \le \delta, \\
0 , &\text{ otherwise}.
\end{cases}
\end{align}
Then it is straightforward to verify that $\varepsilon_\infty(\Phi) = \eta + \delta$,
$\|f\|_{L^\infty} = 1$, and
$C_\infty(f,f_\Phi) = 2\eta$.
By taking $\delta \to 0$, we see that the bound
$C_\infty(f,f_\Phi) \le 2 \|f\|_{L^\infty} \varepsilon_\infty(\Phi)$
is tight.

\end{rmk}

\begin{proof}[Proof of Theorem \ref{thm:deformation_1d}]

Let $I_x$ be the interval $[x,\Psi(x)]$ if $x \le \Psi(x)$, and
$[\Psi(x),x]$ if $\Psi(x) \le x$.
Let $\chi(x,t)$ be $1$ if $t \in I_x$, and $0$ otherwise.
Then an identical proof to that of Lemma \ref{lem:projpert} may be applied
if we show that
\begin{align}
\sup_{t} \int \chi(x,t) dx \le \varepsilon_\infty(\Phi),
\end{align}
in place of the bound \eqref{eq:56401}.

Take any $t \in I$, and  suppose that there is some
$x \le t$ with $t \in I_x$; note that for such $x$, $I_x = [x,\Psi(x)]$,
and so $x \le \Psi(x)$. Let $x^*$ be the smallest such $x$.
Then $x^* \le t \le \Psi(x^*)$. We claim that
for all $x > t$, $t \notin I_x$.
Indeed, since $\Psi$ is increasing and $x > t \ge x^*$,
we have $\Psi(x) > \Psi(x^*) \ge t$.
Since both $x > t$ and $\Psi(x) > t$, $t$ does not lie in $I_x$, as claimed.

Consequently, all $x$ for which $t$ lies in $I_x$
are contained inside the interval $[x^*,t]$.
Since $x^* \le t \le \Psi(x^*)$ and $|x^* - \Psi(x^*)| \le \varepsilon_\infty(\Phi)$,
it follows that $|t - x^*| \le \varepsilon_\infty(\Phi)$ too. 
Furthermore, if $x > t$, then $\chi(x,t) = 0$ since $t \notin I_x$;
and since $x^*$ is the smallest $x$ for which $t \in I_x$,
if $x < x^*$ then $t \notin I_x$, hence $\chi(x,t) = 0$.
Therefore,
\begin{align}
\int \chi(x,t) dx \le \int_{x^*}^{t} 1 dx = |t-x^*| \le \varepsilon_\infty(\Phi).
\end{align}

Analogous reasoning yields the same bound in the case that
there exists $x \ge t$ with $t \in I_x$.
This completes the proof.

\end{proof}

\subsubsection{Translations}

We consider the case where $\Phi(x) = x + v$,
where $v$ is a fixed vector. For simplicity, we
only describe the case where $\eta$ is the uniform measure,
though the results easily generalize.
In this case, the bounds \eqref{eq:main1}
and \eqref{eq:main2}
from Theorem \ref{thm:main_deformations}
are, respectively,
\begin{align}
\SC_p(f,f_\Phi) \le 2^{(p-1)/p} \cdot \|f\|_{M^{p}} \cdot |v|
\end{align}
and
\begin{align}
\SC_p(f,f_\Phi) \le 2^{(p-1)/p} \cdot K_p \cdot \|f\|_{M^{p,\infty}} \cdot |v|,
\end{align}
where
\begin{align}
\label{eq:kp_defined}
K_p = \left( \int_{\mathbb{S}^{d-1}} |u_1|^p \, du \right)^{1/p}.
\end{align}

We show that the factor $2^{(p-1)/p}$ can be removed:

\begin{thm}
\label{thm:translations}
Let $1 \le p \le \infty$.
Let $A$ be a non-empty, bounded, open set in $\RR^d$, and $f$ be in $L^p(A)$.
For a fixed vector $v$,
let
\begin{math}
f_v(x) = f(x+v).
\end{math}
Then
\begin{align}
\label{eq:500504}
\SC_{p}(f,f_v)
\le \|f\|_{M^{p}} \cdot | v |
\end{align}
and
\begin{align}
\label{eq:500505}
\SC_{p}(f,f_v)
\le K_p \cdot \|f\|_{M^{p,\infty}} \cdot | v |,
\end{align}
where $K_p$ is defined in \eqref{eq:kp_defined}.
\end{thm}

\begin{rmk}
When $p=1$, $\SC_1(f,f_v) = \SW_1(f,f_v)$,
and so the bound \eqref{eq:500505} when $p=1$ and $d=2$
matches the $p=1$ case of Theorem 2 of \cite{shi2025fast}.
\end{rmk}

\begin{proof}[Proof of Theorem \ref{thm:translations}]
By a straightforward calculation, for any $u \in \mathbb{S}^{d-1}$,
\begin{math}
(\calP_u f_v)(t) = (\calP_u f)(t + \langle v,u\rangle).
\end{math}
Consequently, by Theorem \ref{thm:deformation_1d},
\begin{align}
\cdist_p(\calP_u f, \calP_u f_v) \le \|\calP_u f\|_{L^p} \cdot |\langle v,u\rangle|.
\end{align}
Then
\begin{align}
\SC_p(f,f_v)^p
\le \int_{\mathbb{S}^{d-1}} \|\calP_u f\|_{L^p}^p \cdot |\langle v,u\rangle|^p \, du
\le |v|^p \cdot \int_{\mathbb{S}^{d-1}} \|\calP_u f\|_{L^p}^p  \, du
\le \|f\|_{M^{p}} \cdot |v|^p,
\end{align}
which is \eqref{eq:500504}; and
\begin{align}
\SC_p(f,f_v)^p \le \sup_{u \in \mathbb{S}^{d-1}} \|\calP_u f\|_{L^p}^p
    \cdot \int_{\mathbb{S}^{d-1}} |\langle v,u\rangle|^p \, du
\le K_p \cdot \|f\|_{M^{p,\infty}}^p \cdot |v|^p,
\end{align}
proving \eqref{eq:500505}.

\end{proof}

\subsubsection{Dilations}

Suppose $\BB\subset \RR^{d}$ is the open unit ball in $\RR^d$
centered at $0$.
Let $\Phi(w) = \alpha w$, where $\alpha > 1$. Then $\varepsilon_\infty(\Phi) = \alpha-1$,
and so bound \eqref{eq:main1} from Theorem \ref{thm:main_deformations} is
\begin{align}
\SC_{\eta,p}(f,f_\Phi) \le 
2^{(p-1)/p} \cdot \|f\|_{M_\eta^p} \cdot ( \alpha - 1).
\end{align}
(We do not consider the bound \eqref{eq:main2}, as for this
choice of $\Phi$ it is never stronger
than \eqref{eq:main1}.)

We can prove a sharper estimate:

\begin{thm}
\label{thm:sliced_dilations}
Let $1 \le p \le \infty$,
and let $f$ be in $L^p(\BB)$.
Suppose $\alpha > 1$, and define $f_\alpha$ by
\begin{math}
f_\alpha(w) = \alpha f(\alpha w) .
\end{math}
Then for
any probability distribution $\eta$ over $\mathbb{S}^{d-1}$,
\begin{align}
\SC_{\eta,p}(f,f_\alpha)
\le \|f\|_{M_\eta^p} \cdot \frac{\alpha-1}{\alpha^{(p-1)/p}}.
\end{align}
\end{thm}

The result follows from the following lemma:
\begin{lem}
\label{lem:dilations}
Using the notation from the statement of Theorem \ref{thm:sliced_dilations},
\begin{align}
\cdist_{p}(\calP_u f,\calP_u f_\alpha)
\le \|\calP_u (|f|)\|_{L^p} \cdot \frac{\alpha-1}{\alpha^{(p-1)/p}}.
\end{align}
\end{lem}

Theorem \ref{thm:sliced_rotations2D} follows immediately
by taking the $p$-th power and averaging over all $u$.

\begin{proof}[Proof of Lemma \ref{lem:dilations}]

Without loss of generality, suppose $u = (1,0,\dots,0)$.
It is enough to show that for all $|t| < 1$,
\begin{align}
\int \sup_{y \in \RR^{d-1}} \chi(x,y,t) \, dx \le \frac{\alpha-1}{\alpha};
\end{align}
this estimate can then be used in place of \eqref{eq:56401} in the proof
of Lemma \ref{lem:projpert}.

Without loss of generality, suppose $0 \le t < 1$.
Let $S_t$ denote the set of all $x$, $|x| < 1$,
satisfying $\sup_{y \in \RR^{d-1}} \chi(x,y,t) = 1$.
If $x \in S_t$, then
$x < t < \alpha x$;
since $t \ge 0$, this restricts $x \ge 0$ as well,
and
$S_t = (t/\alpha,t)$, so $|S_t| = t(1 - 1/\alpha)$,
which is maximized at $t=1$; thus
\begin{align}
\int \sup_{y \in \RR^{d-1}} \chi(x,y,t) \, dx \le 1 - 1/\alpha = \frac{\alpha-1}{\alpha},
\end{align}
as claimed.
Using this estimate in place of the bound $\int \sup_{y \in \RR^{d-1}} \chi(x,y,t) \, dx
\le 2 \varepsilon_\infty(\Phi)$ gives the final bound
\begin{align}
\cdist(\calP_u f, \calP_u f_\Phi)
\le \|\calP_u (|f|)\|_{L^p} (\alpha-1)^{1/p} \left(\frac{\alpha-1}{\alpha}\right)^{1/q}
= \|\calP_u (|f|)\|_{L^p} \cdot \frac{(\alpha-1)}{\alpha^{(p-1)/p}},
\end{align}
completing the proof.

\end{proof}

\subsection{Convolutions}
\label{section:convolutions}

In this section, we
remark on the behavior of the sliced Cram\'er distance
after convolution of its input functions.
This situation occurs commonly in signal and image processing,
where one typically observes signals that have been convolved with
a function induced from the 
measurement apparatus.
The bound we state here is analogous to
Theorem \ref{thm:wasserstein_convo} and Corollary \ref{cor:sw_convo}
for Wasserstein and sliced Wasserstein distances,
respectively.
Our result is a slight generalization of the first part of Theorem 5
in the prior work \cite{zhang2023cramer},
where it is stated in terms of random variables
(and hence only applies to non-negative functions)
and specialized to the case $p=2$;
we think it is valuable to explicitly state and prove
the result for general $p$
and reformulate it in
terms of convolution, where its applicability to signal and image processing
may be more evident. Furthermore, our bound  is sharper
when the convolution kernel takes on negative values,
as can occur in scientific imaging applications.

\begin{thm}
\label{thm:convolutions}
Let $1 \le p \le \infty$.
Suppose $f$ and $g$ are in
$L^p(\RR^d)$ and compactly supported,
$w$ is in $L^1(\RR^d)$ and compactly supported,
and $\eta$ is a probability measure over $\mathbb{S}^{d-1}$.
Then
\begin{align}
\label{eq:convolutions}
\SC_{\eta,p}(f \ast w, g \ast w) \le
\left( \int_{\mathbb{S}^{d-1}} \|\calP_u w\|_{L^1} d\eta(u)   \right) \cdot \SC_{\eta,p}(f, g).
\end{align}
\end{thm}

\begin{rmk}
When $w \ge 0$,
the conclusion of Theorem \ref{thm:convolutions}
becomes
\begin{align}
\label{eq:convolutions2}
\SC_{\eta,p}(f \ast w, g \ast w) \le
\|w\|_{L^1}  \cdot \SC_{\eta,p}(f, g),
\end{align}
which is essentially what appears (for $p=2$) in
\cite{zhang2023cramer}.
However, in certain scientific imaging applications,
such as cryo-electron microscopy,
where $w$ is the point-spread function
of the imaging apparatus,
$w$ may take negative values,
in which case the bound \eqref{eq:convolutions}
is sharper than \eqref{eq:convolutions2}.

\end{rmk}

\begin{proof}

We first show the result for $d=1$, that is,
for when
$f$, $g$  and $w$ are functions on $\RR$.
That is, we will show that
\begin{align}
\label{eq:convo_1d}
\cdist_p(f \ast w, g \ast w) \le \|w\|_{L^1} \cdot \cdist_p(f, g).
\end{align}
Let $H$ be in $\calA_0$,
with $\|H'\|_{L^q} = 1$.
For any $s$, $\|H'(u+s)\|_{L^q(du)} = 1$ too.
Then using Proposition \ref{prop:variational},
\begin{align}
\langle f \ast w - g \ast w, H \rangle
&= \int ((f - g)\ast w)(t) H(t) \, dt
\nonumber \\
&= \int \int (f-g)(t-s) w(s) \,ds \, H(t) \, dt
\nonumber \\
&= \int w(s) \int (f-g)(t-s) H(t) \, dt \, ds
\nonumber \\
&= \int w(s) \int (f-g)(u) H(u+s) \,du \, ds
\nonumber \\
&\le \|w\|_{L^1} \sup_{s} \left| \int (f-g)(u) H(u+s) \, du \right|
\nonumber \\
&\le \|w\|_{L^1} \cdist_p(f,g),
\end{align}
and \eqref{eq:convo_1d} now follows by taking the supremum over all such $H$
and invoking Proposition \ref{prop:variational}.

We now turn to general $d \ge 1$.
First, by
Lemma \ref{lem:projection_convolution},
$\calP_u (w \ast h) = (\calP_u w) \ast (\calP_u h)$.
From the 1D bound,  we then have
\begin{align}
\cdist_p(\calP_u (w \ast f), \calP_u (w \ast g))
&= \cdist_p((\calP_u w) \ast (\calP_u f), (\calP_u w) \ast (\calP_u g))
\nonumber \\
&\le \|\calP_u w\|_{L^1} \cdist_p(\calP_u f, \calP_u g)
\end{align}

Taking $p$-th powers and integrating over $u$ gives the result.

\end{proof}

\section{Discretizations and robustness to noise}
\label{section:discrete}

In this section, we will describe Fourier-based discretizations
of the Cram\'er distance and the 2D sliced Cram\'er distance
with respect to the uniform measure over $\mathbb{S}^1$,
between functions with equal integrals,
and analyze their robustness to additive, heteroscedastic sub-Gaussian noise.
More precisely, we will show that, given vectors of noisy samples from
two smooth functions, the discrete distance
approximates the distances between the smooth functions only, removing
the effect of the noise as the number of samples
grows.
These results are to be expected;
indeed, the Cram\'er distance itself involves applying a smoothing
filter to each input, which, by averaging the samples, naturally has a denoising effect.
The denoising property is also of interest
in contrasting with
Wasserstein and sliced Wasserstein distances, which,
because they are defined between probability measures, do not naturally induce
distances between vectors sampled from a signal-plus-noise model
without modification of the definition.

\subsection{Robustness to noise in 1D}
We define a discrete approximation to the 1D Volterra norm,
which then yields an approximation to the Cram\'er distance.
Let $a < b$ and let $L = b-a$ be the interval length.
Let $n$ a positive integer; we will assume
for simplicity that $n$ is even.
Let $x \in \RR^n$;
the reader should think of the entries $x[j]$ of $x$
as (possibly noisy) samples from a function $f$ on $[a,b]$,
satisfying $\int_{a}^{b} f = 0$;
that is, $x[j] \approx f(t_j)$, $j=0,\dots,n-1$,
where $t_j = a + jL / n$.

Define the values $\alpha[k]$ (the normalized discrete Fourier coefficients
of $x$) by
\begin{align}
\label{eq:alpha_1d}
\alpha[k] = \frac{L}{n} \sum_{j=0}^{n-1} x[j] e^{-2 \pi \mi k t_j / L},
\quad 0 < |k| < n/2.
\end{align}
Then
\begin{math}
\what f(k/L) \approx \alpha[k].
\end{math}
(Note that, since we assume $\int_{a}^{b} f = 0$, we
could define $\alpha[0] =0$, though this value
would not be used.)
Define
\begin{align}
\beta[k] =
\frac{\alpha[k]}{2\pi \mi k/L},
\quad  0 < |k| < n/2,
\end{align}
and
\begin{align}
\beta[0] =
- \sum_{0 < |k| < n/2} \beta[k]  e^{2 \pi \mi k a /L}.
\end{align}
Then the $\beta[k]$ approximate the Fourier coefficients of $\calV f$:
\begin{math}
\what{(\calV f)}(k/L) \approx \beta[k].
\end{math}

Define the function $\nu_x(t)$ by
\begin{align}
\label{eq:nu_x_1d}
\nu_x(t) = \frac{1}{L} \sum_{k=-n/2+1}^{n/2-1} \beta[k] e^{2 \pi \mi t k /L}.
\end{align}
(Note that we do not define $\beta[\pm n/2]$,
because the terms they would contribute to $\nu_x(t)$
would either be purely imaginary or $0$, depending
on the convention; this is also why we do not make use of $\alpha[\pm n/2]$.)
Then for all $t$, $\nu_x(t) \approx (\calV f)(t)$.

We then define the discrete Volterra $p$-norm of the vector $x$ as follows:
\begin{align}
\label{eq:discrete_volterra}
\mathrm{V}_p(x) = \left(\frac{L}{n}\sum_{j=0}^{n-1} |\nu_x(t_j)|^p \right)^{1/p}
\end{align}
when $1 \le p < \infty$, and
\begin{align}
\mathrm{V}_\infty(x) = \max_{0 \le j \le n-1} |\nu_x(t_j)|
\end{align}
when $p=\infty$.
Given two vectors $x$ and $y$ in $\RR^n$, we then define
their discrete Cram\'er distance as
\begin{align}
\what\cdist_p(x,y) = \mathrm{V}_p(x-y).
\end{align}

\begin{rmk}
Using the Fast Fourier Transform (FFT)
and the inverse FFT (IFFT)
to evaluate the $\alpha[k]$ and $\nu_x(t_j)$,
respectively, the entire computation described here can be performed
at cost $O(n \log n)$.
\end{rmk}

We can now state the main result from this section,
which says that the discrete Cram\'er distance
is robust to additive heteroscedastic sub-Gaussian noise:

\begin{thm}
\label{thm:discrete_main1d}
Suppose $f$ and $g$ are $C^r$ functions on $\RR$, $r > 1$, that are supported on $[a,b]$,
and satisfy
$\int_{a}^{b} f = \int_{a}^{b} g$.
Let $Z[0],Z[1],\dots,Z[n-1]$, $\wtilde Z[0],\dots,\wtilde Z[n-1]$
be independent, mean-zero sub-Gaussian
random variables with sub-Gaussian norms $\sigma_j = \|Z[j]\|_{\psi_2}$
and $\wtilde \sigma_j = \|\wtilde Z[j]\|_{\psi_2}$;
and suppose $\sigma > 0$ satisfies
\begin{align}
\frac{1}{n}\sum_{j=0}^{n-1} \sigma_j^2
    + \frac{1}{n}\sum_{j=0}^{n-1} \wtilde \sigma_j^2 \le \sigma^2,
\end{align}
for all $n$.
Let $X_n$ and $Y_n$ be vectors in $\RR^n$
with entries $X_n[j] = f(t_j) + Z[j]$ and $Y_n[j] = g(t_j) + \wtilde Z[j]$.
Then:

\begin{enumerate}

\item
Expected error.
For $1 \le p < \infty$,
there are $A > 0$ and $B = B(p,f,L,r) > 0$
such that for
all $\sigma \ge 0$
and $n \ge 2$,
\begin{align}
\label{eq:expectation_1d}
\EE\big|\what{\cdist}_p(X_n, Y_n)   - \cdist_p(f,g) \big|
    \le A L^{1+1/p} \sqrt{p} \frac{\sigma }{\sqrt{n}}
        + B \max\left\{\frac{\log(n)}{n^r}, \frac{1}{n^2} \right\}.
\end{align}
Furthermore, there are $A > 0$ and $B = B(f,L,r) > 0$
such that for
all $\sigma \ge 0$
and $n \ge 2$,
\begin{align}
\label{eq:expectation_1d_infty}
\EE\big|\what{\cdist}_\infty(X_n, Y_n)   - \cdist_\infty(f,g) \big|
    \le A L \sigma \sqrt{\frac{\log(n)}{n}}
        + B \max\left\{\frac{\log(n)}{n^r}, \frac{1}{n^2} \right\}.
\end{align}

\item
Concentration bound.
For $1 \le p < \infty$,
there is a value $C = C(p,f,L,r) > 0$
such that for all $\sigma > 0$,
$t \ge 0$,
and $n \ge 2$,
\begin{align}
\Prob\left\{ \big|\what{\cdist}_p(X_n, Y_n)   - \cdist_p(f,g) \big| \ge t   \right\}
\le 2  \exp\left(- C \frac{n t^2}{\sigma^2} \right).
\end{align}
Furthermore, there is a value $C = C(f,L,r) > 0$
such that for all $\sigma > 0$,
$t \ge 0$,
and $n \ge 2$,
\begin{align}
\Prob\left\{ \big|\what{\cdist}_\infty(X_n, Y_n)   - \cdist_\infty(f,g) \big| \ge t   \right\}
\le 2 \exp\left(- C \frac{n t^2}{\sigma^2 \log(n)} \right).
\end{align}

\item 
Almost sure limit.
For all $1 \le p \le \infty$,
\begin{math}
\what{\cdist}_p(X_n, Y_n) \to \cdist_p(f,g)
\end{math}
almost surely as $n \to \infty$.

\end{enumerate}

\end{thm}

\begin{rmk}
It is straightforward to extend the definition of $\what{\cdist}_p(x,y)$,
and the results from Theorem \ref{thm:discrete_main1d},
to the setting where $x$ and $y$ contain samples of $f$ and $g$
taken on different grids, by interpolating the estimated Volterra transforms
onto a common grid.
\end{rmk}

\begin{rmk}
\label{rmk:constants}
Throughout the proof of Theorem \ref{thm:discrete_main1d}
and the auxiliary results on which it depends,
we will use ``$C$'' to denote a generic
positive number that may change
from instance to instance;
while we do indicate what parameters
$C$ depends on in each case, we do not
provide bounds on its value.
\end{rmk}

Theorem \ref{thm:discrete_main1d} is a straightforward corollary
of the following two results:

\begin{prop}
\label{prop:discrete_volterra_1d}
Suppose $f$ is a $C^r$ function on $\RR$, $r > 1$, that is supported on $[a,b]$
and satisfies
$\int_{a}^{b} f = 0$.
Let $x \in \RR^{n}$ have entries $x[j] = f(t_j)$, $0 \le j \le n-1$.
Then for all $1 \le p \le \infty$,
there is a value $C = C(p,f,L,r) > 0$ such that
\begin{align}
\left| \mathrm{V}_p(x) - \|f\|_{V^p}\right|
    \le C\max\left\{\frac{\log(n)}{n^r}, \frac{1}{n^2} \right\}
\end{align}
for $n \ge 2$.
\end{prop}

\begin{prop}
\label{prop:noise_1d}
Let $Z[0],Z[1],\dots,Z[n],\dots,$ be independent, mean-zero sub-Gaussian
random variables with sub-Gaussian norms $\sigma_j = \|Z[j]\|_{\psi_2}$;
and suppose $\sigma > 0$ satisfies
\begin{align}
\frac{1}{n}\sum_{j=0}^{n-1} \sigma_j^2 \le \sigma^2,
\end{align}
for all $n$.
For each $n \ge 2$, let $Z_n = (Z[0],\dots,Z[n-1])$.
Then:

\begin{enumerate}

\item
Expectation bound.
There is a universal constant $C>0$ such that for
all $1 \le p < \infty$,
$\sigma \ge 0$,
and
$n \ge 2$,
\begin{align}
\EE\left[ \mathrm{V}_p(Z_n) \right] \le C L^{1 + 1/p} \sqrt{p} \, \frac{\sigma}{\sqrt{n}},
\end{align}
and
\begin{align}
\EE\left[ \mathrm{V}_\infty(Z_n) \right] \le C L \sigma \sqrt{\frac{\log(n)}{n}}.
\end{align}

\item
Concentration bound.
There is a universal constant $C > 0$
such that for all $1 \le p < \infty$,
$\sigma > 0$,
$t \ge 0$,
and $n \ge 2$,
\begin{align}
\Prob\left\{ \mathrm{V}_p(Z_n) \ge t   \right\}
\le 2  \exp\left(- C \frac{n t^2}{p L^{2+2/p}\sigma^2} \right),
\end{align}
and
\begin{align}
\Prob\left\{ \mathrm{V}_\infty(Z_n) \ge t \right\}
\le 2 \exp\left(- C \frac{n t^2}{L^2 \sigma^2 \log(n)} \right).
\end{align}

\item 
Almost sure limit.
For all $1 \le p \le \infty$,
\begin{math}
\mathrm{V}_p(Z_n) \to 0
\end{math}
almost surely as $n \to \infty$.

\end{enumerate}
\end{prop}

\begin{rmk}
When $p$ is an even integer,
the bound in Proposition \ref{prop:discrete_volterra_1d}
can be sharpened to $O(\log(n) / n^r)$.
\end{rmk}

We now show that Theorem \ref{thm:discrete_main1d}
follows from Propositions \ref{prop:discrete_volterra_1d}
and \ref{prop:noise_1d}:

\begin{proof}[Proof of Theorem \ref{thm:discrete_main1d}]

It is enough to show the analogous results for the Volterra
norm of $f$ alone; we can then replace $f$ by $f-g$.
Let $x[j] = f(t_j)$, $0 \le j \le n-1$, so that $X_n = x + Z_n$.
We have
\begin{align}
\mathrm{V}_p(X_n) - \|f\|_{V^p}
= \mathrm{V}_p(x + Z_n) - \|f\|_{V^p}
\le \mathrm{V}_p(x) - \|f\|_{V^p} + \mathrm{V}_p(Z_n),
\end{align}
and
\begin{align}
\mathrm{V}_p(X_n) - \|f\|_{V^p}
= \mathrm{V}_p(x + Z_n) - \|f\|_{V^p}
\ge \mathrm{V}_p(x) - \|f\|_{V^p} - \mathrm{V}_p(Z_n),
\end{align}
and therefore
\begin{align}
\left| \mathrm{V}_p(X_n) - \|f\|_{V^p} \right|
\le \left| \mathrm{V}_p(x) - \|f\|_{V^p} \right| + \mathrm{V}_p(Z_n).
\end{align}
The expected error bounds \eqref{eq:expectation_1d}
and \eqref{eq:expectation_1d_infty}
are then immediate,
as is almost sure convergence. To show concentration:
suppose first that $1 \le p < \infty$.
Then there is a constant $C = C(p,L)$
such that for all $t \ge 0$ and $n \ge 2$,
\begin{align}
\Prob\left\{ \mathrm{V}_p(Z_n) \ge t   \right\}
\le 2  \exp\left(- C \frac{n t^2}{\sigma^2} \right),
\end{align}
By Proposition \ref{prop:discrete_volterra_1d},
by replacing $C$ with a possibly
smaller value $C^* = C^*(p,f,L,r)$,
we can ensure that
for all $n \ge 2$,
\begin{align}
|\mathrm{V}_p(x) - \|f\|_{V^p}| \le \sqrt{\frac{\log(2) \sigma^2}{C^* n}},
\end{align}
and therefore, for this same choice of $C^*$,
by defining $\eta = \log(2)/C^*$, we have
\begin{align}
\left| \mathrm{V}_p(X_n) - \|f\|_{V^p} \right|
\le \sqrt{\eta  \frac{\sigma^2 }{ n}} + \mathrm{V}_p(Z_n).
\end{align}
Now, given $t \ge 0$, we consider two cases.
First, if $n$ is large enough so that
\begin{math}
\sqrt{\eta \, \sigma^2  /  n} \le t/2,
\end{math}
then
\begin{align}
\Prob \left\{\left| \mathrm{V}_p(X_n) - \|f\|_{V^p} \right| \ge t \right\}
& \le \Prob\left\{\mathrm{V}_p(Z_n) \ge t
        - \sqrt{\eta \, \sigma^2  /  n}  \right\}
\nonumber \\
&\le \Prob\left\{\mathrm{V}_p(Z_n) \ge t/2 \right\}
\nonumber \\
&\le 2 \exp\left\{-C^* n(t/2)^2 /  \sigma^2 \right\}
\nonumber \\
&= 2 \exp\left\{-(C^*/4) n t^2 /  \sigma^2 \right\}.
\end{align}
On the other hand, if $\sqrt{\eta \, \sigma^2  / n} \ge t/2$, then
this same bound holds trivially, since
\begin{align}
2 \exp\left(- (C^*/4) \frac{n t^2} {\sigma^2} \right)
=  2 \exp\left(- C^* \frac{n (t/2)^2} {\sigma^2} \right)
\ge  2 \exp\left(- C^* \frac{n (\eta \, \sigma^2  / n)} {\sigma^2} \right)
= 2 \exp\left(- C^* \eta \right)
= 1,
\end{align}
which exceeds $\Prob \left\{\left| \mathrm{V}_p(X_n) - \|f\|_{V^p} \right| \ge t \right\}$.
So, for all $n \ge 2$ and all $t \ge 0$, the bound
\begin{align}
\Prob \left\{\left| \mathrm{V}_p(X_n) - \|f\|_{V^p} \right| \ge t \right\}
\le 2 \exp\left(- (C^*/4) \frac{n t^2} {\sigma^2} \right)
\end{align}
holds, as desired. The proof for $p = \infty$ is similar.

\end{proof}

We now turn to the proofs of Propositions \ref{prop:discrete_volterra_1d}
and \ref{prop:noise_1d}.

\subsubsection{Proof of Proposition \ref{prop:discrete_volterra_1d}}

\begin{lem}
\label{lem:approx_alpha_1d}

Let $f$ satisfy the conditions of Proposition \ref{prop:discrete_volterra_1d},
and let $\alpha$ be defined as in \eqref{eq:alpha_1d}.
Then there is a value $C = C(f,L,r)>0$ such that
for even $n \ge 2$ and $|k| < n/2$,
\begin{align}
\left| \alpha[k] - \what{f}(k/L) \right| \le \frac{C}{n^r}.
\end{align}

\end{lem}

\begin{proof}
This is an immediate consequence of Lemma \ref{lem:fourier_decay_approx}.

\end{proof}

\begin{cor}
\label{cor:voltapprox}
Let $f$ satisfy the conditions of Proposition \ref{prop:discrete_volterra_1d},
and let $\nu_x$ be defined as in \eqref{eq:nu_x_1d}.
Then there is a value $C = C(f,L,r)>0$
such that for $|t| < L$ and $n \ge 2$,
\begin{align}
|\nu_x(t) - (\calV f)(t)| \le C\frac{\log(n)}{n^{r}}.
\end{align}

\end{cor}

\begin{proof}
Applying Lemma \ref{lem:approx_alpha_1d}
we have:
\begin{align}
\left| \sum_{0 < |k| < n/2} \beta[k] e^{2 \pi \mi k t / L}
    - \sum_{0 < |k| < n/2} \what{(\calV f)}(k/L) e^{2 \pi \mi k t / L} \right|
&\le C \sum_{0 < |k| < n/2} \frac{\left| \alpha[k] - \what f(k/L) \right|}{|k|}
\nonumber \\
&\le C \frac{\log(n)}{n^r}.
\end{align}
Since $f$ is $C^r$, $|\what{f}(k/L)| = O(|k|^{-r})$, and therefore
$|\what{(\calV f)}(k/L)| = O(|k|^{-(r+1)})$, and so
the tail may be bounded
\begin{align}
\left| \sum_{ |k| \ge n/2} \what{(\calV f)}(k/L) e^{2 \pi \mi k t / L} \right|
\le C \sum_{ |k| \ge n/2} \frac{1}{|k|^{r+1}}
\le \frac{C}{n^{r}}.
\end{align}

Therefore,
\begin{align}
&\left|\sum_{0 < |k| < n/2} \beta[k]  e^{2 \pi \mi k t /L}
     - \sum_{k \ne 0} \what{(\calV f)}(k/L)  e^{2 \pi \mi k t /L} \right|
\nonumber \\
&\le \left|\sum_{0 < |k| < n/2} (\beta[k] - \what{(\calV f)}(k/L))  e^{2 \pi \mi k t /L} \right|
    + \left|\sum_{|k| \ge n/2} \what{(\calV f)}(k/L)  e^{2 \pi \mi k t /L} \right|
\nonumber \\
&\le \sum_{0 < |k| < n/2} \left|\beta[k] - \what{(\calV f)}(k/L))  \right|
    + \sum_{|k| \ge n/2} \left|\what{(\calV f)}(k/L) \right|
\nonumber \\
&\le C\frac{\log(n)}{n^{r}}.
\end{align}
Taking $t=a$ also shows
\begin{align}
\left| \beta[0] - \what{(\calV f)}(0) \right|
= \left|\sum_{0 < |k| < n/2} \beta[k]  e^{2 \pi \mi k a /L}
     - \sum_{k \ne 0} \what{(\calV f)}(k/L)  e^{2 \pi \mi k a /L} \right|
\le C \frac{\log(n)}{n^{r}},
\end{align}
and consequently
\begin{align}
|\nu_x(t) - (\calV f)(t)|
&\le  \left| \beta[0] - \what{(\calV f)}(0) \right|
+ \left|\sum_{0 < |k| < n/2} \beta[k]  e^{2 \pi \mi k s /L}
     - \sum_{k \ne 0} \what{(\calV f)}(k/L)  e^{2 \pi \mi k s /L} \right|
\nonumber \\
&\le C\frac{\log(n)}{n^r}.
\end{align}

\end{proof}

We now complete the proof of Proposition \ref{prop:discrete_volterra_1d}.
When $1 \le p < \infty$, from the triangle inequality
\begin{align}
\big| \mathrm{V}_p(x) - \|f\|_{V^p} \big|
&= \left| \left(\frac{L}{n}\sum_{j=0}^{n-1} \left|\nu_x(t_j)  \right|^p \right)^{1/p}
    - \|f\|_{V^p}\right|
\nonumber \\
&\le  \left| \left(\frac{L}{n}\sum_{j=0}^{n-1} \left|\nu_x(t_j)  \right|^p \right)^{1/p}
        - \left(\frac{L}{n}\sum_{j=0}^{n-1} \left|(\calV f)(t_j)  \right|^p \right)^{1/p}\right| 
    + \left|\left(\frac{L}{n}\sum_{j=0}^{n-1} \left|(\calV f)(t_j)  \right|^p \right)^{1/p}
        - \|f\|_{V^p} \right|
\nonumber \\
&\le  \left(\frac{L}{n}\sum_{j=0}^{n-1} \left|\nu_x(t_j) - (\calV f)(t_j)  \right|^p \right)^{1/p}
    + \left|\left(\frac{L}{n}\sum_{j=0}^{n-1} \left|(\calV f)(t_j)  \right|^p \right)^{1/p}
        - \|\calV f\|_{L^p} \right|,
\end{align}
with obvious modifications when $p=\infty$.
From Corollary \ref{cor:voltapprox},
the first term is $O(\log(n)/n^r)$,
with constant depending on $f$, $L$, and $r$;
and from Corollary \ref{cor:trapezoidal_pnorm2},
the second term is $O(1/n^2)$,
with constant depending on $p$, $f$ and $L$.
We have therefore shown that for all $n$ sufficiently large,
\begin{align}
\big| \mathrm{V}_p(x) - \|f\|_{V^p} \big|
\le C \max\left\{\frac{\log(n)}{n^r}, \frac{1}{n^2} \right\},
\end{align}
with $C = C(p,f,L,r) > 0$,
completing the proof of
Proposition \ref{prop:discrete_volterra_1d}.

\subsubsection{Proof of Proposition \ref{prop:noise_1d}}

First, we note that the third part of the proposition
(almost sure convergence) follows
immediately from the second part (the concentration bound)
from Lemma \ref{lem:borel_cantelli}.

Now, recall that the vectors $\alpha$ and $\beta$ are defined as follows:
\begin{align}
\alpha[k] = \frac{L}{n} \sum_{j=0}^{n-1} Z[j] e^{-2 \pi \mi k t_j / L},
    \quad 0 < |k| < n/2,
\end{align}
\begin{align}
\beta[k] =
\frac{\alpha[k]}{2\pi \mi k/L} , \quad 0 < |k| < n/2,
\end{align}
and when $k=0$,
\begin{align}
\beta[0] =
- \sum_{0 < |k| < n/2} \beta[k]  e^{2 \pi \mi k a /L}.
\end{align}

Define the random vector $W$ by $W[j] = \nu_Z(t_j)$, that is,
\begin{align}
W[j] = \frac{1}{L} \sum_{k=-n/2+1}^{n/2-1} \beta[k] e^{2 \pi \mi t_j k /L},
\end{align}
for $0 \le j \le n-1$.
Then for $1 \le p < \infty$,
\begin{align}
\mathrm{V}_p(Z_n) = \left( \frac{L}{n} \sum_{j=0}^{n-1} |W[j]|^p \right)^{1/p},
\end{align}
and $\mathrm{V}_\infty(Z_n) = \|W\|_\infty$.

It will be convenient to define the auxiliary vector $X$ by
\begin{align}
X[j] = \sum_{0 < |k| < n/2} \beta[k] e^{2 \pi \mi t_j k /L}.
\end{align}
Then $W[j] = (X[j] + \beta[0])/L$.

\paragraph{Expectation of $\mathrm{V}_p(Z_n)$.}

For a fixed $0 \le \ell \le n-1$,
\begin{align}
X[\ell] &= \sum_{0 < |k| < n/2} \beta[k] e^{2 \pi \mi t_\ell k /L}
\nonumber \\
&= \frac{L}{2 \pi \mi} \sum_{0 < |k| < n/2} \frac{\alpha[k]}{k} e^{2 \pi \mi t_\ell k /L}
\nonumber \\
&= \frac{L}{2 \pi \mi} \sum_{0 < |k| < n/2} \frac{1}{k}
    \left( \frac{L}{n}\sum_{j=0}^{n-1} Z[j] e^{-2 \pi \mi k t_j / L} \right)
        e^{2 \pi \mi t_\ell k /L}
\nonumber \\
&= \frac{L^2}{2 \pi \mi n} \sum_{j=0}^{n-1} Z[j]
        \sum_{0 < |k| < n/2} \frac{e^{2 \pi \mi k (\ell-j) / n}}{k}
\nonumber \\
&= \frac{L^2}{\pi n} \sum_{j=0}^{n-1} Z[j]  \sum_{k=1}^{n/2-1} \frac{\sin(2 \pi (\ell-j) k /n)}{k}
\nonumber \\
&= \frac{L^2}{\pi n} \sum_{j=0}^{n-1} Q_n[\ell-j] Z[j],
\end{align}
where
\begin{align}
Q_n[m] = \sum_{k=1}^{n/2-1} \frac{\sin(2 \pi m k /n)}{k}.
\end{align}
By Lemma \ref{lem:trig_sine}, $|Q_n[m]| \le C$, for all $m$ and an absolute constant $C$.
Since the $Z[j]$ are independent and mean zero,
we may apply Proposition \ref{prop:subgaussian_pythagorean}
to get the bound
\begin{align}
\|X[\ell]\|_{\psi_2}^2
&\le C \frac{L^4}{n^2} \sum_{j=0}^{n-1} Q_n[\ell - j]^2 \|Z[j]\|_{\psi_2}^2
\le C L^4 \frac{\sigma^2 }{n},
\end{align}
where $C > 0$ is a universal constant.
A nearly identical proof shows that $\|\beta[0]\|_{\psi_2}^2 \le C L^4 \sigma^2 / n$.
Since $W[j] = (X[j] + \beta[0])/L$,
it follows that, for all $0 \le j \le n-1$,
\begin{align}
\label{eq:variance_w}
\|W[j] \|_{\psi_2} \le C L \frac{\sigma}{\sqrt{n}},
\end{align}
and hence, for any $1 \le p < \infty$,
from Proposition \ref{prop:subgaussian_pnorm},
\begin{align}
\EE[|W[j]|^p] \le \left( C \sqrt{p} L \frac{\sigma}{\sqrt{n}} \right)^{p},
\end{align}
for universal $C>0$.
Consequently,
\begin{align}
\EE[\mathrm{V}_p(Z_n)^p] = \frac{L}{n}\sum_{j=0}^{n-1} \EE[|W[j]|^p]
\le \left( C \sqrt{p} L^{1+1/p} \frac{\sigma}{\sqrt{n}} \right)^{p},
\end{align}
and therefore, using Jensen's inequality,
\begin{align}
\EE[\mathrm{V}_p(Z_n)] \le \left(\EE[\mathrm{V}_p(Z_n)^p] \right)^{1/p}
\le C \sqrt{p} L^{1+1/p} \frac{\sigma}{\sqrt{n}},
\end{align}
for universal $C>0$.

The case for $p=\infty$
follows immediately from Lemma \ref{lem:expected_maximum_subgaussians}.

\paragraph{Concentration of $\mathrm{V}_p(Z_n)$.}

Fix $1 \le p < \infty$,
so that
\begin{align}
\mathrm{V}_p(Z_n)^p = \frac{L}{n} \sum_{j=0}^{n-1} |W[j]|^p.
\end{align}
We have already shown that each $\|W[j]\|_{\psi_2} \le CL \sigma / \sqrt{n}$,
for universal $C>0$;
therefore $|W[j]|^p$ is sub-Weibull($2/p$) (see Section \ref{sec:subweibull}),
with
\begin{align}
\||W[j]|^p\|_{\psi_{2/p}}
= \|W[j]\|_{\psi_2}^p
\le C^pL^p \sigma^p / n^{p/2}.
\end{align}
When $1 \le p \le 2$, the sub-Weibull parameter $2/p$
is greater than or equal to $1$, and so by the triangle inequality,
\begin{align}
\|\mathrm{V}_p(Z_n)\|_{\psi_2}^p
&= \|\mathrm{V}_p(Z_n)^p\|_{\psi_{2/p}}
\nonumber \\
&= \left \|\frac{L}{n} \sum_{j=0}^{n-1} |W[j]|^p \right\|_{\psi_{2/p}}
\nonumber \\
&\le \frac{L}{n} \sum_{j=0}^{n-1} \| |W[j]|^p \|_{\psi_{2/p}}
\nonumber \\
&\le \frac{L}{n} \sum_{j=0}^{n-1} C^pL^p \sigma^p / n^{p/2}
\nonumber \\
&= \frac{C^p L^{1 + p} \sigma^p}{n^{p/2}},
\end{align}
and hence
\begin{align}
\|\mathrm{V}_p(Z_n)\|_{\psi_2} \le \frac{C L^{1+1/p} \sigma}{\sqrt{n}},
\end{align}
where $C>0$ is universal.

On the other hand, if $p > 2$, the sub-Weibull parameter $2/p$
is between $0$ and $1$;
by Proposition \ref{prop:subweibull_pseudonorm},
there is a universal $C>0$ such that
\begin{align}
\|\mathrm{V}_p(Z_n)\|_{\psi_2}^p
&= \|\mathrm{V}_p(Z_n)^p\|_{\psi_{2/p}}
\nonumber \\
&= \left \|\frac{L}{n} \sum_{j=0}^{n-1} |W[j]|^p \right\|_{\psi_{2/p}}
\nonumber \\
&\le C^{p/2} (2/p)^{-p/2-1}
    \frac{L}{n} \sum_{j=0}^{n-1} \left\| |W[j]|^p \right\|_{\psi_{2/p}}
\nonumber \\
&\le C^{p/2} (2/p)^{-p/2-1}
    \frac{L}{n} \sum_{j=0}^{n-1} C^pL^p \sigma^p / n^{p/2}
\nonumber \\
&= C^p (2/p)^{-p/2-1} \frac{L^{1+p} \sigma^p }{n^{p/2}},
\end{align}
and therefore
$\|\mathrm{V}_p(Z_n)\|_{\psi_2} \le C \sqrt{p} L^{1+1/p} \sigma / \sqrt{n}$,
where $C > 0$ is universal.
Hence, for all $t \ge 0$,
\begin{align}
\Prob\{\mathrm{V}_p(Z_n) \ge t\}
\le 2 \exp\left\{ -C \frac{n t^2}{p L^{2+2/p} \sigma^2} \right\},
\end{align}
for universal $C>0$,
as claimed.

When $p=\infty$,
we may apply Proposition \ref{prop:norm_maximum_subgaussians}
to $\mathrm{V}_\infty(Z_n) = \max_{0 \le j \le n-1} |W[j]|$
to obtain the bound
\begin{align}
\|\mathrm{V}_\infty(Z_n)\|_{\psi_2}
\le C \sqrt{\log(n)} \max_{1 \le j \le n} \|W[j]\|_{\psi_2}
\le C L \sigma \sqrt{\frac{\log(n)}{n}},
\end{align}
for universsal $C>0$;
consequently, for any $t \ge 0$,
\begin{align}
\Prob\{\mathrm{V}_\infty(Z_n) \ge t\}
\le 2 \exp\left\{ -C \frac{n t^2}{L^{2} \sigma^2 \log(n)} \right\},
\end{align}
for universal $C>0$,
which is the claimed bound.

\subsection{Robustness to noise in 2D}

We turn now to the discrete approximation
of the sliced Cram\'er distance in 2D, with respect
to the uniform measure over $\mathbb{S}^1$.
We use the discretization described in \cite{shi2025fast}
and apply the non-uniform discrete Fourier transform to compute the Radon transform
of the input functions.
Let $R > 0$ and $L = 2R$.
Let $n$ a positive integer; throughout this discussion, we will assume
for simplicity that $n$ is even.
Let $x \in \RR^{n^2} = \RR^n \times \RR^n$;
the reader should think of the entries $x[i,j]$ of $x$
as (possibly noisy) samples from a function $f : \RR^2 \to \RR$
supported on the disc of radius $R$ centered at the origin,
and satisfying $\int_{\RR^2} f = 0$;
that is, $x[i,j] \approx f(t_i,t_j)$, $0 \le i, j \le n-1$,
where $t_j = -R + 2R j/ n$.
Throughout the section, we will denote by $\calP_\theta$
the tomographic projection operator onto $(\cos(\theta),\sin(\theta))$.

For $\theta \in [0,\pi)$, define the values
\begin{align}
\label{eq:alpha_2d}
\alpha_\theta[k]
= \frac{L^2}{n^2}\sum_{i,j} x[i,j]
        e^{-2 \pi \mi k (t_i  \cos(\theta) + t_j  \sin(\theta))/L}.
\end{align}
Then for $0 < |k| < n/2$,
\begin{math}
\alpha_\theta[k]
\approx
\what{f}((k/L)\cos(\theta) , (k/L) \sin(\theta))
= \what{(\calP_{\theta} f)} (k/L).
\end{math}

For $0 < |k| < n/2$, let
\begin{align}
\beta_\theta[k] = \frac{\alpha_\theta[k]}{2 \pi \mi k / L}.
\end{align}
For $k=0$, define
\begin{align}
\beta_\theta[0] =
- \sum_{0 < |k| < n/2} \beta_\theta[k] e^{- 2 \pi \mi k R /L}
= - \sum_{0 < |k| < n/2} \beta_\theta[k] (-1)^k.
\end{align}

Define $\nu_x(t,\theta)$ by
\begin{align}
\label{eq:nu_x_2d}
\nu_x(t,\theta) = \frac{1}{L} \sum_{k=-n/2+1}^{n/2-1} \beta_\theta[k] e^{2 \pi \mi t k /L}.
\end{align}
Then $\nu_x(t,\theta) \approx (\calV \calP_\theta f)(t)$.

Let $\theta_\ell = \pi \ell/n $, $\ell = 0,\dots,n-1$.
We define the estimated sliced Volterra norm for $1 \le p < \infty$ to be
\begin{align}
\mathrm{SV}_{p}(x)
= \left(\frac{1}{n} \sum_{\ell=0}^{n-1}
        \frac{L}{n}\sum_{j=0}^{n} |\nu_x(t_j,\theta_\ell)|^p \right)^{1/p},
\end{align}
and
\begin{align}
\mathrm{SV}_{\infty}(x)
= \max_{0 \le j,\ell \le n-1} |\nu_x(t_j,\theta_\ell)|.
\end{align}
Given two vectors $x$ and $y$, define their sliced Cram\'er distance to be
\begin{align}
\what{\SC}_p (x,y) = \mathrm{SV}_{p}(x-y).
\end{align}

\begin{rmk}
Using a non-uniform Fast Fourier Transform (NUFFT)
(for example, see \cite{dutt1993fast, dutt1995fast, fessler2003nonuniform, 
greengard2004accelerating, barnett2021aliasing, barnett2019parallel})
to evaluate the $\alpha_{\theta_\ell}[k]$
and the $\nu_x(t_j,\theta_\ell)$,
the entire computation described here can be performed
at cost $O(n^2 \log n)$.
In our implementation, we use the Flatiron Institute
NUFFT (FINUFFT) \cite{barnett2021aliasing,
barnett2019parallel}.
\end{rmk}

We can now state the main result from this section,
which says that the discrete sliced Cram\'er distance
in 2D
is robust to additive heteroscedastic sub-Gaussian noise:

\begin{thm}
\label{thm:discrete_main2d}
Suppose $f$ and $g$ are $C^r$ functions on $\RR^2$, $r > 2$, that are supported on the disc
of radius $R > 0$ centered at the origin, and which
satisfy
$\int_{\RR^2} f = \int_{\RR^2} g$.
Let $Z[j,k]$, $\wtilde Z[j,k]$, $0 \le j,k \le n-1$,
be independent,
mean-zero sub-Gaussian
random variables with sub-Gaussian norms $\sigma_{jk} = \|Z[j,k]\|_{\psi_2}$
and $\wtilde \sigma_{jk} = \|\wtilde Z[j,k]\|_{\psi_2}$;
and suppose that for all $n$,
$\sigma > 0$ satisfies
\begin{align}
\frac{1}{n^2}\sum_{k=0}^{n-1} \sum_{j=0}^{n-1} \sigma_{jk}^2
    + \frac{1}{n^2}\sum_{k=0}^{n-1} \sum_{j=0}^{n-1} \wtilde \sigma_{jk}^2
\le \sigma^2.
\end{align}
Let $X_n$ and $Y_n$ be vectors in $\RR^{n^2}$
with entries $X_n[j,k] = f(t_j,t_k) + Z[j,k]$ and $Y_n[j,k] = g(t_j,t_k) + \wtilde Z[j,k]$.
Then:

\begin{enumerate}

\item
Expected error.
For $1 \le p < \infty$,
there are $A > 0$ and $B = B(p,f,L) > 0$
such that for
all $\sigma \ge 0$
and $n \ge 2$,
\begin{align}
\label{eq:expectation_2d}
\EE\big|\what{\SC}_p(X_n, Y_n)   - \SC_p(f,g) \big|
    \le A L^{2+1/p} \sqrt{p} \frac{\sigma }{n}
        + \frac{B}{n^2}.
\end{align}
Furthermore, there are $A > 0$ and $B = B(f,L) > 0$
such that for
all $\sigma \ge 0$
and $n \ge 2$,
\begin{align}
\label{eq:expectation_2d_infty}
\EE\big|\what{\SC}_\infty(X_n, Y_n)   - \SC_\infty(f,g) \big|
    \le A L^{2} \sigma \frac{\sqrt{\log(n)}}{n}
        + \frac{B}{n^2}.
\end{align}

\item
Concentration bound.
For $1 \le p < \infty$,
there is  $C = C(p,f,L) > 0$
such that for all $\sigma > 0$,
$t \ge 0$,
and $n \ge 2$,
\begin{align}
\Prob\left\{ \big|\what{\SC}_p(X_n, Y_n)   - \SC_p(f,g) \big| \ge t   \right\}
\le 2  \exp\left(- C \frac{n^2 t^2}{\sigma^2} \right).
\end{align}
Furthermore, there is  $C = C(f,L) > 0$
such that for all $\sigma > 0$,
$t \ge 0$,
and $n \ge 2$,
\begin{align}
\Prob\left\{ \big|\what{\SC}_\infty(X_n, Y_n)   - \SC_\infty(f,g) \big| \ge t   \right\}
\le 2 \exp\left(- C \frac{n^2 t^2}{\sigma^2 \log(n)} \right).
\end{align}

\item 
Almost sure limit.
For all $1 \le p \le \infty$,
\begin{math}
\what{\SC}_p(X_n, Y_n) \to \SC_p(f,g)
\end{math}
almost surely as $n \to \infty$.

\end{enumerate}

\end{thm}

\begin{rmk}
As for Theorem \ref{thm:discrete_main1d},
throughout the proof of Theorem \ref{thm:discrete_main2d}
and the results on which it depends,
``$C$'' will denote a generic
positive number whose value may change
across instances;
see remark \ref{rmk:constants}.
\end{rmk}

Theorem \ref{thm:discrete_main2d} is a corollary
of the following two results
(the proof is the same as in the 1D case):

\begin{prop}
\label{prop:discrete_volterra_2d}
Suppose $f$ is a $C^r$ function on $\RR^2$, $r > 2$, that is supported on the disc
of radius $R > 0$ centered at the origin,
and $\int_{\RR^2} f = 0$.
Let $1 \le p \le \infty$,
and let $x \in \RR^{n^2}$ have entries $x[j,k] = f(t_j,t_k)$, $0 \le j, k \le n-1$.
Then for all $1 \le p \le \infty$,
there is a $C = C(p,f,L) > 0$ such that
\begin{align}
\left| \mathrm{SV}_p(x) - \|f\|_{SV^p}\right|
    \le \frac{C}{n^2},
\end{align}
for $n \ge 2$.
\end{prop}

\begin{prop}
\label{prop:noise_2d}

Let $Z[j,k]$, $j,k \ge 0$, be independent, mean-zero sub-Gaussian
random variables with sub-Gaussian norms $\sigma_{jk} = \|Z[j,k]\|_{\psi_2}$;
and suppose $\sigma > 0$ satisfies
\begin{align}
\frac{1}{n^2}\sum_{j=0}^{n-1} \sum_{k=0}^{n-1}\sigma_{jk}^2 \le \sigma^2,
\end{align}
for all $n$.
For each $n \ge 2$, let $Z_n \in \RR^{n^2}$ have entries $Z[j,k]$.
Then:

\begin{enumerate}

\item
Expectation bound.
There is a universal constant $C>0$ such that for
all $1 \le p < \infty$,
$\sigma \ge 0$,
and
$n \ge 2$,
\begin{align}
\EE\left[ \mathrm{SV}_p(Z_n) \right] \le C L^{2+1/p} \sqrt{p} \,  \frac{\sigma}{n},
\end{align}
and
\begin{align}
\EE\left[ \mathrm{SV}_\infty(Z_n) \right] \le C L^{2} \sigma \frac{\sqrt{\log(n)}}{n}.
\end{align}

\item
Concentration bound.
There is a universal constant $C > 0$
such that for all $1 \le p < \infty$,
$\sigma > 0$,
$t \ge 0$,
and $n \ge 2$,
\begin{align}
\Prob\left\{ \mathrm{SV}_p(Z_n) \ge t   \right\}
\le 2  \exp\left(- C \frac{n^2 t^2}{p L^{4+2/p}\sigma^2} \right),
\end{align}
and
\begin{align}
\Prob\left\{ \mathrm{SV}_\infty(Z_n) \ge t \right\}
\le 2 \exp\left(- C \frac{n^2 t^2}{L^{4} \sigma^2 \log(n)} \right).
\end{align}

\item 
Almost sure limit.
For all $1 \le p \le \infty$,
\begin{math}
\mathrm{SV}_p(Z_n) \to 0
\end{math}
almost surely as $n \to \infty$.

\end{enumerate}
\end{prop}

\subsubsection{Proof of Proposition \ref{prop:discrete_volterra_2d}}

\begin{lem}
\label{lem:approx_alpha_2d}
Let $f$ satisfy the conditions of Proposition \ref{prop:discrete_volterra_2d},
and let $\alpha_\theta$ be defined as in \eqref{eq:alpha_2d}.
Then there is a $C = C(f,L,r) > 0$ such that
\begin{align}
\left| \alpha_\theta[m] - \what{(\calP_\theta f)}(m/L) \right| \le \frac{C}{n^r}
\end{align}
for all $0 \le \theta < \pi$, $n \ge 2$, and $|m| < n/2$.

\end{lem}

\begin{proof}
This is an immediate consequence of Lemma \ref{lem:fourier_decay_approx}.

\end{proof}

\begin{cor}
\label{cor:nu_bound_2d}
Let $f$ satisfy the conditions of Proposition \ref{prop:discrete_volterra_2d},
and let $\nu_x$ be defined as in \eqref{eq:nu_x_2d}.
Then there is a $C = C(f,L,r)>0$ such that
\begin{align}
|\nu_x(t,\theta) - (\calV \calP_\theta f)(t)| \le C\frac{\log(n)}{n^r}
\end{align}
for all $0 \le \theta < \pi$ and $|t| < L$.

\end{cor}

\begin{proof}
For all $t$ and $\theta$,
applying Lemma \ref{lem:approx_alpha_2d} gives
\begin{align}
\left| \frac{1}{L} \sum_{0 < |k| \le n/2-1} \beta_\theta[k] e^{2 \pi \mi t k /L}
        - \frac{1}{L} \sum_{0 < |k| \le n/2-1} \frac{\what{(\calP_\theta f)}(k/L)}{2 \pi \mi k/L}
                e^{2 \pi \mi t k /L} \right|
&\le \frac{C}{n^r} \sum_{0 < |k| \le n/2-1} \frac{1}{|k|}
\nonumber \\
&\le C\frac{\log(n)}{n^r},
\end{align}
and the tail can be bounded
\begin{align}
\left| \sum_{|k| \ge n/2} \frac{\what{(\calP_\theta f)}(k/L)}{2 \pi \mi k}
                e^{2 \pi \mi t k /L}\right|
&= \left| \sum_{|k| \ge n/2} \frac{\what{f}((k/L)\cos(\theta), (k/L)\sin(\theta))}{2 \pi \mi k}
                e^{2 \pi \mi t k /L}\right|
\nonumber \\
&\le C \sum_{|k| \ge n/2} \frac{1}{k^{r+1}}
\nonumber \\
&\le C \frac{1}{n^r}.
\end{align}

Similarly, for all $\theta$,
\begin{align}
\left| \beta_\theta[0] - \what{(\calV \calP_\theta f)}(0) \right|
\le  \sum_{0 < |k| \le n/2-1} \left|\beta_\theta[k]
        - \frac{\what{(\calP_\theta f)}(k/L)}{2 \pi \mi k /L}\right|
    + \sum_{|k| \ge n/2} \left| \frac{\what{(\calP_\theta f)}(k/L)}{2 \pi \mi k/L} \right|
\le C \frac{\log(n)}{n^r}.
\end{align}

It then follows that for all $t$ and $\theta$,
\begin{align}
|\nu_x(t,\theta) - (\calV \calP_\theta f)(t)| \le C \frac{\log(n)}{n^r},
\end{align}
as claimed.

\end{proof}

For brevity, let $G(y,\theta) = (\calV \calP_\theta f)(y)$, i.e.\
\begin{align}
G(y,\theta)
= \int_{-R}^{y} \int_{-R}^{R} f(s \cos(\theta) + t \sin(\theta),
    t \cos(\theta) - s \sin (\theta)) \,dt \, ds.
\end{align}
Then from Corollary \ref{cor:nu_bound_2d},
for a constant $C = C(f,L,r) > 0$ we have
\begin{align}
\label{eq:504001}
&\left| \mathrm{SV}_p(x)
    - \left(\frac{1}{n}\sum_{\ell=0}^{n-1} \frac{L}{n} 
        \sum_{j=0}^{n-1} |G(t_j,\theta_\ell)|^p \right)^{1/p} \right|
\nonumber \\
&= \left| \left(\frac{1}{n} \sum_{\ell=0}^{n-1}
    \frac{L}{n}\sum_{j=0}^{n-1} |\nu_x(t_j,\theta_\ell)|^p \right)^{1/p}
    - \left(\frac{1}{n}\sum_{\ell=0}^{n-1} \frac{L}{n} 
        \sum_{j=0}^{n-1} |G(t_j,\theta_\ell)|^p \right)^{1/p} \right|
\nonumber \\
&\le \left(\frac{1}{n} \sum_{\ell=0}^{n-1}
    \frac{L}{n}\sum_{j=0}^{n-1} |\nu_x(t_j,\theta_\ell) - G(t_j,\theta_\ell)|^p \right)^{1/p}
\nonumber \\
&\le C\frac{\log(n)}{n^r},
\end{align}
with the obvious modifications when $p=\infty$.

Define the function
\begin{math}
I_p(\theta)
= \|G(\cdot,\theta)\|_{L^p(dt)}.
\end{math}
We then have the following lemma:

\begin{lem}
\label{lem:error_svp}
Let $f$ satisfy the conditions of Proposition \ref{prop:discrete_volterra_2d}.
For all $1 \le p \le \infty$,
there is an $N_p$ such that for all $n \ge N_p$,
\begin{align}
\left| \left(\frac{1}{n}\sum_{\ell=0}^{n-1} I_p(\theta_\ell)^p \right)^{1/p}
    - \|f\|_{SV^p} \right|
\le C L^{2+1/p} \left\|\frac{\partial^2 G}{\partial \theta^2} \right\|_{L^\infty}
    \frac{1}{n^2},
\end{align}
with the obvious modification when $p=\infty$,
where $C$ is a universal constant.

\end{lem}

\begin{proof}
First suppose $1 \le p < \infty$.
For each $k \ne 0$,
\begin{align}
\int_{0}^{2\pi} I_p(\theta)^p e^{- \mi k \theta} \, d\theta
= \int_{0}^{2 \pi}  \left(\int_{-R}^{R} |G(t,\theta)|^p \,dt \right) e^{- \mi k \theta} \, d\theta
= \int_{-R}^{R} \int_{0}^{2 \pi}  |G(t,\theta)|^p e^{- \mi k \theta} \,d\theta  \, dt,
\end{align}
and by Lemma \ref{lem:fourier_pth_power}, for all $t$ we have the bound
\begin{align}
\left| \int_{0}^{2 \pi}  |G(t,\theta)|^p e^{- \mi k \theta} \,d\theta \right|
\le C \frac{ L^2 p }{k^2}\left\|\frac{\partial^2 G}{\partial \theta^2} \right\|_{L^\infty}
    \int_{0}^{2\pi} |G(t,\theta)|^{p-1} d\theta,
\end{align}
where $C$ is universal;
therefore,
\begin{align}
\left| \int_{0}^{2 \pi} I_p(\theta)^p e^{- \mi k \theta} \, d\theta \right|
\le C \frac{ L^2 p }{k^2}\left\|\frac{\partial^2 G}{\partial \theta^2} \right\|_{L^\infty}
    \int_{-R}^{R} \int_{0}^{2\pi} |G(t,\theta)|^{p-1} \, d\theta \, dt.
\end{align}

By Lemma \ref{lem:fourier_series_approx},
it then follows that, for all $1 \le p < \infty$,
\begin{align}
\left| \frac{1}{2n}\sum_{\ell=0}^{2n-1} I_p(2 \pi \ell / 2n)^p
    - \frac{1}{2\pi}\int_{0}^{2\pi} I_p(\theta)^p \, d\theta\right|
\le C \frac{ L^2 p }{n^2}\left\|\frac{\partial^2 G}{\partial \theta^2} \right\|_{L^\infty}
    \int_{-R}^{R} \int_{0}^{2\pi} |G(t,\theta)|^{p-1} \, d\theta \, dt,
\end{align}
and, since $I_p$ is $\pi$-periodic, we have
\begin{align}
\frac{1}{2n}\sum_{\ell=0}^{2n-1} I_p(2 \pi \ell / 2n)^p
= \frac{1}{n}\sum_{\ell=0}^{n-1} I_p(\pi \ell / n)^p
= \frac{1}{n}\sum_{\ell=0}^{n-1} I_p(\theta_\ell)^p,
\end{align}
so that
\begin{align}
\label{eq:6050201-1}
\left| \frac{1}{n}\sum_{\ell=0}^{n-1} I_p(\theta_\ell)^p
    - \|f\|_{SV^p}^p \right|
&= \left| \frac{1}{n}\sum_{\ell=0}^{n-1} I_p(\theta_\ell)^p
    - \frac{1}{2\pi}\int_{0}^{2\pi} I_p(\theta)^p \, d\theta\right|
\nonumber \\
&\le C \frac{ L^2 p }{n^2}\left\|\frac{\partial^2 G}{\partial \theta^2} \right\|_{L^\infty}
    \int_{-R}^{R} \int_{0}^{2\pi} |G(t,\theta)|^{p-1} \, d\theta \, dt
\nonumber \\
&\le C \frac{ p L^{2+1/p} }{n^2}\left\|\frac{\partial^2 G}{\partial \theta^2} \right\|_{L^\infty}
    \left(\int_{-R}^{R} \int_{0}^{2\pi} |G(t,\theta)|^{p} \, d\theta \, dt \right)^{(p-1)/p}
\nonumber \\
&= C \frac{ p L^{2+1/p} }{n^2}\left\|\frac{\partial^2 G}{\partial \theta^2} \right\|_{L^\infty}
    \|f\|_{SV^p}^{p-1}.
\end{align}
Assuming $\|f\|_{SV^p} \ne 0$ (for otherwise $f \equiv 0$ and the proof is trivial),
it follows that for all $n$ sufficiently large,
\begin{align}
\frac{1}{n}\sum_{\ell=0}^{n-1} I_p(\theta_\ell)^p \ge \frac{1}{2} \|f\|_{SV^p}^p,
\end{align}
and consequently, applying the mean value theorem to the function
$y \mapsto y^{1/p}$,
\begin{align}
\left| \left(\frac{1}{n}\sum_{\ell=0}^{n-1} I_p(\theta_\ell)^p \right)^{1/p}
    - \|f\|_{SV^p} \right|
&\le C \frac{\left( \|f\|_{SV^p}^p / 2 \right)^{1/p-1}}{p}
    \left(\frac{ p L^{2+1/p}}{n^2}\left\|\frac{\partial^2 G}{\partial \theta^2} \right\|_{L^\infty}
    \|f\|_{SV^p}^{p-1} \right)
\nonumber \\
&\le C L^{2+1/p} \left\|\frac{\partial^2 G}{\partial \theta^2} \right\|_{L^\infty}
    \frac{1}{n^2},
\end{align}
as claimed.

When $p = \infty$: fix any $t$ in $[-R,R]$. Then
from Corollary \ref{cor:trapezoidal_pnorm2}
\begin{align}
\left| \max_{0 \le \ell \le n-1} |G(t,\theta_\ell)|
    - \max_{0 \le \theta \le \pi} |G(t,\theta)| \right|
\le C L^2 \left\|\frac{\partial^2 G}{\partial \theta^2} \right\|_{L^\infty}
    \frac{1}{n^2},
\end{align}
and so
\begin{align}
\left|\max_{0 \le \ell \le n-1}  I_\infty(\theta_\ell) - \|f\|_{SV^\infty}\right|
&= \left| \max_{0 \le \ell \le n-1}  \max_{-R \le t \le R} |G(t,\theta_\ell)|
    - \max_{0 \le \theta \le \pi}  \max_{-R \le t \le R} |G(t,\theta)| \right|
\nonumber \\
&= \left| \max_{-R \le t \le R} \max_{0 \le \ell \le n-1}  |G(t,\theta_\ell)|
    - \max_{-R \le t \le R}  \max_{0 \le \theta \le \pi} |G(t,\theta)| \right|
\nonumber \\
&\le \max_{-R \le t \le R} \left| \max_{0 \le \ell \le n-1} |G(t,\theta_\ell)|
    - \max_{0 \le \theta \le \pi} |G(t,\theta)| \right|
\nonumber \\
&\le C L^2 \left\|\frac{\partial^2 G}{\partial \theta^2} \right\|_{L^\infty}
    \frac{1}{n^2},
\end{align}
as claimed.

\end{proof}

\begin{lem}
\label{lem:error_svp2}

Let $f$ satisfy the conditions of Proposition \ref{prop:discrete_volterra_2d}.
For all $1 \le p \le \infty$,
there is an $N_p$ such that for all $n \ge N_p$,
\begin{align}
\left| \left(\frac{1}{n}\sum_{\ell=0}^{n-1}
            \frac{L}{n}\sum_{j=0}^{n-1}|G(t_j,\theta_\ell)|^p \right)^{1/p}
    - \left(\frac{1}{n}\sum_{\ell=0}^{n-1} I_p(\theta_\ell)^p \right)^{1/p}\right|
\le C L^{2+1/p}
    \left\| \frac{\partial^2 G }{\partial t^2} \right\|_{L^\infty} \frac{1}{n^2},
\end{align}
with the obvious modification when $p=\infty$,
where $C$ is a universal constant.

\end{lem}

\begin{proof}
First, suppose $1 \le p < \infty$.
By Corollary \ref{cor:trapezoidal_pnorm}, for each $\theta$,
\begin{align}
\left| \frac{L}{n}\sum_{j=0}^{n-1}|G(t_j,\theta)|^p - I_p(\theta)^p\right|
&=\left| \frac{L}{n}\sum_{j=0}^{n-1}|G(t_j,\theta)|^p
        -  \int_{-R}^{R} |G(t,\theta)|^p dt \right|
\nonumber \\
&\le C \frac{L^{2} p}{n^2}  \left\| \frac{\partial^2 G }{\partial t^2} \right\|_{L^\infty}
    \int_{-R}^{R} |G(t,\theta)|^{p-1} dt,
\end{align}
and therefore,
\begin{align}
\label{eq:6050201}
\left| \frac{1}{n}\sum_{\ell=0}^{n-1}\frac{L}{n}\sum_{j=0}^{n-1}|G(t_j,\theta_\ell)|^p
    - \frac{1}{n}\sum_{\ell=0}^{n-1} I_p(\theta_\ell)^p\right|
&\le C \frac{L^{2} p}{n^2}  \left\| \frac{\partial^2 G }{\partial t^2} \right\|_{L^\infty}
    \left(\frac{1}{n}\sum_{\ell=0}^{n-1}\int_{-R}^{R} |G(t,\theta)|^{p-1} dt \right)
\nonumber \\
&\le C \frac{L^{2+1/p} p}{n^2}  \left\| \frac{\partial^2 G }{\partial t^2} \right\|_{L^\infty}
    \left(\frac{1}{n}\sum_{\ell=0}^{n-1}\int_{-R}^{R} |G(t,\theta)|^{p} dt \right)^{(p-1)/p}
\nonumber \\
&= C \frac{L^{2+1/p} p}{n^2}  \left\| \frac{\partial^2 G }{\partial t^2} \right\|_{L^\infty}
    \left(\frac{1}{n} \sum_{\ell=0}^{n-1} I_p(\theta_\ell)^p \right)^{(p-1)/p}.
\end{align}
By Lemma \ref{lem:error_svp}, for $n$ sufficiently large,
\begin{align}
\frac{1}{2^{1/p}} \|f\|_{SV^p}
\le \left(\frac{1}{n}\sum_{\ell=0}^{n-1} I_p(\theta_\ell)^p \right)^{1/p}
\le 2^{1/p} \|f\|_{SV^p},
\end{align}
and so  if $n$ is also large enough to satisfy
\begin{align}
n^2 \ge 2^{1+1/p} C L^{2+1/p} p \left\| \frac{\partial^2 G }{\partial t^2} \right\|_{L^\infty}
    \|f\|_{SV^p}^{-1}
\end{align}
then it is also true that
\begin{align}
n^2 \ge 2 C L^{2+1/p} p \left\| \frac{\partial^2 G }{\partial t^2} \right\|_{L^\infty}
    \left(\frac{1}{n}\sum_{\ell=0}^{n-1} I_p(\theta_\ell)^p \right)^{-1/p}
\end{align}
and consequently, for all sufficiently large $n$,
\begin{align}
C L^{2+1/p} p \left\| \frac{\partial^2 G }{\partial t^2} \right\|_{L^\infty}
    \left(\frac{1}{n}\sum_{\ell=0}^{n-1} I_p(\theta_\ell)^p \right)^{(p-1)/p} \frac{1}{n^2}
\le \frac{1}{2} \left( \frac{1}{n}\sum_{\ell=0}^{n-1} I_p(\theta_\ell)^p  \right),
\end{align}
in which case \eqref{eq:6050201} implies
\begin{align}
\frac{1}{n}\sum_{\ell=0}^{n-1}\frac{L}{n}\sum_{j=0}^{n-1}|G(t_j,\theta_\ell)|^p
\ge \frac{1}{2} \left( \frac{1}{n}\sum_{\ell=0}^{n-1} I_p(\theta_\ell)^p  \right)
\ge \frac{1}{4} \|f\|_{SV^p}^p.
\end{align}
Since
\begin{align}
\frac{1}{n}\sum_{\ell=0}^{n-1} I_p(\theta_\ell)^p
\ge \frac{1}{2} \|f\|_{SV^p}^p
\ge \frac{1}{4} \|f\|_{SV^p}^p,
\end{align}
by the mean value theorem applied to $y \mapsto y^{1/p}$ it follows
that
\begin{align}
&\left| \left(\frac{1}{n}\sum_{\ell=0}^{n-1}
            \frac{L}{n}\sum_{j=0}^{n-1}|G(t_j,\theta_\ell)|^p \right)^{1/p}
    - \left(\frac{1}{n}\sum_{\ell=0}^{n-1} I_p(\theta_\ell)^p \right)^{1/p}\right|
\nonumber \\
\le \ \ & C \frac{\left( \|f\|_{SV^p}^p / 4 \right)^{1/p-1}}{p}
    \frac{L^{2+1/p} p}{n^2}  \left\| \frac{\partial^2 G }{\partial t^2} \right\|_{L^\infty}
    \left(\frac{2}{n} \sum_{\ell=0}^{n-1} I_p(\theta_\ell)^p \right)^{(p-1)/p} 
\nonumber \\
\le \ \ & C \frac{\left( \|f\|_{SV^p}^p / 4 \right)^{1/p-1}}{p}
    \frac{L^{2+1/p} p}{n^2}  \left\| \frac{\partial^2 G }{\partial t^2} \right\|_{L^\infty}
    \left(2^{1/p} \|f\|_{SV^p} \right)^{(p-1)}
\nonumber \\
\le \ \ & C L^{2+1/p}
    \left\| \frac{\partial^2 G }{\partial t^2} \right\|_{L^\infty} \frac{1}{n^2}  .
\end{align}
This completes the proof when $1 \le p < \infty$.

When $p = \infty$: fix $\theta$ in $[0,\pi]$.
By Corollary \ref{cor:trapezoidal_pnorm2},
\begin{align}
\left| \max_{0 \le j \le n-1} |G(t_j,\theta)|
    - I_\infty(\theta) \right|
= \left| \max_{0 \le j \le n-1} |G(t_j,\theta)|
    - \max_{-R \le t \le R} |G(t,\theta)| \right|
\le C L^2 \left\|\frac{\partial^2 G}{\partial t^2} \right\|_{L^\infty}
    \frac{1}{n^2},
\end{align}
and so
\begin{align}
\left| \max_{0 \le \ell \le n-1} \max_{0 \le j \le n-1} |G(t_j,\theta_\ell)|
    - \max_{0 \le \ell \le n-1} I_\infty(\theta_\ell) \right|
&\le \max_{0 \le \ell \le n-1} \left| \max_{0 \le j \le n-1} |G(t_j,\theta_\ell)|
    - I_\infty(\theta_\ell) \right|
\nonumber \\
&\le C L^2 \left\| \frac{\partial^2 G }{\partial t^2} \right\|_{L^\infty} \frac{1}{n^2},
\end{align}
as claimed.

\end{proof}

Putting together Lemma \ref{lem:error_svp}
and Lemma \ref{lem:error_svp2},
for all $n$ sufficiently large,
\begin{align}
\left| \left(\frac{1}{n}\sum_{\ell=0}^{n-1}
            \frac{L}{n}\sum_{j=0}^{n-1}|G(t_j,\theta_\ell)|^p \right)^{1/p}
        - \|f\|_{SV_p}\right|
\le C L^{2+1/p}
    \left(\left\| \frac{\partial^2 G }{\partial t^2} \right\|_{L^\infty}
        + \left\| \frac{\partial^2 G }{\partial \theta^2} \right\|_{L^\infty} \right)
    \frac{1}{n^2}.
\end{align}
Combining this with \eqref{eq:504001},
it follows that there is a constant $C = C(p,f,L) > 0$
such that
for all $n \ge 2$,
\begin{align}
\left| \mathrm{SV}_p(x) - \|f\|_{SV^p} \right| \le \frac{C}{n^2},
\end{align}
completing the proof of
Proposition \ref{prop:discrete_volterra_2d}.

\subsubsection{Proof of Proposition \ref{prop:noise_2d}}

First, we note that the third part of the proposition
(almost sure convergence) follows
immediately from the second part (concentration bound)
by using Lemma \ref{lem:borel_cantelli},
as in the 1D case.

Recall that the vectors $\alpha_\theta$ and $\beta_\theta$ are defined as follows:
\begin{align}
\alpha_\theta[k] = \frac{L^2}{n^2}\sum_{i,j} Z[i,j]
        e^{-2 \pi \mi k (t_i \cos(\theta) + t_j \sin(\theta))/L},
    \quad 0 < |k| < n/2,
\end{align}
\begin{align}
\beta_\theta[k] = \frac{\alpha_\theta[k]}{2 \pi \mi k/L},
        \quad 0 < |k| < n/2,
\end{align}
and when $k=0$,
\begin{align}
\beta_\theta[0] =
- \sum_{0 < |k| < n/2} \beta_\theta[k]  (-1)^k.
\end{align}

Define the random vector $W$ by $W[j,\ell] = \nu_Z(t_j,\theta_\ell)$, that is,
\begin{align}
W[j,\ell] = \frac{1}{L} \sum_{k=-n/2+1}^{n/2-1} \beta_{\theta_\ell}[k] e^{2 \pi \mi t_j k /L}
\end{align}
for $0 \le j,\ell \le n-1$.
Then for each $1 \le p < \infty$,
\begin{align}
\mathrm{SV}_p(Z_n) = \left( \frac{L  }{n^2} \sum_{\ell=0}^{n-1} \sum_{j=0}^{n-1}
    |W[j,\ell]|^p \right)^{1/p},
\end{align}
and $\mathrm{SV}_\infty(Z_n) = \|W\|_\infty$.

We define the auxiliary vector $X$ by
\begin{align}
X[j,\ell] = \sum_{0 < |k| < n/2} \beta_{\theta_\ell}[k] e^{2 \pi \mi t_j k /L}.
\end{align}
Note that $W[j,\ell] = (X[j,\ell] + \beta_{\theta_\ell}[0])/L$.

\paragraph{Expectation of $\mathrm{SV}_p(Z_n)$.}
For fixed $0 \le i, \ell \le n-1$,
\begin{align}
X[i,\ell] &= \sum_{0 < |k| < n/2} \beta_{\theta_\ell}[k] e^{2 \pi \mi t_i k /L}
\nonumber \\
&= \frac{L}{2 \pi \mi} \sum_{0 < |k| < n/2} \frac{\alpha_{\theta_\ell}[k]}{k} e^{2 \pi \mi t_i k /L}
\nonumber \\
&= \frac{L^3}{2 \pi \mi n^2} \sum_{0 < |k| < n/2} \frac{1}{k}  \sum_{j,j'} Z[j,j']
        e^{-2 \pi \mi k (t_j \cos(\theta_\ell) + t_{j'} \sin(\theta_\ell))/L}
            e^{2 \pi \mi t_i k /L}
\nonumber \\
&= \frac{L^3}{\pi n^2} \sum_{j,j'} Z[j,j'] \sum_{k=1}^{n/2-1}
        \frac{1}{k} \sin(k 2 \pi(t_j \cos(\theta_\ell) + t_{j'}\sin(\theta_\ell) - t_i)/L).
\nonumber \\
&= \frac{L^3}{\pi n^2} \sum_{j,j'} Z[j,j'] Q_n[i,j,j',\ell],
\end{align}
where
\begin{align}
Q_n[i,j,j',\ell]
= \sum_{k=1}^{n/2-1}
        \frac{1}{k} \sin(k 2 \pi(t_j \cos(\theta_\ell) + t_{j'}\sin(\theta_\ell) - t_i)/L).
\end{align}
From Lemma \ref{lem:trig_sine}, $|Q_n[i,j,j',\ell]| \le C$
for all $i$, $j$, $j'$ and $\ell$ and for a universal constant $C$.
Since the $Z[j,j']$ are independent and mean zero,
we may apply Proposition \ref{prop:subgaussian_pythagorean}
to get the bound
\begin{align}
\|X[i,\ell]\|_{\psi_2}^2
&\le C \frac{L^6}{n^4} \sum_{j,j'} Q_n[i,j,j',\ell]^2 \|Z[j,j']\|_{\psi_2}^2
\le C L^6 \frac{\sigma^2 }{n^2}.
\end{align}

A nearly identical proof works for $\beta[0]$,
showing that
\begin{align}
\|\beta[0]\|_{\psi_2}^2
\le C L^6 \frac{\sigma^2 }{n^2},
\end{align}
where again, $C>0$ is universal.

Since $W[j,\ell] = (X[j,\ell] + \beta_{\theta_\ell}[0])/L$,
it then follows that
\begin{align}
\label{eq:w_variance_2d}
\| W[j,\ell] \|_{\psi_2}
\le C L^2 \frac{\sigma}{n}.
\end{align}
and hence, for any $1 \le p < \infty$,
from Proposition \ref{prop:subgaussian_pnorm},
\begin{align}
\EE[|W[j,\ell]|^p] \le \left( C \sqrt{p} L^2 \frac{\sigma}{n} \right)^{p}.
\end{align}

Consequently,
\begin{align}
\EE[\mathrm{SV}_p(Z_n)^p] = \frac{L}{n^2}\sum_{j,\ell} \EE[|W[j,\ell]|^p]
\le \left( C \sqrt{p} L^{2+1/p} \frac{\sigma}{n} \right)^{p},
\end{align}
and therefore, using Jensen's inequality,
\begin{align}
\EE[\mathrm{SV}_p(Z_n)] \le \left(\EE[\mathrm{SV}_p(Z_n)^p] \right)^{1/p}
\le C \sqrt{p} L^{2+1/p} \frac{\sigma}{n}.
\end{align}
The case when $p=\infty$
follows immediately from Lemma \ref{lem:expected_maximum_subgaussians}.

\paragraph{Concentration of $\mathrm{SV}_p(Z_n)$.}

Fix $1 \le p < \infty$,
so that
\begin{align}
\mathrm{SV}_p(Z_n)^p = \frac{L}{n^2} \sum_{j=0}^{n-1} \sum_{\ell=0}^{n-1}|W[j,\ell]|^p.
\end{align}
We have shown already that $\|W[j,\ell]\|_{\psi_2} \le CL^2 \sigma / n$,
for universal $C>0$.
Therefore, $|W[j,\ell]|^p$ is sub-Weibull(2/p),
with $\| |W[j,\ell]|^p\|_{\psi_{2/p}} = \|W[j,\ell]\|_{\psi_2}^p \le (CL^2 \sigma / n)^p $.
As in the 1D case, when $0 \le p \le 2$, the sub-Weibull parameter $2/p$
is greater than or equal to $1$, and so by the triangle
inequality for $\|\cdot \|_{\psi_{2/p}}$,
\begin{align}
\| \mathrm{SV}_p(Z_n) \|_{\psi_2}^p
&= \| \mathrm{SV}_p(Z_n)^p \|_{\psi_{2/p}}
\nonumber \\
&= \left\| \frac{L}{n^2} \sum_{j=0}^{n-1} \sum_{\ell=0}^{n-1}|W[j,\ell]|^p \right\|_{\psi_{2/p}}
\nonumber \\
&\le  \frac{L}{n^2} \sum_{j=0}^{n-1} \sum_{\ell=0}^{n-1} \left\||W[j,\ell]|^p \right\|_{\psi_{2/p}}
\nonumber \\
&\le  \frac{L}{n^2} \sum_{j=0}^{n-1} \sum_{\ell=0}^{n-1} (CL^2 \sigma / n)^p
\nonumber \\
&= C^p L^{2p+1} \frac{\sigma^p}{n^p},
\end{align}
and so
\begin{align}
\| \mathrm{SV}_p(Z_n) \|_{\psi_2}
\le C L^{2+1/p}  \frac{\sigma}{n},
\end{align}
where $C>0$ is universal.

Also as in the 1D case, when $p > 2$, $0 < 2/p < 1$, and
so by Proposition \ref{prop:subweibull_pseudonorm},
there is a universal $C>0$ such that
\begin{align}
\| \mathrm{SV}_p(Z_n) \|_{\psi_2}^p
&= \| \mathrm{SV}_p(Z_n)^p \|_{\psi_{2/p}}
\nonumber \\
&= \left\| \frac{L}{n^2} \sum_{j=0}^{n-1} \sum_{\ell=0}^{n-1}|W[j,\ell]|^p \right\|_{\psi_{2/p}}
\nonumber \\
&\le  C^{p/2} (2/p)^{-p/2-1}
    \frac{L}{n^2} \sum_{j=0}^{n-1} \sum_{\ell=0}^{n-1} \left\||W[j,\ell]|^p \right\|_{\psi_{2/p}}
\nonumber \\
&\le C^{p/2} (2/p)^{-p/2-1}
\frac{L}{n^2} \sum_{j=0}^{n-1} \sum_{\ell=0}^{n-1} (CL^2 \sigma / n)^p
\nonumber \\
&= C^p (2/p)^{-p/2-1} L^{2p+1} \frac{\sigma^p}{n^p},
\end{align}
and taking $p$-th roots of each side,
\begin{align}
\| \mathrm{SV}_p(Z_n) \|_{\psi_2}
\le C \sqrt{p} L^{2 + 1/p} \frac{\sigma}{n},
\end{align}
where $C>0$ is universal.
Hence, for all $t \ge 0$,
\begin{align}
\Prob\{\mathrm{SV}_p(Z_n) \ge t\}
\le 2 \exp\left\{ -C \frac{n^2 t^2}{p L^{4+2/p} \sigma^2} \right\}
\end{align}
for universal $C>0$, as claimed.

When $p=\infty$,
we use Proposition \ref{prop:norm_maximum_subgaussians}
to get
\begin{align}
\|\mathrm{SV}_\infty(Z_n)\|_{\psi_2}
\le C \sqrt{\log(n)} \max_{0 \le j ,\ell \le n-1} \|W[j,\ell]\|_{\psi_2}
\le C L^2 \sigma \frac{\sqrt{\log(n)}}{n}.
\end{align}
Consequently, for any $t \ge 0$,
\begin{align}
\Prob\{\mathrm{SV}_\infty(Z_n) \ge t\}
\le 2 \exp\left\{ -C \frac{n^2 t^2}{L^4 \sigma^2 \log(n)} \right\},
\end{align}
with universal $C>0$,
which is the claimed bound.

\begin{figure}[h]
\centering
\includegraphics[scale=.25]{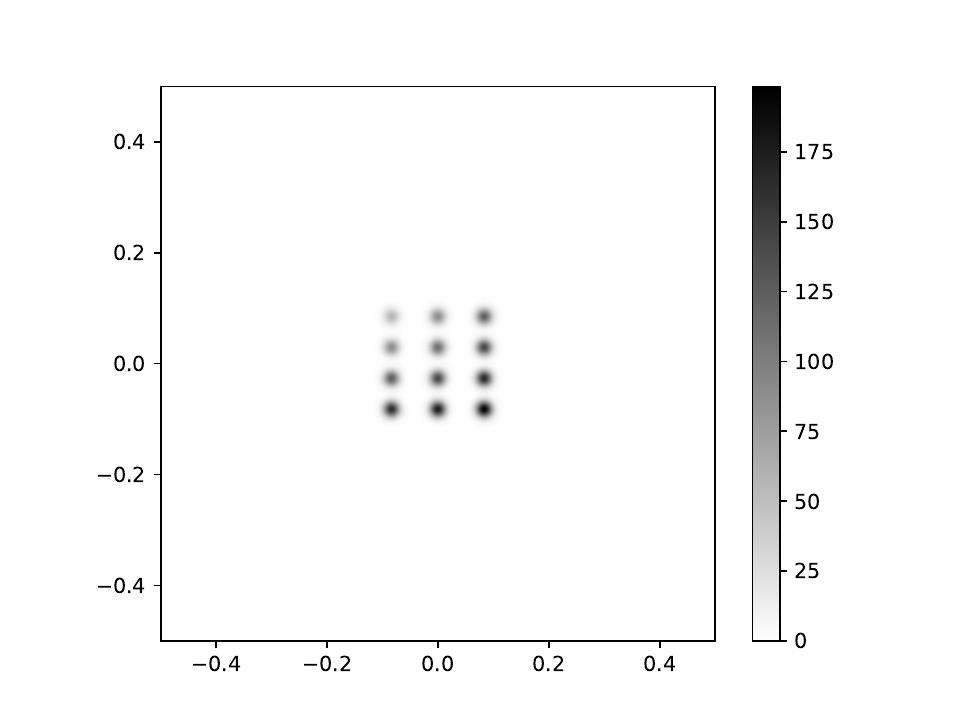}
\includegraphics[scale=.25]{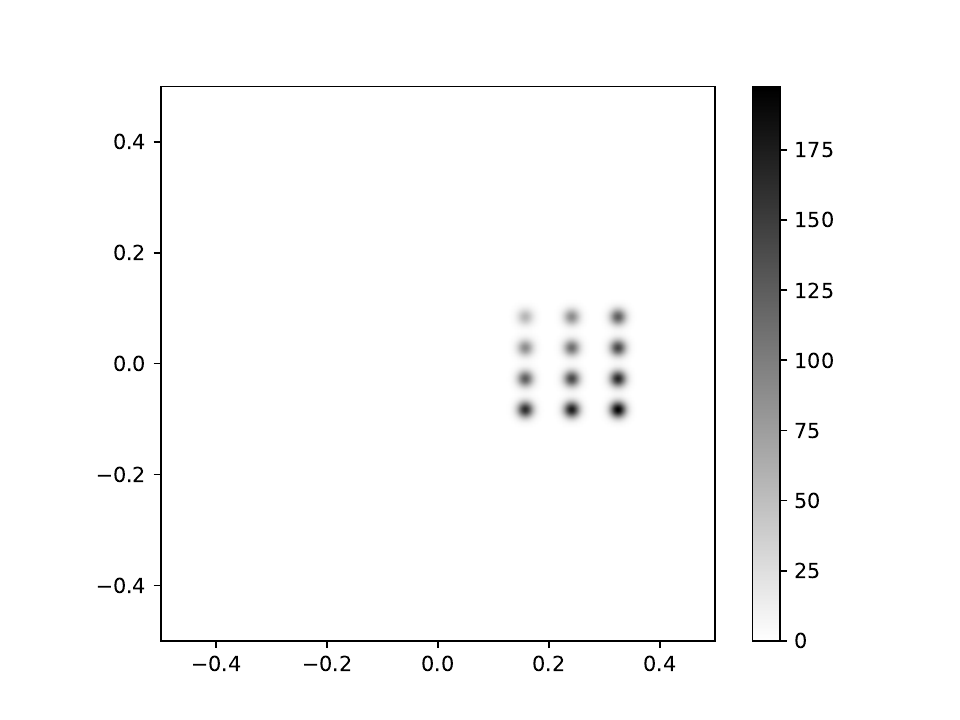}
\\
\includegraphics[scale=.25]{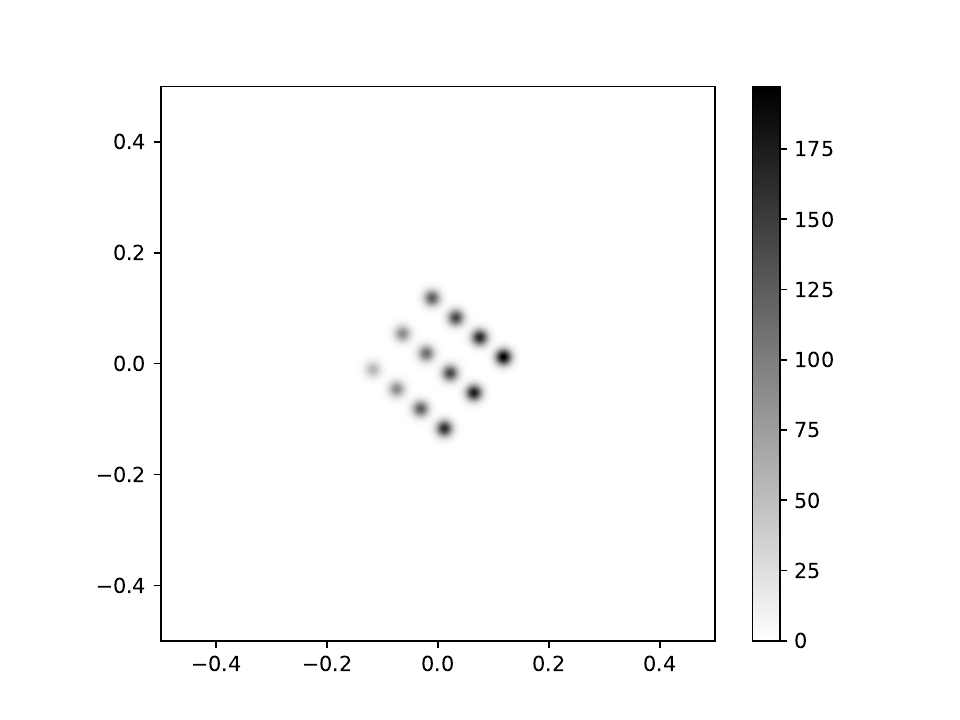}
\includegraphics[scale=.25]{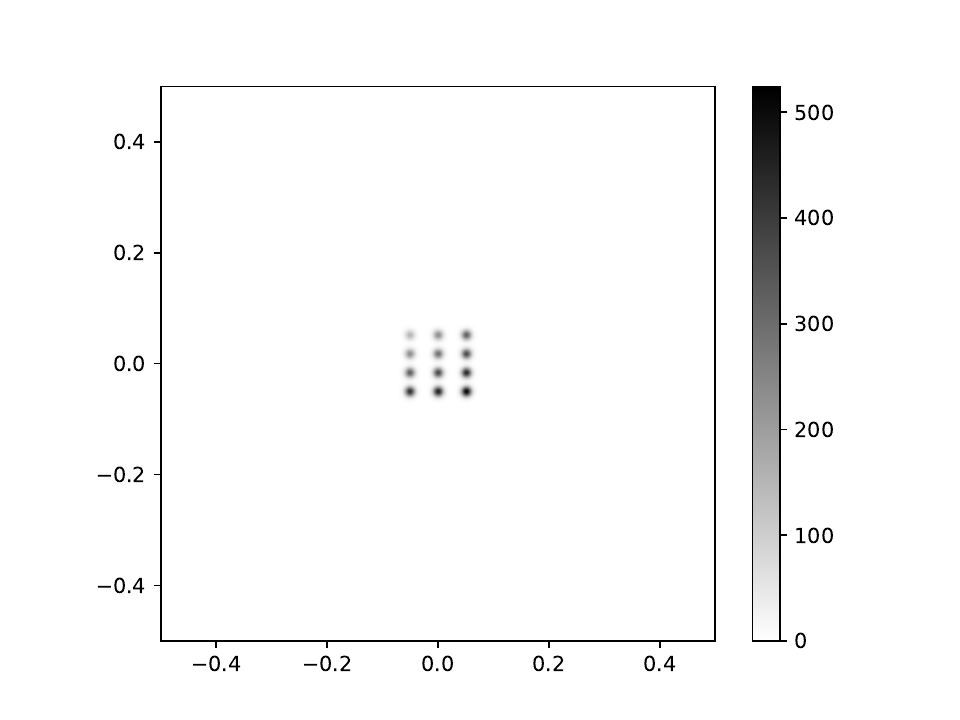}
\caption{The function $f$ described
in Section \ref{sec:numerical_deformations},
and examples of the deformations
applied to $f$.
From left to right, top to bottom:
the original function;
a translation;
a rotation;
a dilation.
}
\label{fig:deformations}
\end{figure}

\begin{figure}[h]
\centering
\includegraphics[scale=.45]{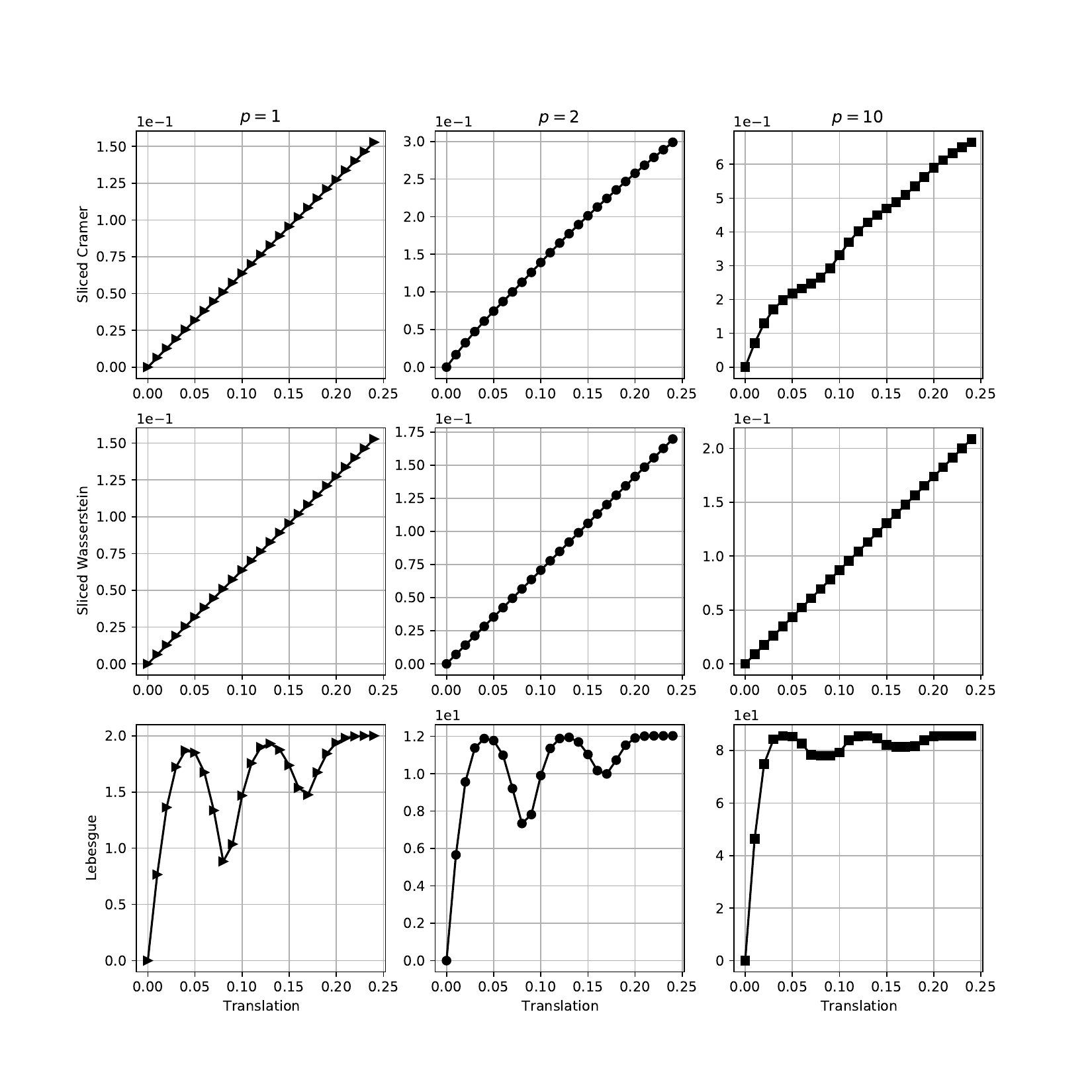}
\caption{Results of the experiment described in Section \ref{sec:numerical_deformations},
showing the distances from $f$ to its translations
as a function of the translation size.}
\label{fig:translations}
\end{figure}

\begin{figure}[h]
\centering
\includegraphics[scale=.45]{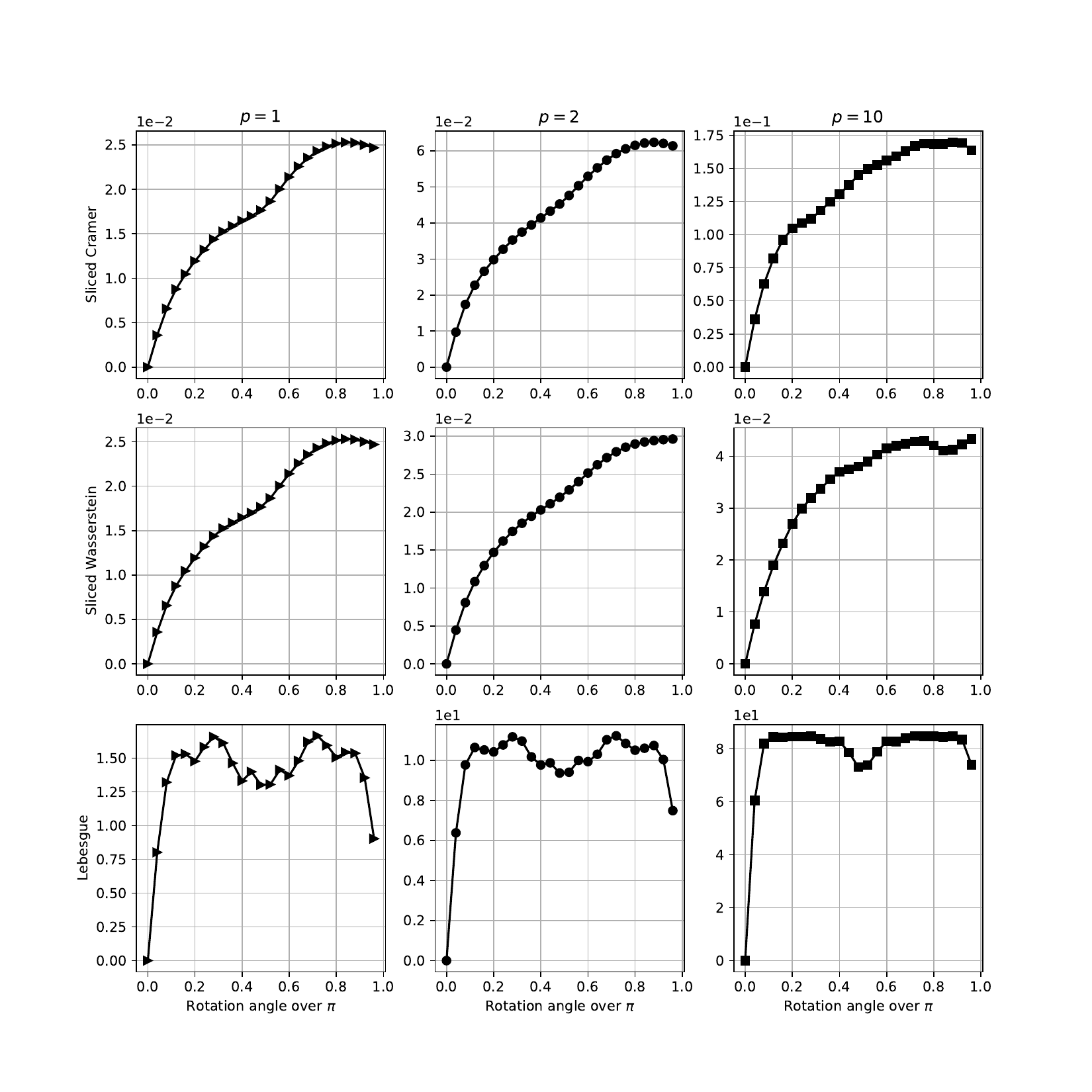}
\caption{Results of the experiment described in Section \ref{sec:numerical_deformations},
showing the distances from $f$ to its rotations
as a function of the rotation angle normalized by $\pi$.}
\label{fig:rotations}
\end{figure}

\begin{figure}[h]
\centering
\includegraphics[scale=.45]{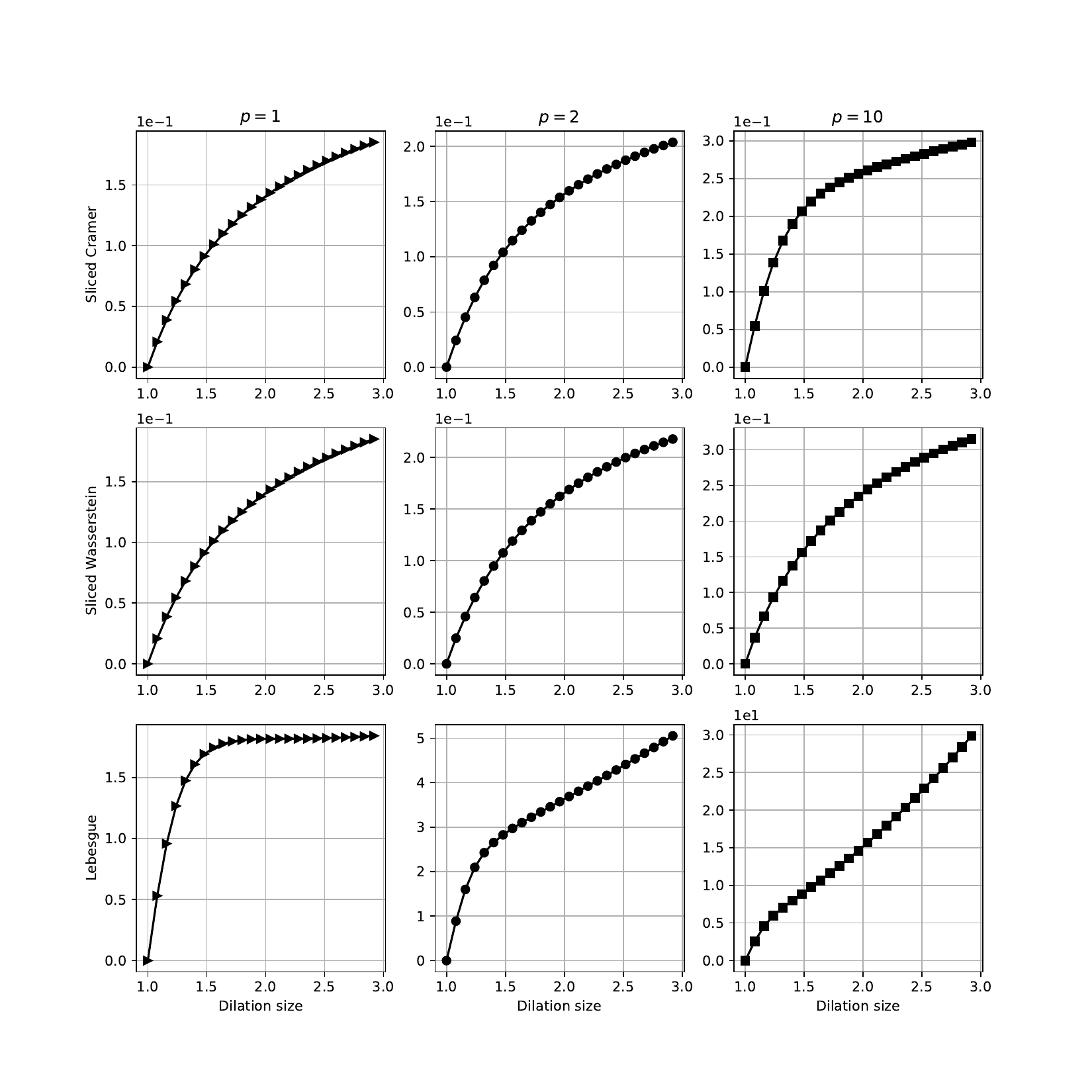}
\caption{Results of the experiment described in Section \ref{sec:numerical_deformations},
showing the distances from $f$ to its dilations
as a function of the dilation size.}
\label{fig:dilations}
\end{figure}

\begin{figure}[h]
\centering
\includegraphics[scale=.35]{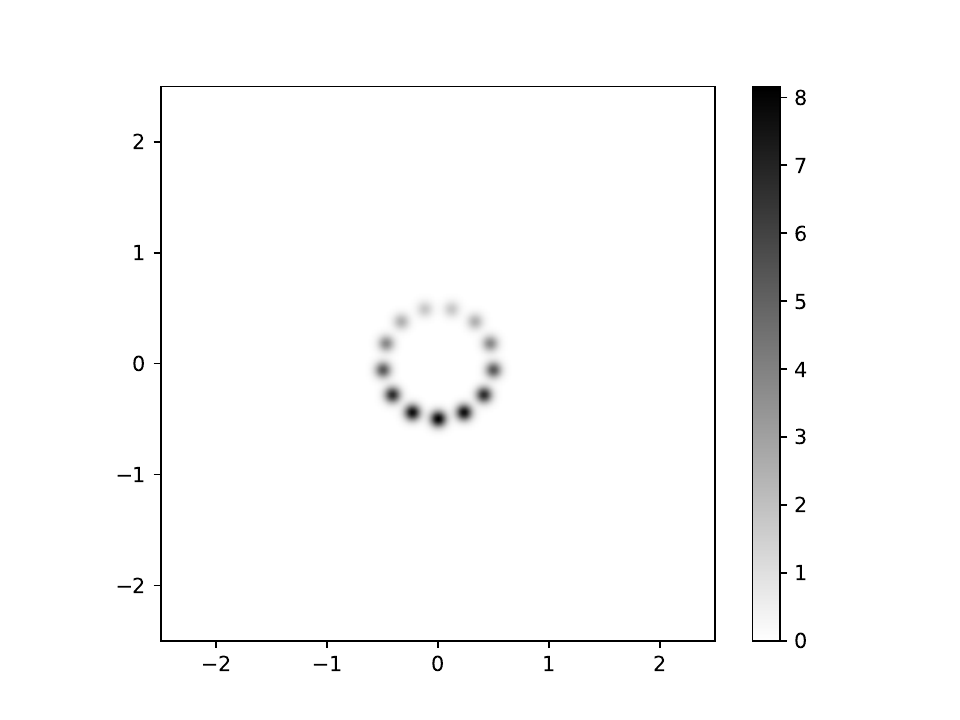}
\includegraphics[scale=.35]{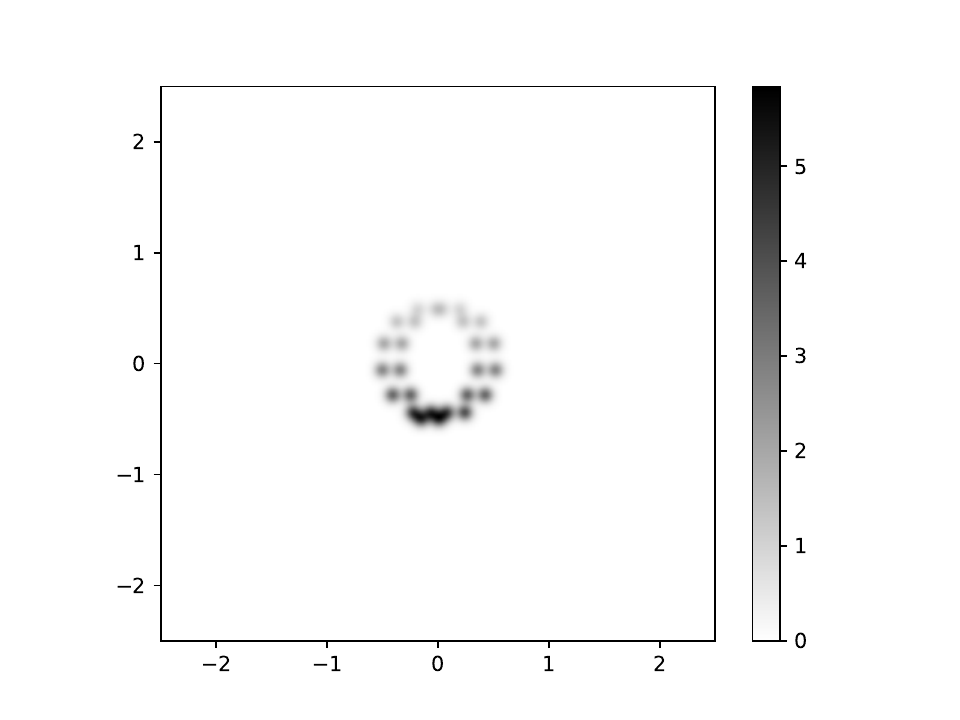}
\caption{Projections used in the experiment described in Section \ref{sec:numerical_rotations3d}.
Left: the projection onto the $xy$-plane of the original spiral function of $f$.
Right: the projection onto the $xy$-plane of a rotation of $f$ within the $yz$-plane.}
\label{fig:projections}
\end{figure}

\begin{figure}[h]
\centering
\includegraphics[scale=.45]{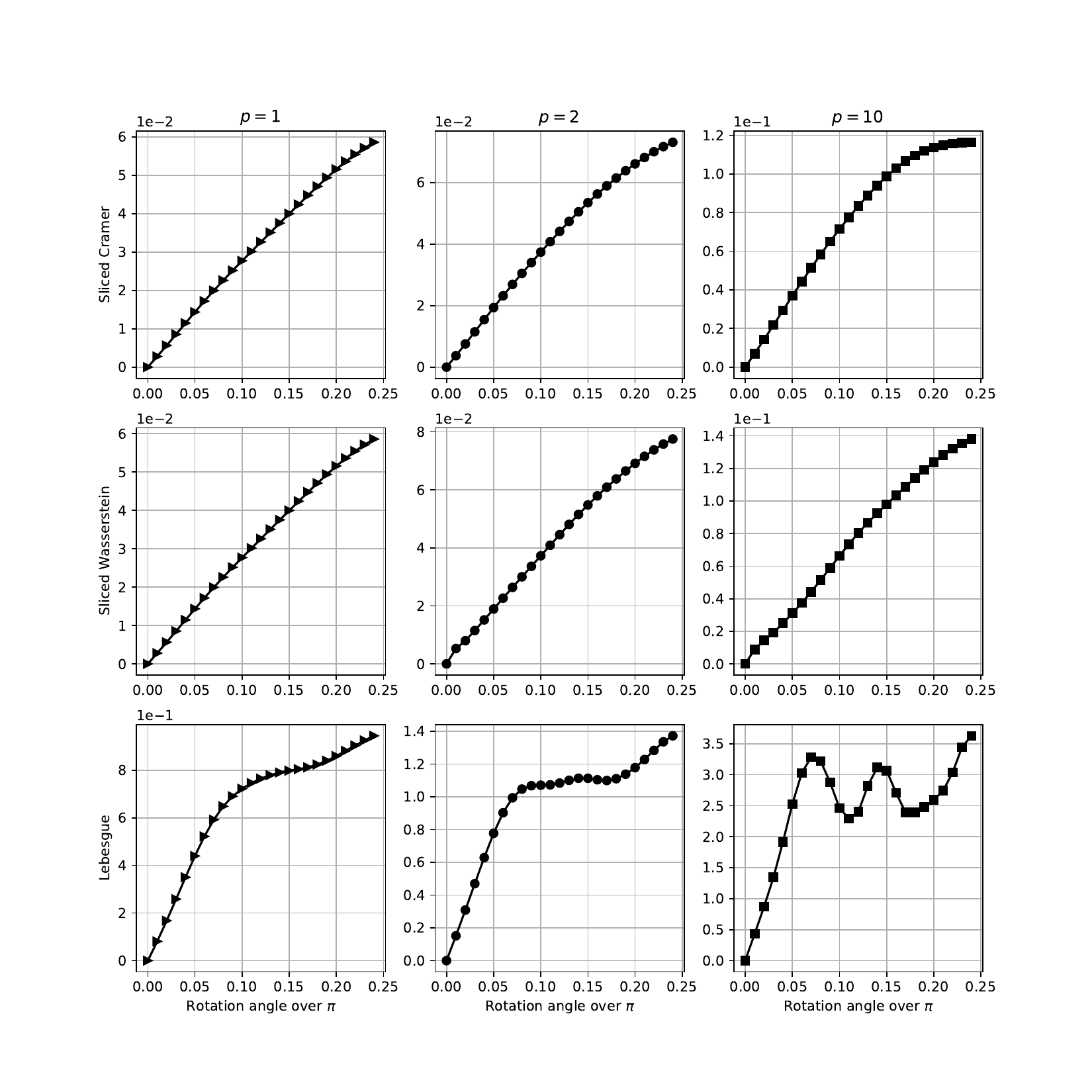}
\caption{Results of the experiment described in Section \ref{sec:numerical_rotations3d},
showing the distances from the projection of $f$ to the projections of
its rotation as a function of the rotation angle normalized by $\pi$.}
\label{fig:distances_projections}
\end{figure}

\begin{figure}[h]
\centering
\includegraphics[scale=.45]{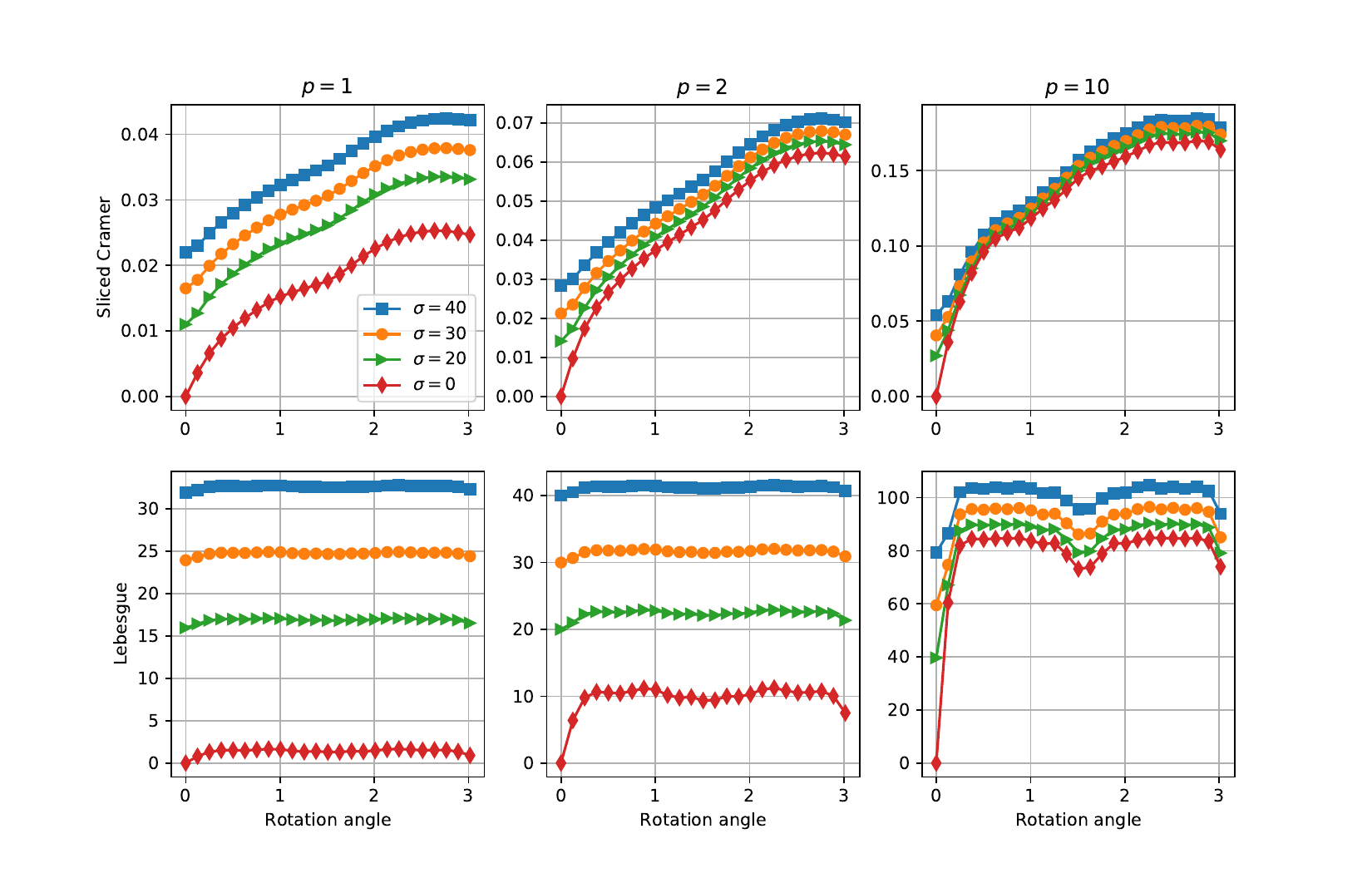}
\caption{Results of the experiment described in Section \ref{section:noise},
showing the distances between $f$ and its noisy rotations
as the rotation angle increases and for different noise levels.}
\label{fig:defnoise}
\end{figure}

\begin{figure}[h]
\centering
\includegraphics[scale=.39]{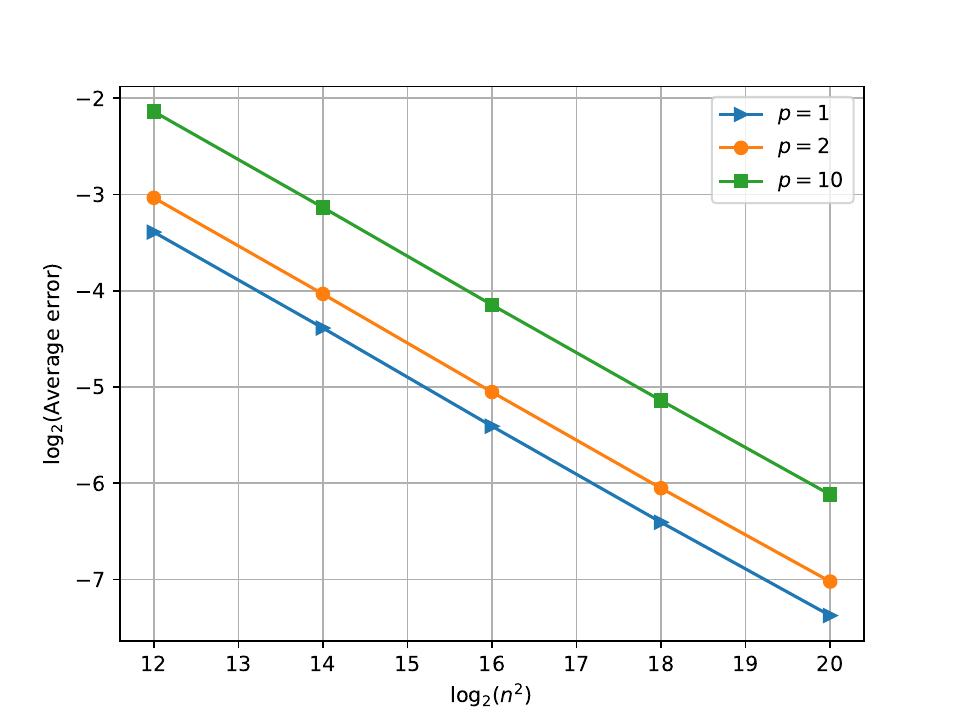}
\includegraphics[scale=.39]{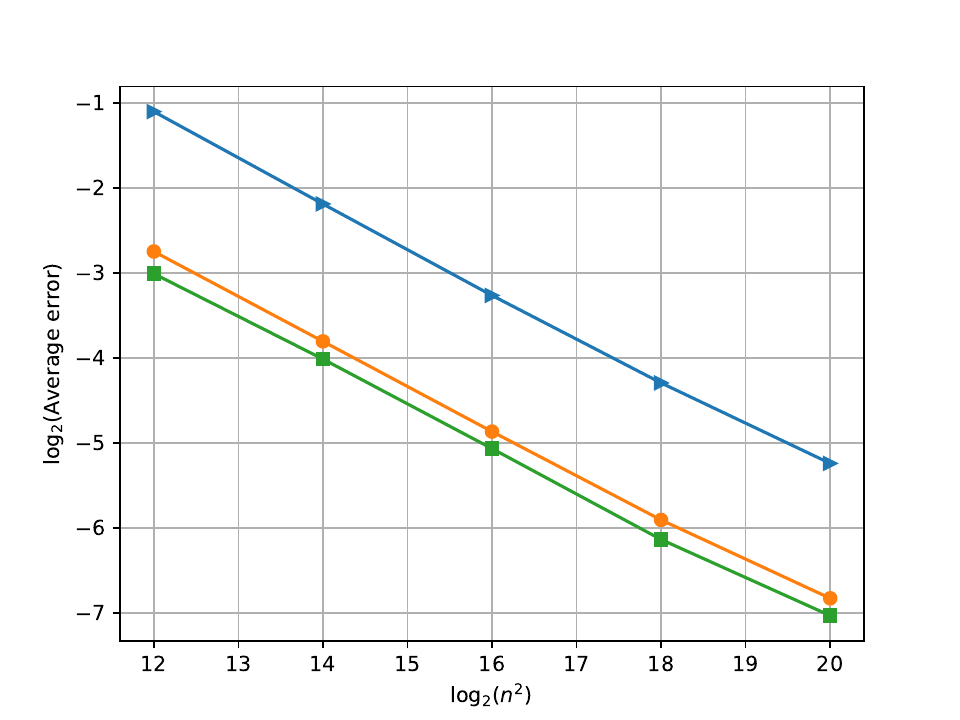}
\caption{Results of the experiment described in Section \ref{section:noise},
showing the average errors as a function of number of samples $n^2$ (in log scale).
Left: sliced Volterra norms of noise alone. Right: average relative errors between the
noisy distances and the true distances.}
\label{fig:signoise}
\end{figure}

\section{Numerical results}
\label{section:numerical}

In this section, we report on numerical
experiments that explore the behavior
of the sliced Cram\'er metrics in 2D under deformations
and noise,
illustrating their behavior on selected examples.
We compare and contrast the behavior of the sliced Cram\'er metrics
with that of the sliced Wasserstein
and Lebesgue distances.
The author's code for the sliced Cram\'er and sliced Wasserstein distances
may be found at
\url{https://github.com/wleeb/SlicedCramer}.

\subsection{Deformations in 2D}
\label{sec:numerical_deformations}

We consider the function $f$
shown in the leftmost panel of Figure \ref{fig:deformations},
which is a convex combination of $12$ isotropic Gaussian functions:
\begin{align}
f(x) = \sum_{i=1}^{3} \sum_{j=1}^{4} w_{ij} e^{-|x - c_{ij}|^2 / \tau},
\end{align}
where $\tau = 1/5000$.
Although $f$ is, strictly speaking, not
compactly supported, it is negligibly small outside of
$[-1/2,1/2] \times [-1/2,1/2]$.
The centers $c_{ij}$ of the Gaussians comprising $f$
are arranged in a
$4$-by-$3$ grid, equispaced in the rectangle
$[-1/12,1/12] \times [-1/12,1/12]$.
The weights $w_{ij}$ assigned to the Gaussians in $f$
may be described as follows:
assigning the numbers $1$ to $4$ to the rows going from top to bottom,
and assigning $1$ to $3$ to the columns going from left to right, the weight assigned
to the Gaussian in position $(i,j)$ is proportional to $\sqrt{i^2 + j^2}$.
These weights are then normalized to sum to $1$,
so that $\|f\|_{L^1} = 1$.

We examine the distances between
$f$ and its deformations.
For each metric $\mathrm{D}$, we compute the distances $\mathrm{D}(f,f_{\Phi_\delta})$,
where $\Phi_\delta$ is a deformation depending on a single parameter $\delta$,
where the displacement of $\Phi_\delta$ grows with $\delta$.
The distances $\mathrm{D}$ are the sliced $p$-Cram\'er, sliced $p$-Wasserstein,
and $p$-Lebesgue, for $p=1,2,10$.
We consider three types of deformations: translations, rotations, and dilations.
Examples of these are displayed in Figure \ref{fig:deformations}.
In all examples, we evaluate the distances for $25$ deformation parameters,
using samples on a $500$-by-$500$ grid.

Figure \ref{fig:translations} shows the distances $\mathrm{D}(f,f_{\Phi_\delta})$
as a function of the translation size $\delta$, where $\Phi_\delta(x,y) = (x+\delta,y)$.
Figure \ref{fig:rotations} shows the distances $\mathrm{D}(f,f_{\Phi_\delta})$
as a function of the rotation angle $\delta$, where
$\Phi_\delta(x,y) = (x \cos(\delta) - y \sin(\delta), x \sin(\delta) + y \cos(\delta))$.
Figure \ref{fig:dilations} shows the distances $\mathrm{D}(f,f_{\Phi_\delta})$
as a function of the dilation parameter $\delta \ge 1$, where
$\Phi_\delta(x) = \delta x$.

From the plots, it is evident that
the sliced $p$-Cram\'er and sliced $p$-Wasserstein distances exhibit similar behavior.
Both metrics also change more smoothly than the Lebesgue distances as the deformation parameter
changes, particularly for translation and rotation.
For dilation, the Lebesgue $1$-distance quickly becomes large and nearly
constant, because the supports of $f$ and its dilations become nearly disjoint
as the dilation size grows. On the other hand, the Lebesgue $p$-distances
for $p > 1$ appear to vary more smoothly with the dilation parameter;
this is because the $p$-norm of
the dilated function grows with the dilation size, and so
the distance in this case is due to the growing size of the single function,
rather than providing any meaningful information about the relationship between
the two functions. By contrast, the sliced $p$-Cram\'er distance
does not grow arbitrarily big as the norm grows; see Remark \ref{rmk:switch}
after the statement of Theorem \ref{thm:main_deformations}.

\subsection{Rotations and projections in 3D}
\label{sec:numerical_rotations3d}

We illustrate
Corollary \ref{cor:main_projections}
by comparing the sliced Cram\'er metrics to the sliced Wasserstein
and Lebesgue distances between 2D projections of rotated volumes.
We consider the function $f$ on $\RR^3$ defined as a convex combination
of $26$ spherical Gaussians centered on a spiral circling the $z$-axis:
\begin{align}
f(x) = \sum_{j=1}^{26} w_j e^{-|x-c_j|^2 / \tau},
\end{align}
where $\tau = 1/200$,
and the centers are
located at the points
\begin{align}
c_j = (0.5 \cos(4 \pi j / 26), 0.5 \sin(4 \pi j / 26), -1/2 + j/25),
\end{align}
with respective heights $w_j$ proportional to $1.5 + \cos(4 \pi j / 26)$,
where $0 \le j \le 25$, normalized to sum to $1$, so that $\|f\|_{L^1} = 1$.
The function is rotated within the $yz$-plane, with rotations between $0$ and $\pi/4$,
and then analytically projected onto the $xy$-plane,
sampled on a $500$-by-$500$ grid over $[-r,r] \times [-r,r]$, with $r = 2.5$.
Example projections are displayed in Figure \ref{fig:projections}.
We evaluate the distances based on these samples
as the rotation angle grows; the distances are displayed in Figure
\ref{fig:distances_projections}.
It is evident that the sliced Cram\'er and sliced Wasserstein distances
grow more slowly with the rotation angle,
and exhibit smoother growth,
than do the Lebesgue distances.

\subsection{Robustness to noise}
\label{section:noise}

We examine the robustness of the sliced Cram\'er distances
to additive noise. We consider the function $f$ from
Section \ref{sec:numerical_deformations},
and add Gaussian noise to the samples of rotations of $f$
on a grid of size $n=512$. We then compute both the sliced Cram\'er distances
and the Euclidean distances
between $f$ and the noisy samples of its rotations.
The noise standard deviations are chosen to be $\sigma=20, 30, 40$.
The distances are averaged over $20$
runs of the experiment. The resulting plots are shown in Figure \ref{fig:defnoise}.
It is evident that the sliced Cram\'er distances are quite robust to noise
at this sample size,
in the sense that though the distances
are inflated, they still
closely track the distances between the noiseless functions (the curve where $\sigma=0$).
For reference, we also plot the Lebesgue distances,
which exhibit much greater inflation due to noise.

For a more quantitative exploration of the effect of noise on
the sliced Cram\'er distances, we show the following experiment.
For increasing values of $n$,
we sample the function $f$ and a translate of $f$ by $1/5$ to
the right, which we denote by $g$,
on a grid of size $n$-by-$n$,
and add Gaussian noise with standard deviation $20$ to the samples of $g$.
For $p=1,2,10$, we evaluate the sliced $p$-Cram\'er distances
between $f$ and the noisy samples of $g$.
For each $n$,
the experiment is repeated $M=1000$ times. We estimate the
true distance $d$ between $f$ and $g$ by evaluating
them on a grid of size $2048$-by-$2048$, and measure the average
absolute relative error between the noisy distances $\what{d}_k$ and $d$:
\begin{align}
\mathrm{err}_{n,p}
= \frac{1}{M} \sum_{k=1}^{M} \frac{\left | d - d_k \right|}{d}.
\end{align}
These value are plotted against $n^2$ in the right panel of
Figure \ref{fig:signoise}, in $\log$ scale.
We also measure the average sliced Volterra $p$-norms of the noise itself,
plotted against $n^2$ in the
left panel of Figure \ref{fig:signoise}.
The average errors decay approximately like $O(1/n)$ as $n$ increases
(that is, the slopes of the plots are close to $-1/2$),
consistent with the error rate established in Theorem \ref{thm:discrete_main2d}.
(Note that we compare the error to the total number of noisy samples $n^2$,
rather than $n$,
to more clearly demonstrate the square root error rate.)

\section{Additional proof details}
\label{section:proof_details}

In this section, we state a number of results
of general interest that are used throughout the proofs in the paper.
Many of these are either well-known, or fairly straightforward
consequences of known results.
We place them here for the reader's easy reference, and
to avoid interrupting the main text.
We provide or sketch proofs of selected results,
in some case for the reader's convenience,
or because we could not find their precise
statements in the literature.

\subsection{Fourier analysis and approximation theory}

\begin{lem}
\label{lem:projection_convolution}
Suppose $f$ and $g$ are functions in $L^1(\RR^d)$,
and $u$ is a unit vector in $\RR^d$.
Then
$\calP_u (f \ast g) = (\calP_u f) \ast (\calP_u g)$.
\end{lem}

\begin{proof}
By the Fourier slice theorem,
\begin{align}
(\calP_u (f \ast g))^{\what{}}(\xi)
&= \what{(f \ast g)}(\xi u)
\nonumber \\
&= \what{f}(\xi u) \what{g}(\xi u)
\nonumber \\
&= \what{(\calP_u f)}(\xi) \what{(\calP_u g)}(\xi)
\nonumber \\
&= [(\calP_u f) \ast (\calP_u g)]^{\what{}}(\xi),
\end{align}
and so, taking inverse Fourier transforms,
$\calP_u (f \ast g) = (\calP_u f) \ast (\calP_u g)$.

\end{proof}

The following is a basic lemma
on the decay and approximation of the Fourier transform
for $C^r$ functions:

\begin{lem}
\label{lem:fourier_decay_approx}
Suppose $f : \RR^d \to \RR$ is $C^r$, where $r > d$, and is supported in $[0,L]^d$.
Then there is a constant $C = C(f,L,r) > 0$
such that $|\what{f}(\nu)| \le C |\nu|^{-r}$,
and for any $\xi \in \mathbb{R}^d$ with $\|\xi\|_\infty \le n/(2L)$ and any $n \ge 2$,
\begin{align}
\left|\frac{1}{n^d}\sum_{k \in \{0,\dots,n-1\}^d}
                f(Lk/n) e^{-2 \pi \mi \langle Lk , \xi\rangle / n}
    -  \what{f}(\xi)\right| \le \frac{C}{n^r}.
\end{align}
\end{lem}

We will also make use of a special case of the periodic version:
\begin{lem}
\label{lem:fourier_series_approx}

Suppose $f : \RR \to \RR$ is $L$-periodic, continuous,
and has Fourier series
\begin{align}
f(x) = \frac{1}{L}\sum_{k=-\infty}^{\infty} c_k e^{2 \pi \mi k x / L}
\end{align}
satisfying $|c_k| \le M / |k|^r$ for $r > 1$ and a constant $M > 0$.
Then there is a value $C = C(r) > 0$ such that for all $n \ge 2$,
\begin{align}
\left|\frac{L}{n}\sum_{j = 0}^{n-1} f(jL/n)
    -  \int_{0}^{L} f(x) dx\right| \le C\frac{M}{n^r},
\end{align}
\end{lem}

\begin{proof}
We have:
\begin{align}
\frac{L}{n} \sum_{j=0}^{n-1} f(jL/n)
&= \frac{1}{n} \sum_{j=0}^{n-1} \sum_{k \in \mathbb{Z}} c_k e^{2 \pi \mi k j/n}
\nonumber \\
&= \int_{0}^{L} f(x) dx
    + \sum_{k \ne 0} c_k \left( \frac{1}{n} \sum_{j=0}^{n-1} e^{2 \pi \mi k j/n} \right)
\nonumber \\
&= \int_{0}^{L} f(x) dx
    + \sum_{k = 0\, \mathrm{mod} \, n, \, k \ne 0} c_k,
\end{align}
and consequently
\begin{align}
\left| \frac{L}{n}\sum_{j=0}^{n-1} f(t_j) - \int_{0}^{L}f(x)dx\right|
\le \sum_{\ell \ne 0} |c_{\ell n}|
\le M\sum_{\ell \ne 0} \frac{1}{|\ell n|^r}
= \frac{M}{n^r} \left(2 \sum_{\ell =1}^{\infty} \ell^{-r}\right).
\end{align}

\end{proof}

Next, we state several results that bound the error
on the approximation
the $L^p$ norm of a periodic function
using the trapezoidal rule.

\begin{lem}
\label{lem:fourier_pth_power}
Suppose $G: \RR \to \RR$ is $L$-periodic and $C^2$.
Then for any $1 \le p < \infty$ and integer $k \ne 0$,
\begin{align}
\left| \int_{0}^{L} |G(t)|^p \,e^{-2 \pi \mi k t / L} \, dt \right|
\le C \frac{L^2 p }{k^2}\|G''\|_{L^\infty} \int_{0}^{L} |G(t)|^{p-1} dt,
\end{align}
where $C>0$ is a universal constant.
\end{lem}

\begin{proof}
First, if $\sign(G(t))$ is constant,
then the function $|G(t)|^p$ is $C^2$,
and the proof we give below can be simplified to give the desired
bound, using integration-by-parts twice; we will skip the details
of this case. We will therefore assume
that $G$ has at least one root,
but is also not constantly zero.
In this case, because $G$ is $L$-periodic,
we can assume without loss of generality
that $G(0) = G(L) = 0$.

Let $A = \{t \in (0,L): G(t) \ne 0\}$.
This set is open and non-empty, hence can be written
as the disjoint union of at most countably many disjoint intervals $(a_j,b_j)$.
Note that $\sign(G(t))$ is constant for $t \in (a_j,b_j)$.
Also, note that $G(a_j) = G(b_j) = 0$ for all $j$; indeed,
if $a_j = 0$, then $G(a_j) = 0$ by assumption; and
if $a_j > 0$ but $G(a_j) \ne 0$,
then there would be some $j'$ with $a_{j} \in (a_{j'},b_{j'}) \subset A$;
but this contradicts that the intervals $(a_j,b_j)$
and $(a_{j'},b_{j'})$ are disjoint.

Let $H_p(x) = |G(x)|^p$.
Suppose first that $p > 1$.
Then on each interval $(a_j,b_j)$,
\begin{align}
H_p'(x) = p\ \sign(G(x)) |G(x)|^{p-1} G'(x),
\end{align}
and
\begin{align}
H_p''(x) = p(p-1)|G(x)|^{p-2} G'(x)^2 + p \ \sign(G(x)) |G(x)|^{p-1} G''(x).
\end{align}
Since $H_p(0) = H_p(L)$, we have, for all $k \ne 0$,
\begin{align}
\label{eq:5040201}
\int_{0}^{L} H_p(x) e^{-2 \pi \mi k x / L} \, dx
&= \frac{L}{2 \pi \mi k} \int_{0}^{L} H_p'(x) e^{-2 \pi \mi k x/L} \, dx.
\end{align}

Let $\chi(x,j)$ be $1$ if $x \in (a_j,b_j)$, and $0$ otherwise.
We then write
\begin{align}
\int_{0}^{L} H_p'(x) e^{-2 \pi \mi k x/L} \, dx
= \int_{0}^{L} \sum_{j} \chi(x,j) H_p'(x) e^{-2 \pi \mi k x/L} \, dx,
\end{align}
and since
\begin{align}
\le \int_{0}^{L} \sum_{j} \chi(x,j) |H_p'(x)|  \, dx
\le L \|H_p'\|_{L^\infty} < \infty,
\end{align}
we may use Fubini's Theorem to switch the order of summation and integration
to obtain
\begin{align}
\int_{0}^{L} H_p'(x) e^{-2 \pi \mi k x/L} \, dx
&= \int_{0}^{L} \sum_{j} \chi(x,j) H_p'(x) e^{-2 \pi \mi k x/L} \, dx
\nonumber \\
&= \sum_{j}  \int_{0}^{L} \chi(x,j) H_p'(x) e^{-2 \pi \mi k x/L} \, dx
\nonumber \\
&= \sum_{j}  \int_{a_j}^{b_j} H_p'(x) e^{-2 \pi \mi k x/L} \, dx.
\end{align}

We now bound this sum.
Using integration-by-parts again gives
\begin{align}
\int_{a_j}^{b_j} H_p'(x) e^{-2 \pi \mi k x/L} \, dx
&= \frac{H_p'(a_{j}^+) e^{-2 \pi \mi k a_j /L} - H_p'(b_{j}^-) e^{-2 \pi \mi k b_{j} /L}}
        {2 \pi \mi k / L}
    + \frac{L}{2 \pi \mi k}\int_{a_j}^{b_{j}} H_p''(x) e^{-2 \pi \mi k x/L} \, dx
\nonumber \\
&= \frac{L}{2 \pi \mi k}\int_{a_j}^{b_{j}} H_p''(x) e^{-2 \pi \mi k x/L} \, dx,
\end{align}
where the end terms vanish because $G(a_j) = G(b_{j}) = 0$
and $H_p'(x) = p \,\sign(G(x)) |G(x)|^{p-1} G'(x)$.
Now, suppose without loss of generality that $G(x) > 0$ on $(a_j,b_{j})$
(an analogous argument will work if $G(x) < 0$).
Then
\begin{align}
\left|\int_{a_j}^{b_{j}} H_p''(x) e^{-2 \pi \mi k x/L} \, dx \right|
\le p\int_{a_j}^{b_{j}} (p-1) G(x)^{p-2} G'(x)^2 dx
    + p\int_{a_j}^{b_{j}} G(x)^{p-1} |G''(x)| dx.
\end{align}

For the first integral, we
let $D(x) = G(x)^{p-1}$, so that $D'(x) = (p-1) G(x)^{p-2} G'(x)$;
then using integration-by-parts,
\begin{align}
p \int_{a_j}^{b_{j}} (p-1) G(x)^{p-2} G'(x)^2 dx
&= p\int_{a_j}^{b_{j}} D'(x) G'(x) dx
\nonumber \\
&= p (D(b_{j}) G'(b_{j}) - D(a_{j}) G'(a_{j}))
    - p\int_{a_j}^{b_{j}} D(x) G''(x) dx
\nonumber \\
&= - p\int_{a_j}^{b_{j}} G(x)^{p-1} G''(x) dx,
\end{align}
and therefore,
\begin{align}
\left|\int_{a_j}^{b_{j}} H_p''(x) e^{-2 \pi \mi k x/L} \, dx \right|
\le 2 p\int_{a_j}^{b_{j}} |G(x)|^{p-1} |G''(x)| dx
\le 2 p \|G''\|_{L^\infty}\int_{a_j}^{b_{j}} |G(x)|^{p-1} dx.
\end{align}

Summing over $j$ then gives
\begin{align}
\left| \int_{0}^{L} H_p'(x) e^{-2 \pi \mi k x/L} \, dx \right|
&= \left| \sum_{j}  \int_{a_j}^{b_{j}} H_p'(x) e^{-2 \pi \mi k x/L} \, dx \right|
\nonumber \\
&= \left| \sum_{j}  \frac{L}{2 \pi \mi k}
    \int_{a_j}^{b_{j}} H_p''(x) e^{-2 \pi \mi k x/L} \, dx \right|
\nonumber \\
&\le \frac{L}{2 \pi |k|} \sum_{j}
    \left|\int_{a_j}^{b_{j}} H_p''(x) e^{-2 \pi \mi k x/L} \, dx \right|
\nonumber \\
&\le \frac{L}{2 \pi |k|} \sum_{j} 2 p \|G''\|_{L^\infty}\int_{a_j}^{b_{j}} |G(x)|^{p-1} dx
\nonumber \\
&\le \frac{Lp}{\pi |k|} \|G''\|_{L^\infty} \int_{0}^{L} |G(x)|^{p-1}dx.
\end{align}

Finally, using \eqref{eq:5040201}, we have
\begin{align}
\left|\int_{0}^{L} |G(x)|^p \, e^{-2 \pi \mi k x / L} \, dx \right|
&= \left|\int_{0}^{L} H_p(x) e^{-2 \pi \mi k x / L} \, dx \right|
\nonumber \\
&= \frac{L}{2 \pi |k|} \left| \int_{0}^{L} H_p'(x) e^{-2 \pi \mi k x/L} \, dx \right|
\nonumber \\
&\le \frac{L^2 p}{2 \pi^2 k^2} \|G''\|_{L^\infty} \int_{0}^{L} |G(x)|^{p-1}dx.
\end{align}

Taking the limit as $p \to 1^+$
also shows the corresponding bound for $p=1$ as well:
\begin{align}
\left|\int_{0}^{L} |G(x)| e^{-2 \pi \mi k x / L} \, dx \right|
\le \frac{L^3}{2\pi^2 k^2} \|G''\|_{L^\infty},
\end{align}
which completes the proof.

\end{proof}

\begin{cor}
\label{cor:trapezoidal_pnorm}

Suppose $G: \RR \to \RR$ is $L$-periodic and $C^2$.
Let
\begin{align}
t_j = \frac{j}{n}L, \quad 0 \le j \le n.
\end{align}
Then for any $1 \le p < \infty$,
and all $n \ge 2$,
\begin{align}
\left|\frac{L}{n}\sum_{j=0}^{n-1} |G(t_j)|^p
    - \int_{0}^{L} |G(t)|^p dt  \right|
\le C \frac{L^{2} p}{n^2}  \|G''\|_{L^\infty} \int_{0}^{L} |G(t)|^{p-1} dt,
\end{align}
where $C$ is a universal constant.

\end{cor}

\begin{proof}
This follows from Lemma \ref{lem:fourier_series_approx}
and Lemma \ref{lem:fourier_pth_power}.

\end{proof}

\begin{cor}
\label{cor:trapezoidal_pnorm2}
Suppose $G: \RR \to \RR$ is $L$-periodic and $C^2$.
Let
\begin{align}
t_j = \frac{j}{n}L, \quad 0 \le j \le n.
\end{align}
Then for any $1 \le p < \infty$,
there is $N_p$ such that for all $n \ge N_p$,
\begin{align}
\left| \left(\frac{L}{n}\sum_{k=0}^{n-1} |G(t_k)|^p \right)^{1/p}
    - \left(\int_{0}^{L} |G(t)|^p dt \right)^{1/p} \right|
\le C \frac{L^{2+1/p} \|G''\|_{L^\infty}}{n^2},
\end{align}
with the obvious modification
when $p = \infty$, where $C$ is a universal constant.
In particular, we have the bound
\begin{align}
\left| \left(\frac{L}{n}\sum_{k=0}^{n-1} |G(t_k)|^p \right)^{1/p}
    - \left(\int_{0}^{L} |G(t)|^p dt \right)^{1/p} \right|
\le \frac{C}{n^2}
\end{align}
for all $n \ge 2$, with constant $C = C(p,G,L) > 0$.
\end{cor}

\begin{proof}
Suppose $G$ is not constantly $0$ (otherwise the result is trivial).
For brevity, let $I = \|G\|_{L^p}^p/2$.
Let
\begin{align}
T_{n,p} = \frac{L}{n} \sum_{k=0}^{n-1} |G(t_k)|^p.
\end{align}
Then for all $n$ sufficiently large,
$T_{n,p} \ge I$.
The function $y \mapsto y^{1/p}$ has derivative $y^{1/p-1} / p$,
which has maximum value $I^{1/p-1}/p$ when $y \ge I$.
Hence, by the mean value theorem, for all $n$ sufficiently large,
\begin{align}
\left| T_{n,p}^{1/p} - \|G\|_{L^p} \right|
&\le \left| T_{n,p} - \|G\|_{L^p}^p \right| (I^{1/p-1}/p)
\nonumber \\
&\le C L^{2+1/p} p\|G\|_{L^p}^{p-1} \|G''\|_{L^\infty} \frac{1}{n^2}
    \frac{(\|G\|_{L^p}^p/2)^{1/p-1} }{p}
\nonumber \\
&\le C \frac{L^{2+1/p} \|G''\|_{L^\infty}}{n^2}.
\end{align}
This is the desired bound for $1 \le p < \infty$.

When $p=\infty$, suppose again that $G$ is not constantly zero, since
otherwise the result is trivial.
Let $x^*$ satisfy $\|G\|_{L^\infty} = |G(x^*)|$.
If $x^* = 0$, then
then, since $t_0 = 0$,
$\max_{0 \le k \le n-1}|G(t_k)| = |G(0)| = \|G\|_{L^\infty}$.
(If $x^* = L$, then since $G$ is $L$-periodic,
we can also take $x^*=0$.)
Assume, then, that $x^*$ lies in the interior of $(0,L)$.
Then $G'(x^*) = 0$, and so
a second-order Taylor expansion gives
\begin{align}
|G(x) - G(x^*)| \le C\|G''\|_{L^\infty} |x - x^*|^2,
\end{align}
where $C > 0$ is universal.
Consequently, if $t_{k^*}$ is a grid point
within $L/n$ of $x^*$, we have
\begin{align}
|G(t_{k^*}) - G(x^*)| \le C \frac{L^2}{n^2} \|G''\|_{L^\infty},
\end{align}
and therefore, since $\max_{0 \le j \le n-1}|G(t_j)| \ge |G(t_{k^*})|$,
\begin{align}
\left| \max_{0 \le j \le n-1}|G(t_j)| - \|G\|_{L^\infty} \right|
&= |G(x^*)| - \max_{0 \le j \le n-1}|G(t_j)|
\nonumber \\
&\le |G(x^*)| - |G(t_{k^*})|
\nonumber \\
&\le |G(x^*) - G(t_{k^*})|
\nonumber \\
&\le C\frac{L^2 \|G''\|_{L^\infty}}{n^2},
\end{align}
which is the desired result.

\end{proof}

The following result is standard and may be found (in more general form)
in Chapter V of \cite{zygmund2003trigonometric}:

\begin{lem}
\label{lem:trig_sine}
There is a universal constant $C > 0$ such that
for any positive integer $m$ and real number $A$,
\begin{align}
\left|\sum_{k=1}^{m} \frac{\sin(k A)}{k} \right| \le C.
\end{align}
\end{lem}

\subsection{Probability and concentration}

We start with the following well-known corollary of the Borel-Cantelli Lemma
(see e.g.\ Chapter 2, Section 10 of \cite{shiryaev2015probability}):

\begin{lem}
\label{lem:borel_cantelli}
Let $R_1,R_2,\dots$ be a sequence of random numbers.
Suppose that for all $\epsilon > 0$,
\begin{align}
\sum_{n=1}^{\infty} \Prob\{R_n > \epsilon\} < \infty.
\end{align}
Then $R_n \to 0$ almost surely.
\end{lem}

\subsubsection{Properties of sub-Gaussian random variables}

\begin{prop}
\label{prop:subgaussian_pnorm}
If $R$ is a sub-Gaussian random variable,
then for any $p \ge 1$,
\begin{align}
\left(\EE[|R|^p] \right)^{1/p} \le C \sqrt{p} \|R\|_{\psi_2},
\end{align}
where $C>0$ is a universal constant.
\end{prop}

\begin{prop}
\label{prop:subgaussian_mgf}
If $R$ is a sub-Gaussian random variable
with mean zero,
then for any $s \in \RR$,
\begin{align}
\EE[e^{s R}] \le e^{C s^2 \|R\|_{\psi_2}^2},
\end{align}
where $C > 0$ is a universal constant.
\end{prop}

The proofs of Proposition \ref{prop:subgaussian_pnorm}
and Proposition \ref{prop:subgaussian_mgf}
may be found in \cite{vershynin2025high}.
Next, we prove a standard bound on the expected value
of the maximum of sub-Gaussians:

\begin{lem}
\label{lem:expected_maximum_subgaussians}
Let $R[0],R[1],\dots,R[n-1]$ be mean zero sub-Gaussian random variables,
with sub-Gaussian norms $\|R[j]\|_{\psi_2} \le \tau$, $0 \le j \le n-1$,
and let $R = (R[0],\dots,R[n-1])$.
Then
\begin{align}
\EE[\|R\|_\infty] \le C \tau \sqrt{\log(n)},
\end{align}
where $C > 0$ is a universal constant.
\end{lem}

\begin{proof}
First, extend $R$ by defining $R[j] = -R[j-n]$
for  $n \le j \le 2n-1$. Then
\begin{align}
\|R\|_{\infty}
&= \max\{R[0],\dots,R[n-1],-R[0],\dots,-R[n-1]\}
\nonumber \\
&= \max\{R[0],\dots,R[n-1],R[n],\dots,R[2n-1]\}.
\end{align}
Note that for all $j$, since $R[j]$ is a mean-zero
sub-Gaussian with $\|R[j]\|_{\psi_2} \le \tau$,
by Proposition \ref{prop:subgaussian_mgf}
the moment generating functions satisfy the bound
\begin{align}
\EE[e^{s R[j]}] \le e^{C \tau^2 s^2}.
\end{align}
Using this bound,
for any $s > 0$, we have
\begin{align}
e^{\EE[s \|R\|_\infty]}
&\le \EE\left[ e^{s \|R\|_\infty}\right]
\nonumber \\
&= \EE\left[ \max_{0 \le j \le 2n-1}  e^{s R[j]}\right]
\nonumber \\
&\le \EE\left[ \sum_{j=0}^{2n-1}  e^{s R[j]}\right]
\nonumber \\
&= \sum_{j=0}^{2n-1}   \EE\left[ e^{s R[j]}\right]
\nonumber \\
&\le 2n  \max_{0 \le j \le 2n-1} \EE\left[ e^{s R[j]}\right]
\nonumber \\
&\le 2 n  e^{C \tau^2 s^2}
\end{align}
and so, taking the log of each side gives
\begin{align}
\EE\left[ \|R\|_\infty \right]
\le C \left(\frac{\log(n)}{s} + s \tau^2 \right),
\end{align}
and taking $s =  \sqrt{\log(n)} / \tau$ gives the bound
\begin{align}
\EE\left[ \|R\|_\infty \right] \le C \tau \sqrt{\log(n)},
\end{align}
as desired.

\end{proof}

The next result bounds the sub-Gaussian norm
of the maximum of sub-Gaussians; see
Proposition 2.7.6 in \cite{vershynin2025high}.

\begin{prop}
\label{prop:norm_maximum_subgaussians}
If $R_1,\dots,R_m$
are sub-Gaussian random variables, then
\begin{align}
\left\|\max_{1 \le i \le m} R_i \right\|_{\psi_2}
\le C \sqrt{\log(m)} \max_{1 \le i \le m} \|R_i\|_{\psi_2},
\end{align}
where $C > 0$ is a universal constant.

\end{prop}

The next result appears as 
Proposition 2.7.1 in \cite{vershynin2025high}.

\begin{prop}
\label{prop:subgaussian_pythagorean}
There is a universal constant $C>0$ such that if
$R[0],\dots,R[n-1]$ are independent sub-Gaussian random variables,
\begin{align}
\left\| \sum_{j=0}^{n-1} R[j]\right\|_{\psi_2}^2
\le C \sum_{j=0}^{n-1}  \left\| R[j]\right\|_{\psi_2}^2.
\end{align}
\end{prop}

\subsubsection{Sub-Weibull random variables}
\label{sec:subweibull}

For $\alpha > 0$, define the function $\psi_\alpha(x) = \exp(x^\alpha) - 1$
on $[0,\infty)$. Note that $\psi_\alpha$ is increasing, and $\psi_\alpha(0) = 0$.
Define
\begin{align}
\|X\|_{\psi_\alpha} = \inf\{t > 0 \, : \, \EE[\psi_\alpha(|X|/t)] \le 1\}.
\end{align}
Random variables $X$ with $\|X\|_{\psi_\alpha} < \infty$
are called \emph{sub-Weibull($\alpha$)};
see \cite{kuchibhotla2022moving}.
Note that $\|X\|_{\psi_\alpha}$ is a norm when $\alpha \ge 1$, since $\psi_\alpha$
is convex. When $0 < \alpha < 1$, $\|X\|_{\psi_\alpha}$
does not satisfy $\|X+Y\|_{\psi_\alpha} \le \|X\|_{\psi_\alpha} + \|Y\|_{\psi_\alpha}$,
but the following still holds:
\begin{prop}
\label{prop:subweibull_pseudonorm}
Let $0 < \alpha < 1$. Then there is a universal constant $C > 0$
such that for all sub-Weibull($\alpha$) random variables
$X_1,\dots,X_n$,
\begin{align}
\|X_1 + \dots + X_n\|_{\psi_\alpha}
\le C^{1/\alpha} \alpha^{-1/\alpha-1}(\|X_1\|_{\psi_\alpha} + \dots + \|X_n\|_{\psi_\alpha}),
\end{align}
\end{prop}
Proposition \ref{prop:subweibull_pseudonorm}
follows by adapting the proof of Proposition 2.6.1
in \cite{vershynin2025high}
and using the fact that $\sup_{r \ge 1} r^{-1/\alpha} (\EE|X|^r)^{1/r}$
is a proper norm.

The following result is immediate from the definition of
sub-Weibull:

\begin{lem}
\label{lem:weibull_subgaussian}
Let $p \ge 1$.
A random variable $R$ is sub-Gaussian
if and only if $X = |R|^p$ is sub-Weibull($2/p$),
in which case $\|X\|_{\psi_{2/p}} = \|R\|_{\psi_2}^{p}$.
\end{lem}

\section{Conclusion}
\label{section:conclusion}

This paper has proven a number of properties of sliced Cram\'er distances.
We have characterized their growth under deformations
of the inputs, proven stability to convolutions,
provided bounds on the discretization errors, and shown robustness
to heteroscedastic sub-Gaussian noise.

Looking forward, there are several questions
that follow from the present work.
First, at the core of the sliced Cram\'er distances are, of course,
the 1D Cram\'er distances. These are quite simple objects,
defined by applying a smoothing filter to the input functions and
then evaluating the ordinary Lebesgue distance.
It is therefore of interest to consider metrics
based on other families of smoothing filters,
and to characterize their behavior under deformations
and robustness to noise.
Such metrics are abundant in applications
in both Euclidean and non-Euclidean settings,
and, unlike the sliced Cram\'er distances,
often involve filtering/smoothing the inputs at multiple scales
\cite{mishne2016hierarchical,
mishne2017datadriven,
leeb2016holder, leeb2018mixed, jacobs2008approximate, mishne2019comanifold}.
In future work, we plan to
study such metrics' growth under deformations
and their stability to noise.

Second, while this paper has shown
theoretical properties of the sliced Cram\'er
distances, it is of interest to explore their behavior
in scientific applications,
such as clustering images or volumes.
For instance,
Corollary \ref{cor:main_projections}
suggests that they may be well-suited for analyzing data from cryo-electron
microscopy (cryo-EM), in which one observes two-variable projections
of a three-variable volume (a molecule),
at unknown viewing directions, from which the volume
is to be determined \cite{singer2020computational, bendory2020single,
doerr2016single}.
It would be of interest to see whether sliced Cram\'er metrics
perform well relative to other metrics that have been proposed
for image clustering and parameterizing volumes in cryo-EM,
such as Wasserstein-type metrics
\cite{singer2024alignment, rao2020wasserstein, zelesko2020earthmover, shi2025fast}.
One such application
is heterogeneity analysis
\cite{zhong2021cryodrgn, frank2016continuous, liao2015efficient,
aizenbud2019maxcut, lederman2020hyper, scheres2016processing, toader2023methods,
lederman2020representation},
for which the use of Wasserstein distances has been proposed \cite{rao2020wasserstein}.

Finally, in certain applications one seeks metrics
that are not only robust
to all deformations, but invariant to a specific class of deformations,
such as rotations and/or translations
\cite{shi2025fast, rao2020wasserstein, zhao2014rotationally}.
In principle, any metric can be made invariant to a specified set of
deformations by simply minimizing the distance over the class of deformations.
In some cases of practical interest, such as
rigid alignment, this minimization can be done
with only small extra computational cost \cite{shi2025fast, rangan2020factorization}.
In the context of the present work, this raises  interesting questions.
First, it is known that in the presence of noise, alignment accuracy
deterioriates \cite{robinson2004fundamental, aguerrebere2016fundamental};
studying how invariant distances behave under noise, or more generally
in any setting where alignment cannot be done to high precision, is therefore
of interest.
Second, it is natural to ask how introducing invariance
to one class of deformations, such as rotations, impacts robustness
to other deformations. Questions along these lines will be
pursued in future work.

\subsection*{Acknowledgements}

I thank Amit Singer for feedback on an earlier version of
this work, and Liane Xu for pointing out useful references
and the connection between sliced $2$-Cram\'er
metrics and energy distances.
I also express gratitude for the detailed reviewer comments
that substantially improved the exposition of the manuscript.
I acknowledge support from NSF CAREER award DMS-2238821
and the McKnight Foundation.

\bibliographystyle{plain}
\bibliography{refs_cramer}

\end{document}